\documentclass[a4paper,11pt,english,reqno]{amsart}

\usepackage[english]{babel}
\usepackage{amssymb}
\usepackage{color}
\usepackage{appendix}
\usepackage{stmaryrd}
\usepackage{bm, upgreek}
\usepackage{soul}
\usepackage{mathrsfs}
\usepackage{hyperref}
\usepackage{mathtools}
\usepackage{dsfont}
\usepackage{tikz}
\usetikzlibrary{decorations.pathreplacing}
\numberwithin{equation}{section}

\usepackage[letterpaper,top=2cm,bottom=2cm,left=2cm,right=2cm,marginparwidth=1.75cm]{geometry}
\renewcommand{\leq}{\leqslant}
\renewcommand{\geq}{\geqslant}
\newcommand{\cbar}{\overline{C}}
\newcommand{\cstar}{\overline{C}_{*}}

\newcommand{\N}{\mathbb{N}}
\newcommand{\wn}{\widetilde{N}}
\newcommand{\R}{\mathbb{R}}

\newcommand{\W}{\mathcal{W}}
\newcommand{\E}{\mathbb{E}}

\newcommand{\UTH}{{\sf UT}_{{H}}(\delta)}
\newcommand{\nUTH}{\widetilde{{\sf UT}}_H(\delta)}

\renewcommand{\P}{\mathbb{P}}

\newcommand{\G}{\mathbb{G}}
\newcommand{\Gd}{\mathbb{G}^{d}_{n}}
\newcommand{\Xd}{\mathcal{X}^{d}_{n}}

\newcommand{\tu}[1]{\textup{#1}}

\renewcommand{\d}{\widehat{d}}

\newcommand{\core}{\tu{core}}

\newcommand{\set}[1]{[\![#1]\!]}

\usepackage{amsthm}
\newtheorem{theorem}{Theorem}[section]
\newtheorem{proposition}[theorem]{Proposition}
\newtheorem{definition}[theorem]{Definition}
\newtheorem{corollary}[theorem]{Corollary}

\newtheorem{lemma}[theorem]{Lemma}

\newtheorem{assumption}[theorem]{Assumption}

\newcommand{\low}{\textup{low}}
\newcommand{\high}{\textup{high}}
\newcommand{\bad}{\textup{bad}}
\renewcommand{\sc}{\textup{sc}}
\newcommand{\tup}{i_{\star},\bm{e}, \bm{e}_{\#}, \bm{v}, \bm{v}_{\#}}
\newcommand{\x}{\widehat{x}}

\theoremstyle{definition}
\newtheorem{remark}{Remark}[section]

\usepackage{amsmath}
\usepackage{graphicx}

\definecolor{darkgreen}{RGB}{0,100,0}

\definecolor{purple}{rgb}{0.9,0,0.8}

\definecolor{gray}{rgb}{0.7,0.7,0.7}

\definecolor{darkblue}{RGB}{0,87,184}
\definecolor{darkred}{RGB}{178,34,34}
\definecolor{darkpurple}{RGB}{128,0,128}
\definecolor{deepteal}{RGB}{0,112,112}
\newcommand{\dteal}[1]{#1}

\newcommand{\dred}[1]{#1}

\begin{document} 
\title[Large deviations of regular subgraph counts in random regular graphs]{Upper tail large deviations of regular subgraph counts in sparse random regular graphs}

\author{Anirban Basak}
\address{Anirban Basak, International Centre for Theoretical Sciences, Tata Institute of Fundamental Research, Bangalore, India}
\email{anirban.basak@icts.res.in}
\author{Shaibal Karmakar}
\address{Shaibal Karmakar, International Centre for Theoretical Sciences, Tata Institute of Fundamental Research, Bangalore, India}
\email{shaibal.karmakar@icts.res.in}

\begin{abstract}
\dteal{For a $\Delta$-regular connected graph $H$ we obtain the upper tail large deviations speed and rate function for the number of copies of $H$ in $\mathbb{G}_n^d$, a random $d$-regular graph on $n$ vertices, in the regime $np^{\Delta/2} \gg (\log n)^{1/(v_H-2)}$ and $p \ll 1$, where $p:=d/n$ and $v_H$ is the number of vertices, and also identify the typical structure of $\mathbb{G}_n^d$ conditioned on this upper tail event. For $\Delta \geq 3$, the speed and the rate function are same with those in the Erd\H{o}s-R\'enyi setting. In contrast, for $\Delta=2$ (i.e.~in case of cycles), due to the degree constraint in $\mathbb{G}_n^d$, although speed remains the same the rate function turns out to be different. Moreover, conditioned on the upper tail event of cycle counts in $\mathbb{G}_n^d$, we show that the random graph contains a disjoint collection of clique and biclique like structures with probability approaching to one, a stark contrast compared to the Erd\H{o}s-R\'enyi setting.}

\dteal{In the complementary regime, i.e.~for $1 \ll n^{v_H}p^{e_H} \ll (\log n)^{\alpha^*_H/(\alpha^*_H-1)}$, where $e_H$ is the number of edges in $H$, $\alpha^*_H$ is the fractional independence number of $H$, and $H$ is a strictly balanced graph (and thus including the regular case), we show that the upper tail large deviations speed and rate function for the number of copies of $H$ match with those in the Erd\H{o}s-R\'enyi setting. 
In this regime, we further show that almost all of the unlabeled copies of $H$ conditioned on the upper tail event are edge disjoint.
}		
\end{abstract}

\maketitle
\section{Introduction and main results}

\subsection{Setup and main results} In the past fifteen years there have been a significant interest in studying large deviations in random graphs. One of the early influential works in this area dates back to \cite{varadhan2011} where, using the graph limit theory and Szemer\'edi's regularity lemma, Chatterjee and Varadhan derived large deviations of any statistic that is continuous under the cut metric for {\em dense} (i.e.~the edge probability does not change with the size of the graph) Erd\H{o}s-R\'enyi graphs. Since then there have been an enormous amount of works in studying large deviations in  random graphs (mostly Erd\H{o}s-R\'enyi), both dense and sparse. Beyond the setting of Erd\H{o}s-R\'enyi graphs, where the edges are jointly independent, results have been quite limited. 

One natural model of random graph that does not have the joint independence of its edges and is widely studied in the literature is the random $d$-regular graph model on $n$ vertices, to be denoted hereafter by $\Gd$ (without loss of generality we assume that $nd \in 2 \N$ which ensures that such graphs exist). The work of Dhara and Sen \cite{Dhara_Sen_2022} obtained large deviations of observables in $\Gd$, with $d \asymp n$ (see  Section \ref{sec:notation} for the notational conventions adopted in this article), that are continuous in the cut metric. On the other hand, for the sparse case, i.e.~$d \ll n$, for certain sparsity ranges, Bhattacharya and Dembo \cite{BD} and Gunby \cite{Gunby2020UpperTO} obtained upper tail large deviations of the number of homomorphisms of subgraphs in $\Gd$ (see Section \ref{sec:background} for the sparsity range and the choices of the subgraphs in these works). 
The goal of this present article is to study the upper tail large deviations of regular subgraph counts in sparse $\Gd$ for the entire sparsity ranging from $d \ll n$ to all the way up to the appearance threshold, establishing analogs of results of \cite{ABRB, HMS} that were obtained in the Erd\H{o}s-R\'enyi setting, and identify the typical structures of $\Gd$ conditioned on these upper tail events.

\dteal{Throughout this article, for ease in writing and to compare our results with those in the Erd\H{o}s-R\'enyi setting, we set $p:=d/n$.} It will be convenient to have the following definition: 
Fix $\delta >0$ and set 
\[
\UTH:= \left\{N(H, \Gd)\geq (1+\delta)n^{v_{H}}p^{e_{H}} \right\},
\]
where $v_H$ and $e_H$ denote the number of vertices and edges of $H$, respectively, and for graphs $H$ and $G$ we write $N(H,G)$ to denote the number of labeled copies of $H$ in $G$. 
The following is the first main result of this article. 

\begin{theorem}[Localized regime]\label{thm:loc-regime}
Let $H$ be a $\Delta$-regular {connected} graph and $p \in (0,1)$ be such that $np^{\Delta/2}\gg (\log n)^{1/(v_H -2)}$  and $p\ll 1$. For any $\delta >0$,
\begin{equation*}
    \lim_{n\rightarrow \infty}-\frac{\log\P\left(\dteal{\UTH}\right)} {n^{2}p^{\Delta} \log(1/p)}=\dteal{\Psi_H(\delta)},
\end{equation*}
\dteal{where 	
\begin{equation}
\label{eq:rate function defn}
\Psi_{H}(\delta)\coloneqq  
\frac{1}{2}\begin{cases}
    \lfloor \delta\rfloor+\{\delta\}^{2/v_{H}}, \quad &\tu{if } \Delta =2, \\
    \delta^{2/v_{H}}, &\tu{if } \Delta\geq 3.
\end{cases}
\end{equation}}
\end{theorem}


Our next main result deals with the range of $p$  not covered by Theorem \ref{thm:loc-regime}. 
\dteal{It requires a few notation. 
For any graph $H$, we define $\tu{Aut}(H)$ to be the number of automorphisms of $H$. That is, it is the number of bijective maps $\phi:V(H)\mapsto V(H)$ such that $(u,v)\in E(H)$ if and only if $(\phi(u), \phi(v))\in E(H)$. 
For any graph $J$, a fractional independent set is a map $\alpha: V(J) \mapsto [0,1]$  such that $\alpha(u) + \alpha(v) \leq 1$ whenever $(u,v) \in E(J)$. The fractional independence number of $J$, to be denoted by $\alpha^*_J$, is the
largest value of $\sum_{v\in V(J)} \alpha(v)$ among all fractional independent sets $\alpha$ in $J$. Further, a graph $H$ is said to be strictly balanced the ratio $e_J/v_J$, for $\emptyset \neq J \subseteq H$, is uniquely maximized for $J=H$. 
}
\begin{theorem}[Poisson regime]
\label{thm:poisson}
Let $H$ be a strictly balanced graph. Fix $\delta>0$. For $p\in(0,1)$ satisfying $1\ll n^{v_{H}}p^{e_{H}} \ll (\log n)^{\alpha_{H}^{*}/(\alpha_{H}^{*}-1)}$, we have
\begin{equation}\label{eq:poisson-rate-function}
    \lim_{n\rightarrow \infty} \frac{-\log \P (\dteal{\UTH})}{n^{v_{H}}p^{e_{H}}/\tu{Aut}(H)}\ =\ (1+\delta)\log(1+\delta)-\delta =: I(\delta).
\end{equation}
\end{theorem}

\dteal{A few comments are in order. It is easy to see that any strictly balanced graph must be connected and any regular connected graph is strictly balanced. Moreover, it is well known for any regular graph $\alpha^*_H=v_H/2$. Therefore, for any regular connected graph $H$ Theorems \ref{thm:loc-regime} and \ref{thm:poisson} cover the entire sparsity regime $1 \ll np^{\Delta/2} \ll n$ except for the regime $np^{\Delta/2} \asymp (\log n)^{1/(v_H-2)}$, \dred{where a mixture of localized and Poisson behavior is expected}. 

It is known in the literature that for a strictly balanced graph $H$, one has the following dichotomous behavior: 
\[
\lim_{n \to \infty} \P(H \subseteq \Gd) = 
\begin{cases}
0 & \text{if } n^{v_H}p^{e_H}  \ll 1,\\
1 & \text{if } n^{v_H}p^{e_H}  \gg 1.
\end{cases}
\]  
Moreover, if $n^{v_H}p^{e_H} \asymp 1$ then the number of unlabeled copies of $H$ in $\Gd$ converges weakly to some non degenerate Poisson distribution. See \cite[Theorems 1.3 and 1.4]{KSV}. These show that the lower bound assumption in Theorem \ref{thm:poisson} extends all the way up to the {\em appearance threshold} and  is optimal. 

\dred{Let us also mention in passing that the conclusion of Theorem \ref{thm:poisson} continues to hold under the relaxed upper bound $n^{v_H}p^{e_H} \leq C_H (\log n)^{\alpha^*_H/(\alpha^*_H-1)}$ for some $C_H < \infty$ (see Theorem \ref{thm:poisson-new}). Since, the definition of $C_H$ is notationally heavy, we have chosen not to include this extension in this introduction section.}
}

\begin{remark}
\dteal{It is instructive to compare Theorems \ref{thm:loc-regime} and \ref{thm:poisson} with the corresponding results in the Erd\H{o}s-R\'enyi setting (cf.~\cite{ABRB, BK25, HMS}). At a very high level, the upper tail large deviation events in the localized and Poisson regimes are determined essentially by $O(n^2p^\Delta \log(1/p))$ and $\Theta(n^{v_H}p^{e_H})$ many edges, respectively. See Section \ref{sec:proof-outline} for an explanation on this. In their respective regimes  both these bounds are $o(nd)$. Although the edges in $\Gd$ are not independent, using switching arguments and a double counting argument one can show that any collection of $o(nd)$ edges are nearly independent (see for e.g.~\eqref{eq:almost independence lemma-old} and Theorem \ref{thm:almost independence}). 
Heuristically, this essentially results the same speed and rate function in the Poisson regime and in the localized regime for $\Delta \geq 3$. The difference in the rate function in the localized regime, compared to the Erd\H{o}s-R\'enyi setting for $\Delta =2$, is due to the degree rigidity in $\Gd$. These are only heuristics. It needs a considerable amount of work to justify these heuristics, which along with the conditional structure below, are the content of this article.}
\end{remark}


In the context of large deviations in random graphs probably it is more interesting (and often more challenging) to identify the driving mechanisms that cause these rare events. In other words, the problem is to identify the new geometric features that emerge in the random graph under consideration when conditioned on these rare events.  
\dteal{Building on the arguments employed in the proofs of Theorems \ref{thm:loc-regime} and \ref{thm:poisson}, we identify the typical structure $\Gd$ conditioned on these upper tail events.}

\begin{theorem}[Conditional structure in the localized regime]
Consider the same setup as in Theorem \ref{thm:loc-regime}. Fix any $\chi \in (0,1)$. There exists some constant $c_\chi >0$ such that the following statements hold:
\begin{enumerate}
\item[(i)] Assume $\Delta \geq 3$. Then, for all large $n$,  
\[
\P\left( \tu{Cliq}_\chi(\delta)\mid \UTH\right) \geq 1- \exp(-c_\chi r_{n,p, \Delta}),
\]
where we use the shorthand $r_{n,p, \Delta}:= n^2 p^\Delta \log(1/p)$, and 
\[
\tu{Cliq}_{\chi}(\delta) \coloneqq \big\{\exists U\subseteq \set{n} \text{ of size at least }(1-\chi)\delta^{1/v_H} np^{\Delta/2} \text{ and } \min_{u\in U} \tu{deg}_{\G_n^d[U]} ({u})\geq (1-\chi) |U|\big\}.					
\]
\item[(ii)] Let $H=C_{2\ell+1}$ for some $\ell \in \N$. Then, for all large $n$,
\[
\P\left(\tu{Disj-Cliq}_\chi(\lfloor \delta \rfloor, \{\delta\}^{1/v_H}) \mid \UTH \right) \geq 1- \exp(-c_\chi r_{n,p, \Delta}), 
\]
where for any $m \in \N_0:=\N \cup\{0\}$, $\theta \in [0,1)$, and $\chi \in (0,1)$,
\begin{multline}\label{eq:djt-clique}
\tu{Disj-Cliq}_\chi(m, \theta):= \Big\{ \exists \text{ pairwise disjoint subsets } \{U_i\}_{i \in \set{m+1}}  \text{ of } \set{n} \text { such that } \\
\min_{i \in \set{m}} |U_i| \geq (1-\chi)d, \quad |U_{m+1}| \geq (1-\chi) \theta d, \text{ and } \min_{i \in \set{m+1}} \min_{u \in U_i} \deg_{\G_n^d[U_i]}(u)/|U_i| \geq (1-\chi) \Big\}.
\end{multline}
\item[(iii)] Let $H= C_{2\ell}$ for some $\ell \in \N$. Then, for all large $n$, 
\[
\P\left(\bigcup_{m, m' \in \N_0: m+2m' =\lfloor \delta \rfloor} \tu{Disj-Cliq-Bicliq}_\chi(m, m', \{\delta\}^{1/v_H}) \mid \UTH \right) \geq 1- \exp(-c_\chi r_{n,p, \Delta}), 
\]
where for $m, m' \in \N_0$, $\theta \in [0,1)$, and $\chi \in (0,1)$, 
\begin{multline}\label{eq:djt-biclique}
\tu{Disj-Cliq-Bicliq}_\chi(m, m', \theta)\\
:= \Big\{ \exists \text{ pairwise disjoint subsets } \{U_i\}_{i \in \set{m+1}}, \{V_i\}_{i \in \set{m'}}, \text{ and } \{W_i\}_{i \in \set{m'}} 
\text{ of } \set{n} \text { such that } \\
\min\left\{\min_{i \in \set{m}} |U_i|, \min_{i \in \set{m'}} \min\left\{ |V_i|, |W_i|\right\} \right\}\geq (1-\chi)d, \quad |U_{m+1}| \geq (1-\chi) \theta d, \\
\min_{i \in \set{m+1}} \min_{u \in U_i} \deg_{\G_n^d[U_i]}(u)/|U_i| \geq (1-\chi), \text{ and }\\
\min\left\{ \min_{i \in \set{m'}} \min_{u \in V_i} \deg_{\G_n^d[V_i,W_i]}(u)/|W_i|, \min_{i \in \set{m'}} \min_{u \in W_i} \deg_{\G_n^d[V_i,W_i]}(u)/|V_i| \right\} \geq (1-\chi)  \Big\}.
\end{multline}
\end{enumerate}
\label{thm:cond-loc}
\end{theorem}
For $\theta =0$, the lower bound on the cardinality of $U_{m+1}$ in \eqref{eq:djt-clique} does not exclude the case $U_{m+1}=\emptyset$. In that case, we adopt the convention that the inequality $\min_{u \in U_{m+1}} \deg_{\G_n^d[U_{m+1}]}(u)/|U_{m+1}| \geq 1-\chi$ holds vacuously.  Similar convention is adopted in \eqref{eq:djt-biclique}. 

Theorem \ref{thm:cond-loc} shows that for $\Delta \geq 3$, conditionally on the upper tail event, with a large probability the random graph $\Gd$ contains a clique like structure of an appropriate size that is responsible for almost all the excess copies. While for $\Delta =2$, due to the degree bound in $\Gd$ the picture, although similar in flavor, is more intricate. See Figure \ref{fig:localized} below for a schematic representation of this picture for $\Delta=2$. 


\newcommand{\gframeC}{%
\draw[gray!55,line width=0.4pt,rounded corners=2mm] (-9,-13) rectangle (137,12.5);
\node[font=\footnotesize,gray!60,anchor=west] at (-7.8,10.5) {$\mathbb{G}_n^d$};}

\newcommand{\backgroundC}{%
\begin{scope}[gray!45,line width=0.25pt]
\foreach \p/\q in {4/10.6, 13/7.4, 19/9.2, 27/10.5, 33/7.2, 42/10.1,
                   49/7.6, 55/10.3, 63/8.4, 72/10.6, 79/7.3, 87/9.9, 93/7.1,
                   101/10.7, 108/8.1, 116/10.2, 124/7.4, 129/9.9}
  \fill (\p,\q) circle (0.5);
\foreach \a/\b/\c/\d in {13/7.4/19/9.2, 19/9.2/27/10.5,
                         33/7.2/42/10.1, 49/7.6/55/10.3, 63/8.4/72/10.6,
                         72/10.6/79/7.3, 87/9.9/93/7.1, 101/10.7/108/8.1,
                         108/8.1/116/10.2, 116/10.2/124/7.4}
  \draw (\a,\b)--(\c,\d);
\end{scope}}

\newcommand{\bipC}[2]{%
\draw[blob] (#1,0) ellipse (4 and 6);
\draw[blob] (#2,0) ellipse (4 and 6);
\foreach \pa/\pb in {2.57/4.60, 4/0, 2.57/-4.60}
\foreach \qa/\qb in {-2.57/4.60, -4/0, -2.57/-4.60}
  \draw[black!60,line width=0.3pt] ({#1+\pa},\pb) -- ({#2+\qa},\qb);}


\begin{figure}[htbp]
\centering
\begin{tikzpicture}[x=1mm,y=1mm,
blob/.style={draw,line width=0.6pt,fill=gray!12},
lbl/.style={font=\small}]

\begin{scope}

\gframeC \backgroundC
\foreach \x in {4,26,84} \draw[blob] (\x,0) circle (6); \node at (55,0) {$\cdots$};
\foreach \x in {4,26,84} \node[lbl,below=0.6mm] at (\x,-6) {$\widetilde K_d$}; \draw[blob] (125,0) circle (3.8);
\node[lbl,below=0.6mm] at (125,-3.8) {$\widetilde K_{\{\delta\}^{1/\ell}d}$};
\draw[decorate,decoration={brace,amplitude=1.5mm}] (-3,14) -- (91,14) node[midway,above=1.9mm,lbl] {$\lfloor\delta\rfloor$ almost-cliques $\widetilde K_d$};
\end{scope}

\begin{scope}[yshift=-37mm]

\gframeC \backgroundC
\foreach \x in {4,34} \draw[blob] (\x,0) circle (6);
\node at (19,0) {$\cdots$};
\foreach \x in {4,34} \node[lbl,below=0.6mm] at (\x,-6) {$\widetilde K_d$};
\draw[decorate,decoration={brace,amplitude=1.5mm}] (-3,14) -- (41,14) node[midway,above=1.9mm,lbl] {$\lfloor\delta\rfloor-2k$ almost-cliques $\widetilde K_d$};

\bipC{52}{68}\bipC{92}{108}
\node at (80,0) {$\cdots$};
\node[lbl,below=0.6mm] at (60,-6) {$\widetilde K_{d,d}$};
\node[lbl,below=0.6mm] at (100,-6) {$\widetilde K_{d,d}$};
\draw[decorate,decoration={brace,amplitude=1.5mm}] (47.5,14) -- (112.5,14) node[midway,above=1.9mm,lbl] {$k$ almost-bicliques $\widetilde K_{d,d}$};

\draw[blob] (125,0) circle (3.8);
\node[lbl,below=0.6mm] at (125,-3.8) {$\widetilde K_{\{\delta\}^{1/\ell}d}$};
\end{scope}
\end{tikzpicture}

\caption{{\em Typical structure of $\mathbb{G}_n^d$ conditioned on the rare event ${\sf UT}_{C_\ell}(\delta)$ in the localized regime}: Here $\widetilde K_m \subseteq K_m$ denotes a graph with $e(K_m\setminus \widetilde K_m) = o(m^2)$, and  $\widetilde K_{m,m} \subseteq  K_{m,m}$ a bipartite graph with $e(K_{m,m}\setminus \widetilde K_{m,m})=o(m^2)$.
\newline
\texttt{Top panel ($\ell \in 2 \N+1$)}: $\delta d^\ell$ many excess copies of $C_\ell$ are carried by $\lfloor\delta\rfloor$ disjoint almost cliques of size $(1-o(1))d$ together with one of size $(1-o(1)) \{\delta\}^{1/\ell}d$. Rest of the graph contains about $d^\ell$ many copies of $C_\ell$, the typical number of copies of $C_\ell$ in $\mathbb{G}_n^d$. 
\newline
\texttt{Bottom panel ($\ell \in 2 \N$)}: The configuration, responsible for the excess copies, may instead consist of
$\lfloor\delta\rfloor-2k$ almost cliques $\widetilde K_d$ together with $k$ almost bicliques $\widetilde K_{d,d}$, for some $ 0 \leq k \leq \lfloor \delta\rfloor/2$, and as above an almost clique of size $(1-o(1)) \{\delta\}^{1/\ell}d$.}
\label{fig:localized}
\end{figure}
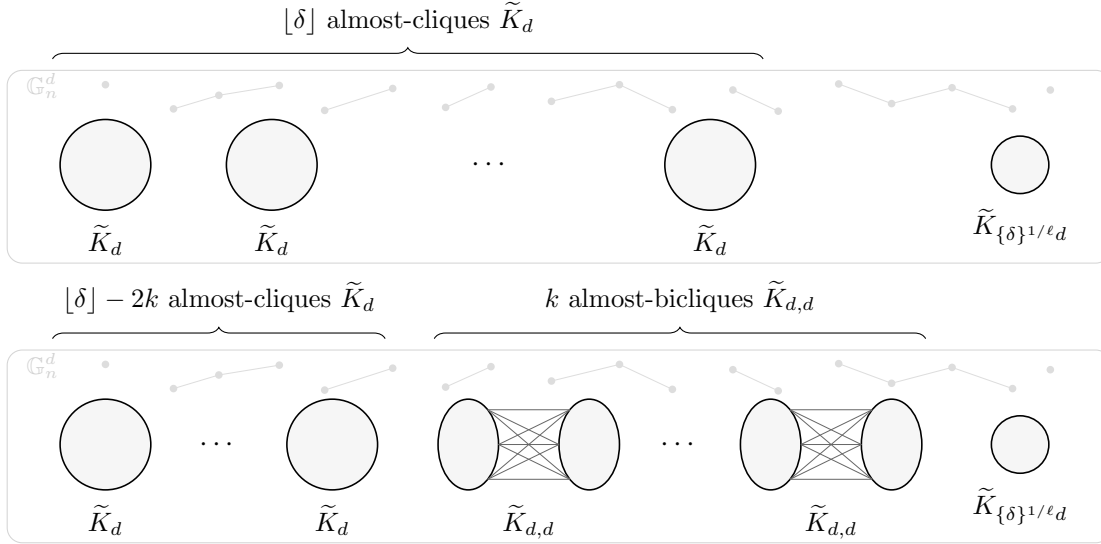

\dteal{Next we proceed to state the result on the conditional structure of $\Gd$ in the Poisson regime. It requires some more notation: a labeled copy $I$ of $H$ in $\Gd$ will be termed `edge non-disjoint' if there exists at least another labeled copy $I'$ of $H$, distinct from $I$, i.e.~$E(I)\neq E(I')$, such that $E(I) \cap E(I') \neq \emptyset$. We set $N_{\tu{non-disj}}(H, \Gd)$ to be the number of such edge non-disjoint copies $H$ in $\Gd$. Write $N_{\tu{disj}}(H,\Gd):= N(H, \Gd) - N_{\tu{non-disj}}(H, \Gd)$ to denote the number of edge disjoint labeled copies of $H$ in $\Gd$. 

                
\begin{theorem}[Conditional structure in the Poisson regime]
\label{thm:conditional-poisson}
Consider the same set up as in Theorem \ref{thm:poisson}. 
There exists some constant $c_{H}>0$ depending only on $H$ such that for any $\chi>0$ sufficiently small and all large $n$, 
\begin{equation*}
    \P\left(N_{\tu{disj}}(H, \Gd)\ \geq \ (1+\delta-\chi)n^{v_{H}}p^{e_{H}} \mid \UTH\right)\geq 1-\exp(-c_{H} (\chi n^{v_{H}}p^{e_{H}})^{1/\alpha_{H}^{*}}\log n).
\end{equation*}
\end{theorem}
}

%


\newcommand{\gframe}{%
\draw[gray!55,line width=0.4pt,rounded corners=2mm] (-9,-23) rectangle (137,20);
\node[font=\footnotesize,gray!60,anchor=west] at (-7.8,17.4) {$\mathbb{G}_n^d$};}

\newcommand{\background}{%
\begin{scope}[gray!45,line width=0.25pt]
\foreach \p/\q in {-4/13.0, 4/17.9, 13/12.2, 19/15.1, 27/17.6, 33/11.9, 42/16.8, 49/12.4, 55/17.2, 63/13.9, 72/17.6, 79/12.1, 87/16.5, 93/11.7, 101/17.8, 108/13.4, 116/16.9, 124/12.3, 129/16.6, 6/-20.6, 15/-19.2, 25/-21.1, 35/-19.5, 44/-20.9, 58/-19.4, 66/-21.3, 76/-19.6, 90/-20.9, 100/-19.4, 108/-21.2, 121/-19.8}
  \fill (\p,\q) circle (0.55);
\foreach \a/\b/\c/\d in {-4/13.0/4/17.9, 13/12.2/19/15.1, 19/15.1/27/17.6,
                         33/11.9/42/16.8, 49/12.4/55/17.2, 63/13.9/72/17.6,
                         72/17.6/79/12.1, 87/16.5/93/11.7, 101/17.8/108/13.4,
                         108/13.4/116/16.9, 116/16.9/124/12.3,
                         6/-20.6/15/-19.2, 25/-21.1/35/-19.5, 35/-19.5/44/-20.9,
                         58/-19.4/66/-21.3, 76/-19.6/90/-20.9, 100/-19.4/108/-21.2}
  \draw (\a,\b)--(\c,\d);
\end{scope}}

\newcommand{\hx}[5]{%
\begin{scope}[shift={(#1,#2)},rotate=#3]
\foreach \a in {0,60,...,300} \coordinate (q\a) at (\a:#5);
\draw[line width=0.55pt,fill=#4] (q0)--(q60)--(q120)--(q180)--(q240)--(q300)--cycle;
\foreach \a in {0,60,...,300} \fill (q\a) circle ({#5*0.155});
\end{scope}}
\newcommand{\hxo}[3]{\hx{#1}{#2}{#3}{none}{4.5}}       
\newcommand{\hxf}[3]{\hx{#1}{#2}{#3}{black!14}{4.5}}   

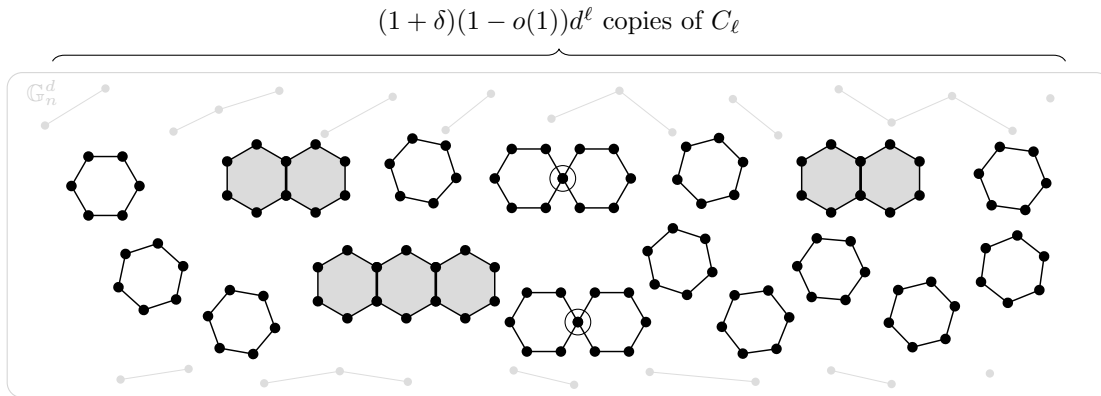
\begin{figure}[htbp]
\centering
\begin{tikzpicture}[x=1mm,y=1mm,lbl/.style={font=\small}]

\gframe
\background

\hxo{4}{5}{0}     \hxo{46}{7}{-12}  \hxo{84}{7}{15}   \hxo{124}{6}{-8} \hxo{10}{-7}{18}  \hxo{22}{-13}{-10}                  \hxo{80}{-5}{-18} \hxo{90}{-13}{8}  \hxo{100}{-6}{-5} \hxo{112}{-12}{14}\hxo{124}{-6}{22}

\hxo{60}{6}{0} \hxo{69}{6}{0}   \draw[line width=0.4pt] (64.5,6) circle (1.7);
\hxo{62}{-13}{0} \hxo{71}{-13}{0}   \draw[line width=0.4pt] (66.5,-13) circle (1.7);

\hxf{24}{6}{30}\hxf{31.8}{6}{30}   \draw[line width=1pt] (27.9,3.75) -- (27.9,8.25);

\hxf{100}{6}{30}\hxf{107.8}{6}{30}   \draw[line width=1pt] (103.9,3.75) -- (103.9,8.25);

\hxf{36}{-8}{30}\hxf{43.8}{-8}{30}\hxf{51.6}{-8}{30}   \draw[line width=1pt] (39.9,-10.25) -- (39.9,-5.75);
\draw[line width=1pt] (47.7,-10.25) -- (47.7,-5.75);

\draw[decorate,decoration={brace,amplitude=1.6mm}] (-3,21.5) -- (131,21.5)
 node[midway,above=2mm,lbl] {$(1+\delta)(1-o(1)) d^{\ell}$ copies of $C_\ell$};

\end{tikzpicture}

\caption{{\em Typical structure of $\mathbb{G}_n^d$ conditioned on the rare event $\UTH$ in the Poisson regime}: Almost all copies of $H=C_\ell$ (with $\ell=6$) are edge disjoint (not necessarily vertex disjoint), unshaded in the figure. Note the stark contrast with Figure \ref{fig:localized}.}
\label{fig:poisson}
\end{figure}

\begin{remark}[Conditional structure in the Erd\H{o}s-R\'enyi setting]
Some of the arguments used in the proofs of Theorems \ref{thm:cond-loc} and \ref{thm:conditional-poisson} can be used to identify the typical structure of the Erd\H{o}s-R\'enyi graph conditioned on the upper tail of the subgraph count. Indeed, upon appropriately adapting the arguments of Theorem \ref{thm:cond-loc}(i)  and Theorem \ref{thm:conditional-poisson} one can check that those extend to the Erd\H{o}s-R\'enyi setting mutatis mutandis. 
\end{remark}

\subsection{Background and related results}\label{sec:background}
The upper tail problem for subgraph counts in an Erd\H{o}s-R\'enyi graph $\G(n,p)$ (we use $p$ also to denote the edge connectivity probability in an Erd\H{o}s-R\'enyi graph), termed as the \textit{infamous upper-tail} problem by Janson and Ruci\'{n}ski \cite{JR2002}, has a very long and rich history. 
Early bounds on the probability of the upper tail event 
were obtained by Kim and Vu \cite{kimandvu}, and Janson, Oleszkiewicz, and Ruci\'{n}ski \cite{JOR}, using polynomial concentration inequalities and moment methods. However, these were not sharp. Chatterjee \cite{chatterjee2012} and DeMarco and Kahn \cite{demarcokahn1} independently established the correct large deviations speed for triangles up to the Poisson regime boundary and later it was extended for cliques in \cite{demarcokahn2}. 


In the dense regime, i.e.~$p\asymp 1$, Chatterjee and Varadhan \cite{varadhan2011} established a large deviation principle on the space of all graphons with respect to the topology induced by the cut metric and therefore by the contraction principle derived large deviations for the number of homomorphisms of $H$ in $\G(n,p)$ as well. 
This further expresses the relevant rate function as the solution of a certain mean-field variational problem on the space of all graphons. There has been a few work identifying the solution of this variational problem. For example, in the context  of upper tail Lubetzky and Zhao \cite{Lubetzky_Zhao_replica_symmetry_2015} determined the exact \textit{replica symmetry} phase boundary for regular $H$, characterizing when the constant graphon is the unique optimizer, while similar results were obtained by Zhao \cite{Zhao17} for the lower tail. 



It is well known in the community that the approach of \cite{varadhan2011} does not quite extend well to the sparse setting, e.g.~for $p \asymp n^{-a}$ for some $a >0$. For the sparse setting, in a seminal work Chatterjee and Dembo \cite{chatterjeedembo}  developed a general large deviation principle for nonlinear functions of independent random variables whose gradients are of low complexity, reducing the large deviation problem at hand to a mean-field variational problem. This, in particular, allowed them to derive the large deviations of subgraph counts for certain ranges of sparsity. The range of sparsity was successively improved in \cite{eldan, cookdembo, Augeri}. In identifying the rate function, similar to the dense case, one needs to determine the solution to the mean-field variational problem at hand. These were obtained in \cite{lubetzkyzhao, BGLZ}.

The recent most breakthrough was made in \cite{HMS} where results analogous to Theorem \ref{thm:poisson} and  Theorem \ref{thm:loc-regime} for $H$ a clique (for other regular graphs the results of \cite{HMS} were not the optimal). Building on the general strategy of \cite{HMS} the work \cite{ABRB} proved Theorem \ref{thm:loc-regime} for all regular connected graph $H$, thereby settling the upper tail large deviations problem for regular subgraph counts in sparse Erd\H{o}s-R\'enyi graphs except for a couple of boundary cases (in terms of sparsity).  Very recently this program was extended in \cite{BK25, antonir2026} for the upper tail of irregular subgraph counts, revealing additional mechanisms and, in some regimes, a possible failure of the so called na\"{i}ve mean-field approximation. 
Let us mention in passing that there have been a few works in studying large deviations in Erd\H{o}s-R\'enyi hypergraphs (cf.~\cite{CDP, CN25}). 

In the context of random regular graphs 
the classical work of Wormald \cite{WORMALD1981} and McKay, Wormald, and Wysocka \cite{McKay_Wormald_Wysocka_2004}  established Poisson limits for short cycles, while subsequent works of Kim, Sudakov and Vu \cite{KSV}, Gao and Wormald \cite{Gao_Wormald_2008},  Gao \cite{PuGao}, and Sah and Sawhney \cite{Sah_Sawhney_2023} studied subgraph thresholds, probabilities and asymptotic distributions, without addressing sharp upper tail large deviations. 

Focusing on the large deviations, as already alluded to, in the dense regime Dhara and Sen \cite{Dhara_Sen_2022} established a large deviation principle in the space of all graphons for the sequence of probability measures induced by the uniform random graph with prescribed degree sequences. This, in particular, yields large deviations of the number of homomorphisms of any subgraph in $\Gd$, for $d$ such that $\lim_{n \to \infty} d/n \in (0,1)$.
%
%
%
%
By extending the framework of \cite{cookdembo} Bhattacharya and Dembo \cite{BD} obtained the first general sharp upper tail results for sparse random regular graphs. They reduced the upper tail problem to a constrained variational problem and solved it for graphs $H$ with regular $2$-cores. Gunby \cite{Gunby2020UpperTO} similarly extended \cite{cookdembo} to general subgraph homomorphism counts. Using fractional vertex covers, generalized H\"{o}lder inequalities, and together with the lower bounds for subgraph probabilities of prescribed subgraphs in $\Gd$, he established sharp asymptotics for a broad class of graphs, while retaining a logarithmic gap in the exponent for general $H$. However, both approaches inherit sparsity restrictions from the quantitative approximation methods of Cook and Dembo \cite{cookdembo}, leaving a substantial part of the localized regime $np^{\Delta/2}\gg(\log n)^{1/(v_H-2)}$ unresolved for connected $\Delta$-regular graphs. 

The best known structural results for triangle rich random regular graphs are \cite{collet_eckmann2002} and  \cite{Hoorn_Lippner_Mossel_2022}. The latter showed that conditioning on at least $cT_{\max}$ triangles, where $T_{\max}=\frac{n}{3} {d\choose 2}$, typically produces disjoint almost-cliques of size approximately $d$. This conditioning is much stronger than a fixed multiplicative deviation above the mean. In contrast, conditional structure results for sparse $\G(n,p)$ were established by Cook and Dembo \cite{CD24}, identifying clique and hub configurations for $p\gg n^{-1/(\Delta+1)}$. In our earlier work \cite{BK25}, we extended these to connected non bipartite regular graphs throughout the full localized regime $np^{\Delta/2}\gg(\log n)^{1/(v_H-2)}$, and obtained analogous results for irregular graphs in the regime $n^{-1/\Delta}\ll p\ll1$.

\begin{remark}[Na\"ive mean-field approximation]
As already noted above, a recurring feature in the large deviations of homomorphism (or subgraph) counts in an Erd\H{o}s-R\'enyi graph is that the large deviation probability (in the localized regime) can be expressed as the solution of the relevant mean-field variational problem. We refer the reader to \cite[Section 1.4]{BK25} for a description of and a more detailed discussion on  this variational problem. 

It is natural to seek whether such a phenomenon continues to hold for $\Gd$ as well. 
Bhattacharya and Dembo \cite{BD} showed that it indeed holds for the upper tail of homomorphism counts in the the sparsity regime they considered. As we will explain below, the same phenomenon continues to hold for the upper tail of subgraph counts through the entire localized regime. 

To this end, let us formulate the relevant variational problem. Let $\xi\coloneqq \left(\xi_{i,j}\right)_{1\leq i<j\leq n}\in [0,1]^{n\choose 2}$, and define the product measure $\mu_{\xi} \coloneqq \bigotimes_{1\leq i<j\leq n} \tu{Ber}(\xi_{i,j})$. We write $\E_{\mu_{\xi}}[N(H, \G_n)]$ for the expected number of labeled copies of $H$ in the random graph $\G_n$ with edge distribution $\mu_{\xi}$. By an abuse of terminology, we call $\xi$ to be $d$-regular if \[\sum_{j=1}^{n}\xi_{i,j} =d, \quad \tu{for every }i\in \set{n},\]
with the convention $\xi_{i,j}=\xi_{j,i}$ and $\xi_{i,i} = 0$. Thus, under $\mu_{\xi}$, every vertex has expected degree $d$.

Define
\[\Phi^{\tu{inj}}_{d,H}(\delta)\coloneqq \inf\{I_{p}(\xi):\xi\in [0,1]^{n\choose 2}\tu{ is }d\tu{-regular},\  \E_{\mu_{\xi}}[N(H,\G_n)]\geq (1+\delta)\E[N(H, \Gd)]\},\]
\[I_{p}(\xi)\coloneqq \sum_{1\leq i<j\leq n}I_{p}(\xi_{i,j}), \quad \tu{and}\quad I_{p}(x)\coloneqq x\log\frac{x}{p}+(1-x)\log\frac{1-x}{1-p}\quad \tu{for }x\in [0,1], \]
with the convention $0\log 0 =0$. We will argue that 
\begin{equation}\label{eq:mean-field}
\lim_{n\rightarrow \infty} \frac{\Phi_{d,H}^{\tu{inj}}(\delta)}{n^{2}p^{\Delta}\log(1/p)} =\Psi_{H}(\delta).
\end{equation}
Combining this 
with Theorem \ref{thm:loc-regime}, we indeed obtain that the na\"{i}ve mean-field approximation holds in the entire localized regime $np^{\Delta/2}\gg (\log n)^{1/(v_H-2)}$ and $p\ll1$.

Turning to derive \eqref{eq:mean-field} we further introduce some more notation. 
Let $\G_{n,\xi}$ to be the weighted graph on vertex set $\set{n}$ with edge weights $(\xi_{i,j})_{1\leq i<j\leq n}=\xi$. Set 
\[
N(H, \G_{n,\xi}) \coloneqq \sum_{\substack{\phi:V(H)\rightarrow \set{n}\\ \tu{injective}}}\sum_{(u,v)\in E(H)}\xi_{\phi(u), \phi(v)} \quad \text{ and } \quad {\rm Hom}(H, \G_{n,\xi}) \coloneqq \sum_{\phi:V(H)\rightarrow \set{n}}\sum_{(u,v)\in E(H)}\xi_{\phi(u), \phi(v)}.
\]
and 
\[\Phi_{d,H}(\delta)\coloneqq \inf\{I_{p}(\xi):\xi\in [0,1]^{n\choose 2}\tu{ is }d\tu{-regular},\  {\rm Hom}(H, \G_{n,\xi})\geq (1+\delta)\E[N(H, \Gd)]\}.\]
As 
$\E_{\mu_{\xi}}[N(H, \G_n)] = N(H, \G_{n,\xi})\leq \tu{Hom}(H, G_{n,\xi})$ it is immediate that $\Phi_{d,H}(\delta) \leq \Phi_{d,H}^{\tu{inj}}(\delta)$. Further, by \cite{KSV}, $\E[N(H,\Gd)]=(1+o(1))n^{v_H}p^{e_H}$. Therefore, applying \cite[Theorem 1.1]{BD}, for a $\Delta$-regular graph $H$ we get
\begin{equation}\label{eq:mean-field-1}
\liminf_{n\rightarrow \infty} \frac{\Phi_{d,H}^{\tu{inj}}(\delta)}{n^{2}p^{\Delta}\log(1/p)} \geq \lim_{n\rightarrow \infty} \frac{\Phi_{d,H}(\delta)}{n^{2}p^{\Delta}\log(1/p)} =\Psi_H(\delta).
\end{equation}
Although \cite[Theorem~1.1]{BD} is stated for the regime $ n^{-1/\Delta}\ll p \ll 1$ the reader may inspect that same proof continues to hold for the entire sparsity regime $n^{-2/\Delta} \ll p \ll 1$. Finally it remains to note that for the candidate optimizer $\widehat \xi$ used in the proof of \cite[Theorem 1.1]{BD} one also has that $N(H, \G_{n,\widehat \xi})\geq (1+\delta-o(1))n^{v_H}p^{e_H}$ and $I_p(\widehat \xi) = (1-o(1)) \Psi_H(\delta) n^2p^\Delta \log(1/p)$. Therefore, one deduces \eqref{eq:mean-field} from \eqref{eq:mean-field-1}.  




\end{remark}

\subsection{Notational conventions}\label{sec:notation} 
\dteal{For any two sequences of positive real numbers $\{a_{n}\}$ and $\{b_{n}\}$, we write $b_{n}=o(a_{n})$, $a_n\ll b_{n}$, and $b_{n}\gg a_{n}$ to denote $\lim_{n\rightarrow\infty } a_{n}/b_{n} = 0 $. We use $a_{n} \lesssim b_{n}$ and $b_n \gtrsim a_n$ to denote $\limsup_{n\rightarrow \infty} a_{n}/b_{n}<\infty$.  
We write $a_{n}\asymp b_{n}$ or $a_{n}=\Theta(b_{n})$ if $a_{n}\lesssim b_{n}$ and $a_n \gtrsim b_{n}$. For $x \in \R$ we use the standard notation $\lfloor x \rfloor$, $\lceil x \rceil$, and $\{x\}$ to denote the largest integer smaller than or equal to $x$, the smallest integer greater than or equal to $x$, and the fractional part of $x$, respectively.

For a graph ${G}$ we write $V({G})$ and $E({G})$ to denote its vertex set and edge set respectively. We will use $v({G}) = |V({G})|$ and $e({G}) = |E({G})|$, where $|\cdot|$ is the cardinality of a set. For typographic reasons we also use $v_{{G}}$ and $e_{{G}}$ instead of $v(\tu{G})$ and $e({G})$. We also write $K_n$ to denote the complete graph on $\set{n}:=\{1,2,\ldots,n\}$. The notation $K_{n,n}$ is used to denote the complete bipartite graph on $\set{2n}$ with all its edges from $\set{n}$ to $\set{2n}\setminus \set{n}$. We write $C_\ell$ to denote the cycle graph of length $\ell$. 
For any finite set $S \subseteq V$ we write ${\bf 1}_S$ to be vector that is one on $S$ and zero on $V\setminus S$. 
For $S, T \subseteq V(G)$ disjoint, we write $G[S]$ to denote the subgraph of $G$ spanned by the edges with both ends in $S$, and $G[S,T]$ to denote the bipartite subgraph spanned by the edges with one end in $S$ and the other end in $T$.} 

\subsection{Outlines of the proofs}\label{sec:proof-outline} \dteal{In this short section we highlight the key ideas behind the proofs of the main results.}

\dteal{A key difference between an Erd\H{o}s-R\'enyi graph and a random regular graph is that the edges in the latter random graph are not independent. However, a moment's thought reveals that any pair of edges in $\Gd$ are weakly dependent. Using a forward and backward switching and a double counting argument, as employed in \cite{KSV}, one can quantify this observation to obtain a bound of the following form:  
\begin{equation} \label{eq:almost independence lemma-old}
         \widetilde{p}\ \cdot  \left(1-O\left(p+\frac{e(G)}{nd}\right)\right) \cdot(1-C_{uv,G}) \leq \P\left(\{(u,v)\}\subseteq \Gd \big| G \subseteq \Gd\right)\ \leq \widetilde{p}\ \cdot  \left(1+ O\left(p+\frac{e(G)}{nd}\right)\right),
            \end{equation} 			
for any $G\subseteq K_n$ and $(u,v) \notin E(G)$ so that $e(G)=o(nd)$ and the maximum degree of the subgraph spanned by the edge disjoint union of $G$ and the edge $\{(u,v)\}$ is at most $d$, where  
\begin{equation}\label{eq:p_tilde}
            {\widetilde{p} = \widetilde{p}_{uv,G_\star}}\coloneqq p\cdot \left(1-\frac{\tu{deg}_{G_\star}(u)}{d}\right)\cdot\left(1-\frac{\tu{deg}_{G_\star}(u)}{d}\right) \quad \text{ and } \quad C_{uv,G} \coloneqq \left(d-\min\{\tu{deg}_{G}(u),\tu{deg}_{G}(v)\}\right)^{-1}.
        \end{equation}
%
Iterating \eqref{eq:almost independence lemma-old} (we actually use a variant of this bound; see Section \ref{sec:almost independence}) one can indeed show that a collection of `small' number of edges in $\Gd$ are nearly independent. This observation is used repeatedly in the proofs of all four theorems mentioned in the introduction section in various capacities.}

\dteal{
\subsubsection{Poisson regime} The proof relies on the following simple observation: Let $X$ be a nonnegative integer random variable with mean $\mu \gg 1$ such that 
\begin{equation}\label{eq:poisson-high-mom}
\left|\E(X^{\underline{t}}) - \mu^{\underline{t}}\right| =o(\mu^t), \qquad \text{ for all } t \in \set{\delta(1-o(1))\mu, \delta (1+o(1))\mu},
\end{equation}
where $x^{\underline{t}}:= x(x-1) \cdots (x-t+1)$ and $\delta >0$ is fixed but arbitrary. Then, it is not difficult to deduce that
\[
- \mu^{-1} \cdot \log \P(X \geq (1+\delta) \mu) = I(\delta)(1+o(1))
\]
(recall the definition of $I(\cdot)$ from \eqref{eq:poisson-rate-function}). The proof is not difficult: for one direction one can use Markov's inequality, while the other direction follows from a tilting argument. For further details one may look at the proof of \cite[Proposition 8.3]{HMS}. 

In the context of upper tail large deviations of subgraph counts in random graphs, this idea was implemented in \cite{HMS} to derive the large deviations for regular graphs and later in \cite{BK25} to obtain the same for strictly balanced graphs in Erd\H{o}s-R\'enyi random graphs in the sparsity regimes, where a Poisson behavior was expected. This required bounding expectations of `cluster's of given sizes.  

At a high level, in this current work we also use the same general idea to prove Theorem \ref{thm:poisson}. However, here as we work with random regular graphs, the independence among the edges that were extensively used in \cite{BK25, HMS}, is now lost. In partciular, neither the edges in a given cluster nor the edges across disjoint clusters are indpendent. To navigate this challenge we  repeatedly need to use the near independence of a small number of edges (in this case for $\Theta(\mu)$ many edges, with $\mu= (1+o(1)) n^{v_H}p^{e_H}$, as needed in \eqref{eq:poisson-high-mom})
to obtain (upper and lower) bounds on the conditional expectation of  clusters of varying sizes to 
leading to the desired upper and lower bounds as required by \eqref{eq:poisson-high-mom}. The proof of Theorem \ref{thm:conditional-poisson} also follows from this cluster analysis argument as well.} 

\subsubsection{Localized regime:~large deviations speed and rate function}\dteal{The proosf of the lower and upper bounds follow two different routes. The proof of the lower bound is much simpler. Here, as in the Erd\H{o}s-R\'enyi setting, the strategy is to plant a subgraph $G \subseteq \Gd$ with a small number of edges, in this case $\Theta(n^2p^\Delta)$ many edges, show that the number of labelled copies of $H$ in the rest of $\Gd$ is $(1-o(1))n^{v_H}p^{e_H}$ with probability $1-o(1)$, $N(H,G) \geq \delta(1-o(1))n^{v_H}p^{e_H}$, and the negative $\log$-probability of the event $\{G \subseteq \Gd\}$ is upper bounded by $(1+o(1))n^2p^\Delta \log(1/p) \Psi_H(\delta)$. 
To implement these steps, the near independence of a small number of edges is needed to be used again. 
}

\dteal{The proof of the large deviation upper bound proceeds via a different route. Our approach uses 
the general strategy that was introduced in \cite{HMS} and later extended in \cite{ABRB} to tackle the sparser regime and arbitrary regular graphs. To explain the idea let us first describe some of their key steps.}  

\dteal{At a high level, \cite{ABRB, HMS} aim to find a collection of subgraphs $\mathcal{G}:= \mathcal{G}_\varepsilon$ such that
\begin{equation}\label{eq:idea-core}
\P(\nUTH) \leq (1+\varepsilon) \P(\exists G \in \mathcal{G} \text{ such that } G \subseteq \G),
\end{equation}
for any $\varepsilon >0$, where $N(H,G) \geq \delta (1-o(1)) \E_\G[N(H, \cdot)]$ for all $G \in \mathcal G$,  $\E_\G[N(H, \cdot)]$ denotes the expected number of copies of $H$ in $\G$,
\[
\nUTH:= \{N(H,\G) \geq (1+\delta) \E_{\G}[N(H,\cdot)]\},
\]
and $\G$ is the random graph under consideration. Applying \eqref{eq:idea-core} one then uses a union bound to obtain
\begin{equation}\label{eq:idea-core-1}
\P(\nUTH) \leq (1+\varepsilon)  \sum_{e_0} \mathcal{N}_{e_0} \mathfrak{p}_{e_0},
\end{equation}
where the sum is over all allowable choices of the number of edges $e_0$ for subgraphs belonging to $\mathcal{G}$, 
\[
\mathcal{N}_{e_0}:= |\mathcal{G}_{e_0}|, \quad \mathcal{G}_{e_0}:= \{G \in \mathcal{G}: e(G)=e_0\}, \quad \text{ and } \quad \mathfrak{p}_{e_0}:= \P(G \subseteq \G) \text{ for } G \in \mathcal{G}_{e_0}.
\]
In \cite{HMS} it was then shown that, when $\G$ is the Erd\H{o}s-R\'enyi graph, for certain classes of graphs and certain sparsity regimes, one has 
\begin{equation}\label{eq: log-mathcal-N}
\log \mathcal{N}_{e_0} =o(\log(1/\mathfrak{p}_{e_0})), \quad \text{ for all allowable choices of } e_0. 
\end{equation}
The proof completes by observing that $(1+o(1))\log (1/\mathfrak{p}_{e^\star_0})$ equals the desired large deviation speed times the rate function, where $e_0^\star$ is the minimum number of edges a subgraph $G \in \mathcal{G}$ must have.} 

\dteal{However, \eqref{eq: log-mathcal-N} no longer holds for sparser regime, for e.g.~when $H=C_4$ 
and $p \ll n^{-1/2-\xi}$, where $\xi>0$ some small positive constant.  To resolve this issue, \cite{ABRB} noted that \eqref{eq:idea-core} can be strengthened to obtain the following: 
%
\begin{equation}\label{eq:idea-strong-core}
(1+\varepsilon)^{-1}\P(\nUTH)  \leq \P(\exists G \in \widehat{\mathcal{G}}_C \text{ such that } G \subseteq \G) 
+ \P(\exists G \in \mathcal{G} \text{ with } e(G) \geq C e_0^\star \text{ such that } G \subseteq \G),
\end{equation}
where $\widehat{\mathcal G}_C \subseteq \mathcal{G}$ is some subclass of graphs and $C < \infty$ some large constant determined during the course of the proof. For $H$ non-bipartite they showed that an analog of \eqref{eq: log-mathcal-N} holds for the subclass $\widehat{\mathcal{G}}_C$. However, for $H$ bipartite it fails again. At a high level the obstacle stems from the existence some bipartite subgraph $\emptyset \neq G'\subseteq G \in \mathcal{G}$ containing a non negligible fraction of $N(H,G)$. They navigated this challenge by noting that for such scenarios one must have $e(G) \geq (1+c) e_0^\star$, for some $c >0$. Hence, it is not necessary to have an analog of \eqref{eq: log-mathcal-N} to carry out the union bound. Instead, they derived an appropriate bound on the number of these cases so that it perfectly balances the probability of the existence of such graphs yielding the required upper bound on the upper tail probability. Then, it remains to tackle the second term in the RHS of \eqref{eq:idea-strong-core} and show that it is negligible at the large deviations scale. For $n^2 p^\Delta \geq (\log n)^{O(1)}$ the argument to tackle that term turns out to be fairly simple, while in the other regime \cite{ABRB} used a rather involved chaining type argument. 



}

\dteal{To implement the above strategy in the context of a random regular graph, as the reader have alreday observed, 
the first obstacle 
is the dependence among its edges. To this end, upon investigating the proof of \eqref{eq:idea-core}, that build on the moment argument of Janson, Oleszkiewicz, and Ruci\'nski \cite{JOR}, one notes that it does {\em not} use any assumption on the joint law of the edges, and hence extends to the current setting. Further, one notes that the maximum number of edges a graph $G \in \mathcal{G}$ can have is $o(nd)$. Therefore, the near independence of the edges in $\Gd$ mentioned above, can be invoked, and thus the strategy explained above should work in random regular graph setting, modulo the changes due to the difference between the rate functions in the Erd\H{o}s-R\'enyi and the random regular graph settings. 
}


\dteal{ 
It should be mentioned that the arguments of \cite{ABRB} that obtains a bound on the first term in the RHS of \eqref{eq:idea-strong-core} is not sufficiently strong to yield a meningful bound when tasked with finding the typical structure of $\Gd$ conditioned on $\UTH$ (for e.g.~it will not yield \eqref{eq:cond-str-idea}, see Section \ref{subsubsec:loc-cond} below for its necessity). Hence, apprpriate refinements are necessary. Details are in Section \ref{sec:pf-lm-513}. 

It is also worth adding that while tackling the second term in the RHS of \eqref{eq:idea-strong-core} we find a new argument that byapsses the chaining argument of \cite{ABRB}, and is substantially simpler. The key observation is that one can decompose the vertices of any graph $G \in \mathcal G$ into high and low degree vertices (unlike \cite{ABRB, HMS} the degrees of these low degree vertices are bounded by some constant depending only on $H$) such that the number of edges incident to the set of high degree vertices is at least  $\bar c_H e(G)$, for some $\bar c_H>0$ depending only on $H$. This then yields the bound $\log {\mathcal N}_{e_0} \leq (1-\bar c_H/2) {\bm e} \log(1/p)$ (compare with \eqref{eq: log-mathcal-N}) near the Poisson regime boundary.  Therefore, by choosing $C$ large in \eqref{eq:idea-strong-core}, one can now indeed apply a union bound to control the second term in the RHS of \eqref{eq:idea-strong-core}.
}

\subsubsection{Localized regime:~conditional structure}\label{subsubsec:loc-cond}
\dteal{While proving the upper bound on the large deviation probability, with some additional effort, we also obtain that  
\begin{equation}\label{eq:cond-str-idea}
\P(\exists G \in \mathcal{G} \text{ with } e(G) \geq (1+\varepsilon ^{1/2}) e_0^\star \text{ such that } G \subseteq \Gd) \ll \P(\UTH),
\end{equation}
where we recall that $e_0^\star$ is the minimum number of edges a graph $G \in \mathcal{G}$ must possess. Therefore, to derive the condtional structure of $\Gd$ conditioned on the upper tail event one simply to needs to understand the structure of graphs $G \in \widetilde{\mathcal{G}}_\varepsilon:= \{G \in \mathcal{G}: e(G) \leq (1+\varepsilon ^{1/2}) e_0^\star\}$. For $\Delta \geq 3$ the required structural proerty follows from our previous work \cite{BK25}. 

The case $\Delta=2$, i.e.~when $H=C_\ell$ for some $\ell \geq 3$, requires a new argument. Recalling the definition of the rate function \eqref{eq:rate function defn} and the near independence among the edges in $\Gd$, one notes that planting any graph that is a disjoint union of $\lfloor \delta \rfloor$ copies of $K_d$ 
(or $\lfloor \delta \rfloor -2k$ many copies of $K_d$ and $k$ many copies of $K_{d,d}$ 
for some integer $k$ such that $0 \leq k \leq \lfloor \delta\rfloor/2$) and a copy of 
$K_{\{\delta\}^{1/\ell}d}$, provides the correct lower bound on the large deviation probability. 
Hence, in light of \eqref{eq:cond-str-idea}, Theorem \ref{thm:cond-loc} demands that any graph $G \in \widetilde{\mathcal{G}}_\varepsilon$ must consist of loosely connected copies of some approximate versions of these cliques and blicliques. 

Since, ${\rm Hom}(C_\ell, G)$ can be expressed as the trace of powers of the adjacency matrix, a natural direction is to employ some spectral argument. Indeed, as a first step, relying on the strict concavity of the map $(x_1,x_2,\ldots, x_s) \mapsto \sum_{i \in \set{s}} x_i^{2/\ell}$ for any $s \in \N$ and $\ell \geq 3$ and some assoaicted functions, and the degree upper bound, we show that $A(G)$, the adjacency matrix of a graph $G$, for any graph $G \in \widetilde{\mathcal{G}}_\varepsilon$, must have $\lfloor \delta \rfloor$ many eignevalues close to $d$ and an eigenvalue close to $\{\delta\}^{1/\ell} d$ in abolsute values (if $\ell \in 2 \N+1$ then the absolute values can be further removed), and the rest of the eigenvalues have a negligible $\ell^2$-mass.  The task then reduces to establishing that such a graph must be of the form mentioned in Therem \ref{thm:cond-loc}. 
}

\dteal{For ease in explaining the core idea, let us consider the case $\delta=2$. In this case, the step above yields that for $\ell \in 2\N+1$ there are two eigenvalues of $A(G)$ that are close to $d$, and for $\ell \in 2 \N$ one eigenvalue is close to $d$ and the other is close to $\pm d$. 
These together with the fact that the maximum degree is bounded by $d$ yields that for $\ell$ odd, the second eigenvalue of the normalized laplacian $\mathcal{L}(G)$ is close to zero, and therefore $G$ must possess two loosely connected components in it. This is the standard Cheeger inequality (cf.~\cite[Chapter 2]{Chung97}). On the other hand, if $\ell$ is even and has an eigenvalue close to $-d$, one finds that the largest eigenvalue of $\mathcal{L}(G)$ 
is close to two and therefore it must contain a bipartite compoenent that is loosely connected to the rest of the graph (cf.~\cite{Trev12}). However, neither of these assert any bounds on the sizes of  these components. In fact, unless additional assumptions are imposed, such as regularity, any such bound on the sizes of the relevant components is simply false.} 

\dteal{Nevertheless, in our setting we make the following simple yet a key observation: if we augment the graph $G$ by adding a ghost vertex $v_\star$ and edges between this ghost vertex and any vertex $v \in V(G)$ with edge weight $(d-\deg_G(v))$ to obtain a new (wieghted) graph $\widetilde G$, then the aforementioned bounds on the eigenvalues of $A(G)$ 
still translate to bounds on the eigenvalues of $\mathcal{L}(\widetilde G)$, similar to those for $\mathcal{L}(G)$ as mentioned above. 
Further, the added advantage is that all vertices in $V(G) \subseteq V(\widetilde G)$ have weighted degree $d$ in $\widetilde G$. Therefore, one can now use this added regularity and employ \cite{Chung97} and \cite{Trev12} to obtain the necessary structure as required in Theorem \ref{thm:cond-loc}. 

For $\delta \in \N$ one can iteratively apply the above scheme. However, this makes the error bounds, although not so important for our end goal, lossy. So, we resort to the more recent works \cite{LOGT2014} and \cite{Liu2015}, where they develop higher order Cheeger and dual Cheeger inequalities for weighted graphs. For $\delta \notin \N$ an additional step is required. Namely, one needs to show that in a near regular graph (this would follow from the last step), one cannot have only two large eigenvalues unless both of them are close to each other in absoute values. The proof relies on two observations: under the assumption that those two large eigenvalues are separated from each other the leading eigenvector of $A(G)$ 
is close to a constant vector, 
and the eigenvector corresponding to the sceond largest eigenvalue can essentially only take two values.} 





\subsection*{Organization of the rest of the paper} \dteal{Section \ref{sec:almost independence} provides results on the near independence of a small number of edges that are necessary for the main results of this article. In Section \ref{sec:poisson-regime} we prove large deviations speed and rate function in the Poisson regime, as well as the typical structure of $\Gd$ conditioned on the upper tail event in the same regime. Sections \ref{sec:lower bound} and \ref{sec:loc-regime-ubd} yield large deviation lower and upper bounds in the localized regime, respectively. Finally, in Section \ref{sec:conditional-loc-regime} we prove the typical behavior of $\Gd$ conditioned on the upper tail event in the localized regime.} 


\subsection*{Acknowledgements} AB thanks Riddhipratim Basu for some helpful comments. Research was partially supported by DAE Project no.~RTI4001 via ICTS, by an ANRF ARG MATRICS Grant (grant no.~ARGM/2025/001927/MTR), and by the Infosys Foundation via the Infosys-Chandrashekharan Virtual Centre for RandomGeometry.

\subsection*{Declaration on AI use} All the arguments and results in this paper were completed by the authors by mid-July 2026. The authors spent the subsequent three months in writing up the manuscript. {AI tools were used only to assist in preparing the TikZ code for the figures. No AI tools were used in generating proof ideas or in assistance with writing.}

\section{Near independence of a small number of edges}\label{sec:almost independence}
A key ingredient in our analysis of the upper tail is a quantitative near independence estimate for edge probabilities in a random $d$-regular graph. Since the edges of $\Gd$ are globally dependent, conditioning on the presence of a graph $G$ \dteal{alters} the probability that an additional edge $(u,v)$ is present. \dteal{We need to estimate these change (from both sides) when $G$ contains a small number of edges. One such control are the bounds in \eqref{eq:almost independence lemma-old}, which are obtained by an adaptation of the switching arguments of \cite{KSV}. After completion of the proof of \eqref{eq:almost independence lemma-old} and those of subsequent bounds using \eqref{eq:almost independence lemma-old}, we learnt of the work \cite{PuGao} which also derive similar bounds. Hence, in the exposition below we do not present the proof of \eqref{eq:almost independence lemma-old}, and rather directly work with the bounds from \cite{PuGao}. 
}	
%
%


\dteal{For ease in writing, set} 
\begin{equation}\label{eq:def-Xd}
    \widehat{\Xd} := \left\{G \subseteq K_n:\ \tu{deg}_{G}(i)\leq d,\ \forall i\in \set{n} \right\}, \text{ and } \Xd := \left\{G \subseteq K_n:\ \tu{deg}_{G}(i)= d,\ \forall i\in \set{n} \right\}.
\end{equation}
\dteal{In words}, $\Xd$ is the collection of $d$-regular graphs on $\set{n}$, while $\widehat{\Xd}$ consists of graphs that may occur as subgraphs of a $d$-regular graph.	

\begin{theorem}[{\cite[Theorem~2]{PuGao}}]
    \label{thm:almost independence}
        Assume $2\leq d=np\ll n$. Let $G\in \widehat{\Xd}$ with $e(G)=o(nd)$. Fix two vertices $u,v\in \set{n}$ such that $ (u,v)\notin E(G)$ and $\max\{\tu{deg}_{G}(u), \tu{deg}_{G}(v)\}\leq d-1$. \dteal{Let $\widetilde p_{uv,G}$ be as in \eqref{eq:p_tilde}.} 
        Then for sufficiently large $n$, we have
        \begin{equation}
            \label{eq:almost independence lemma}
            \P\left(\{(u,v)\}\subseteq \Gd \big| G\subseteq \Gd\right)\ = \widetilde{p}\cdot  \left(1+O\left(p+\frac{e(G)}{nd}\right)\right)
        \end{equation}
\end{theorem}

\dteal{We have stated \cite[Theorem~2]{PuGao} in a slightly simplified form, tailored to our setting. 
The author in \cite{PuGao} proves a more precise expansion for the conditional edge probability. The proof is based on a switching argument that compares $d$-regular graphs containing $(u,v)$ with those not containing $(u,v)$, while preserving the prescribed graph $G$. The resulting estimate shows that, when $G$ contains $o(nd)$ edges, the leading effect of the conditioning is local. It replaces the marginal edge density $p=d/n$ by the product of the fractions of unused degree at $u$ and $v$.}

In \dteal{many} of our applications, we use Theorem \ref{thm:almost independence} only a bounded number of times, typically to compute the probability that $\Gd$ contains a fixed copy of $H$, conditional on the presence of a sparse collection of edges. The following simpler consequence, which bounds both residual degree factors by their worst case, will be more convenient \dteal{for those uses}.
\begin{corollary}
\label{cor:almost independence lemma(simple)}
Assume the same setup as Theorem \ref{thm:almost independence}. Set  $\eta \coloneqq d^{-1}\cdot\max\{\tu{deg}_{G}(u), \tu{deg}_{G}(v)\}$, then for all sufficiently large $n$,
\begin{equation}
    (1-\eta)^{2}\cdot p^{1+o(1)}\leq(1-o(1))\cdot (1-\eta)^{2} \cdot p\leq \P\left(\{(u,v)\}\subseteq \Gd\big| G\subseteq \Gd\right)\leq (1+o(1))\cdot p\leq p^{1-o(1)}.
\end{equation}
\end{corollary}

For Sections \ref{sec:poisson-regime} and \ref{sec:lower bound}, the coarser form above \dteal{will no longer be}  sufficient. In these arguments we may expose a growing collection of edges, and hence need to keep track both of the loss in available degree and of the accumulated error at each step. We therefore record the following consequence of Theorem \ref{thm:almost independence}, obtained by exposing the edges of a prescribed graph $I$ sequentially.		
\begin{corollary}
    \label{cor:almost independence lemma(poisson)}
    Assume $2\leq d =np\ll n$.
    Let $G, I\in \widehat{\Xd}$ satisfy $G\cup I\in \widehat{\Xd}$, \dteal{$E(I)\cap E(G)=\emptyset$} and $e(G\cup I)=o(nd)$. 
Define
    \begin{equation}
        \label{eq:D_I,G definition}
        D_{I, G} \coloneqq \prod_{u\in V(I)} \prod_{i=\dteal{0}}^{\tu{deg}_{I}(u)-1}\left(1-\frac{\tu{deg}_{G}(u)+i}{d}\right).
    \end{equation}
Then there exists an absolute constant $L>0$ such that, with
\begin{equation}
\label{eq:error bound on independence}
f_{k} \coloneqq  \exp\left( L\left(p+\frac{k}{nd}\right)\right),
\end{equation}
we have, for sufficiently large $n$,
    \begin{equation}\label{eq:bounds-on-copy}
        p^{e(I)}\cdot D_{I,G}\cdot f_{e(G\cup I)}^{-e(I)}\leq \P(I\subseteq \Gd| G\subseteq \Gd )\leq  p^{e(I)}\cdot D_{I,G}\cdot f_{e(G\cup I)}^{e(I)}.
    \end{equation}
\end{corollary}
Here $D_{I,G}$ accounts for the cumulative loss of available degree at the vertices of $I$ as its edges are sequentially exposed, while the factors $f_{e(G\cup I)}^{\pm e(I)}$ collect the error terms arising from successive applications of Theorem \ref{thm:almost independence}.
\begin{proof}
    Assume $e(I) = m$ and let $e_{1}, \ldots, e_{m}$ be an enumeration of edges in $I$. Consider the nested sequence of subsets:
    \begin{equation*}
        \dteal{G \eqqcolon G_{1}\subseteq G_{2}\subseteq  \cdots \subseteq G_{m}},
    \end{equation*}
    where for $i\in \{2,\ldots, m\}$ we define
    \begin{equation*}
        G_{i} \coloneqq G\cup \bigcup_{j=1}^{i-1} \{e_{i}\}.
    \end{equation*}
    Using Theorem \ref{thm:almost independence} we can expand as follows:
    \begin{equation*}
        \P(I\subseteq \Gd|G\subseteq \Gd) = \prod_{i=1}^{m}\P(\{e_{i}\}\subseteq \Gd|G_{i}\subseteq \Gd) = \prod_{i=1}^{m}\widetilde{p}_{e_{i},G_{i}}\cdot \left(1+O\left(\frac{d}{n}+\frac{e(G_i)}{nd}\right)\right),
    \end{equation*}
    where $\widetilde{p}_{e_{i},G_{i}}$ is as defined in \eqref{eq:p_tilde}. See that  for any $i\in \set{m}$, we have $\tu{deg}_{G_{i}}(u) = \tu{deg}_{G}(i)+ \sum_{j=1}^{i-1}\mathds{1}(u\in e_{j})$. Since each $\dteal{u} \in V(I)$ appears exactly $\tu{deg}_{I}(\dteal{u})$ many times in the sequence $(e_{j})_{j\in \set{m}}$, therefore using the previous observation we can write
    \begin{equation}
        \label{eq:p,D}
        \prod_{i=1}^{m}\widetilde{p}_{e_i, G_{i}} = \dteal{p^m} \dteal{\prod_{u \in V(I)}} \prod_{i=1}^{m} \cdot \prod_{u\in e_{i}}\left(1-\frac{\tu{deg}_{G_{i}}(u)}{d}\right) = p^{m}\cdot D_{I,G}.
    \end{equation}
    Further note that we have $d=o(n)$ and $e(G_i)\leq e(G\cup \dteal{I})=o(nd)$ for each $i$. This yields,
    \begin{equation}
        \label{eq:bound on error}
        1+O\left(\frac{d}{n}+\frac{e(G_i)}{nd}\right) = 1+O\left(\frac{d}{n}+\frac{e(G\cup I)}{nd}\right) = \exp\left(O\left(\frac{d}{n}+\frac{e(G\cup I)}{nd}\right)\right),  
    \end{equation}
where the last step follows by using the inequality $e^{-2x}\leq 1-x\leq 1+x\leq e^{x}$ whenever $0<x\leq \frac{1}{2}$. Combining \eqref{eq:p,D} and \eqref{eq:bound on error} gives our desired result.
\end{proof}

\section{Poisson regime}\label{sec:poisson-regime}
In this section \dteal{we prove} Theorems \ref{thm:poisson} and \ref{thm:conditional-poisson}. In fact, \dteal{we prove} a stronger version of Theorem \ref{thm:poisson} which \dteal{allows us} to consider the sparsity \dteal{regime $n^{v_H}p^{e_H} \leq C_H (\log n)^{\alpha^*_H/(\alpha_H^*-1)}$, for some $C_H < \infty$ to be defined below}. To state that result we need to introduce the following set of notation: For any given $C>0$, set 
    \begin{equation}
        \label{eq: defn of M_C}
        M_{C} \coloneqq \sup_{x\geq 1} \frac{f(x)+\log C}{x^{1/\alpha_{H}^{*}}}, \quad \text{ where } f(x):= x\log x - (x-1)\log(x-1).
    \end{equation}
Since $f(x)/ x^{1/\alpha_H^*} \to 0$ as $x \to \infty$, it follows that for any $C \in (0, \infty)$ the constant $M_C < \infty$. 		
    Further define 
    \begin{equation}
        \label{eq:def of constants}
        \lambda:=\frac{e_{H}}{v_{H}},\,  \lambda^{*} := \max_{\emptyset \neq J\subsetneq H} \frac{e_{J}}{v_{J}},\, r_{H}\coloneqq 1+\frac{1}{v_{H}}\left(\frac{1}{\lambda^{*}}-\frac{1}{\lambda}\right), \, \phi_{H}\coloneqq \frac{1}{\lambda}\cdot \frac{r_{H}-1}{r_H}, \, \text{ and } \, \rho^{\star}_{C}\coloneqq \left(\frac{\phi_{H}}{2M_{C}}\right)^{\frac{\alpha_{H}^{*}}{\alpha_{H}^{*}-1}}.
    \end{equation}
    \begin{theorem}
    \label{thm:poisson-new}
        Fix $\delta>0$, $\varepsilon\in (0,1)$ and choose a constant $C_{0}>\delta+\varepsilon$. Let $H$ be a strictly balanced graph and $\rho\in (0,\rho^{\star}_{C_0})$. For 
        $1 \ll n^{v_{H}}p^{e_{H}} \leq \frac{\rho}{C_0} \cdot \tu{Aut}(H) \cdot(\log n)^{\alpha_{H}^{*}/(\alpha_{H}^{*}-1)}$, we have
            \begin{equation*}
                (1-\varepsilon-o(1))\cdot I(\delta)\ \leq -  \frac{\log \P(\UTH)}{n^{v_{H}}p^{e_{H}}/\tu{Aut}(H)}\ \leq\ (1+\varepsilon+o(1))\cdot I(\delta),
            \end{equation*}
            {where $I$ is as in \eqref{eq:poisson-rate-function}}. 
    \end{theorem}
    
\begin{remark}
For $H=K_3$, the regime $np\asymp\log n$ is conjectured to exhibit a transition from Poisson type behavior to a mixture of Poisson and localized type upper-tail behavior. More precisely, the rate function is expected to arise from an optimization problem balancing the cost of producing an excess of essentially edge disjoint triangles against that of planting a clique (or cliques depending on $\delta$). In case this conjectured picture holds, the constants in Theorem \ref{thm:poisson-new} do not seem to be optimal, and the conclusion of Theorem \ref{thm:poisson-new} should remain valid up to a larger threshold. 
\end{remark}		

It is obvious that Theorem \ref{thm:poisson-new} implies Theorem \ref{thm:poisson}. The definitions of the constants appearing in \eqref{eq: defn of M_C} and \eqref{eq:def of constants} may seem mysterious at the moment. The reader will later notice that these constants emerge naturally from the proof of the upper bound on the factorial moments of the number of unlabelled copies of $H$ in $\Gd$.

Throughout the rest of the section we fix a strictly balanced graph $H$ with minimum degree $2$ and maximum degree $\Delta$. For ease in writing the proofs, {\em only} in this section, we use the nomenclature `copy of $H$' to mean an unlabelled subgraph of $K_n$ isomorphic to $H$. We will use the shorthand $\widetilde N$ to denote the number of (unlabelled) copies of $H$ in $\Gd$. As already mentioned in Section \ref{sec:proof-outline}, the key to the proof of upper tail large deviations speed and rate function \dteal{in the Poisson regime} is tight upper and lower bounds on $\E[\widetilde N^{\underline{t}}]$ for $t \asymp n^{v_H}p^{e_H}$, where we recall $x^{\underline{y}}:= x(x-1)\cdots (x-y+1)$, for $x, y \in \N$. The next two lemmas provide the necessary upper and lower bounds.

\begin{lemma}
            \label{lem:estimate t factorial moment lower bound}
        {Fix some positive constants $\{\beta_i\}_{i \in \set{4}}$. Assume $d\geq 2$ and  $n^{v_{H}}p^{e_{H}}\leq \beta_1 (\log n)^{\beta_2}$. Then, uniformly over all integers $t$  satisfying $1\leq t\leq \beta_3(\log n)^{\beta_4}$}, there exists a constant $c_{H}>0$, depending only on $H$, such that for sufficiently large $n$, 
        \[\E[\wn^{\underline{t}}]\geq (1-n^{-c_{H}})\cdot \left(\frac{n^{v_{H}}p^{e_{H}}\cdot D_{H}}{\tu{Aut}(H)}\right)^{t},\]
        where 
        \begin{equation}
            \label{eq:def of D_H}
            D_{H}\coloneqq \prod_{u\in V(H)}\prod_{i=1}^{\tu{deg}_{H}(u)-1}\left(1-\frac{i}{d}\right).
        \end{equation}
    \end{lemma}
    
    {\begin{lemma}
\label{lem:estimate t factorial moment}
    Let $\beta, C_0>0$ and $\rho\in (0,\rho_{C_0}^{\star})$ (recall \eqref{eq:def of constants}). Assume $n^{v_{H}}p^{e_{H}}\leq \beta(\log n)^{\alpha^*_H/(\alpha^*_H-1)}$. Then, uniformly over all integers $t$ satisfying $1\leq t\leq \min\{C_0\cdot  n^{v_{H}}p^{e_{H}}D_{H}\cdot \tu{Aut}(H)^{-1},\ \rho (\log n)^{\alpha^*_H/(\alpha^*_H-1)}\}$, there exists a constant $\widehat c_{\rho,H}>0$ (depending only on $\rho, H$), such that for sufficiently large enough $n$,
    \[\E[\wn^{\underline{t}}]\leq (1+n^{-\widehat c_{\rho, H}})\cdot \left(\frac{n^{v_{H}}p^{e_{H}}\cdot D_{H}}{\tu{Aut}(H)}\right)^{t}. \]
\end{lemma}

Postponing the proofs of Lemmas \ref{lem:estimate t factorial moment lower bound} and \ref{lem:estimate t factorial moment} for the moment, we now prove Theorem \ref{thm:poisson-new}.

\begin{proof}[Proof of Theorem \ref{thm:poisson-new}]
Fix $\varepsilon >0$ arbitrarily small. We will prove that 
\begin{equation}\label{eq:fac-mom-bd}
\left|\E[\wn^{\underline{t}}] - \E[\wn]^{t}\right| \leq n^{-\frac{1}{2}\bar c_{\rho, H}} \cdot \E[\wn]^{t}, \quad \text{ for all } t \in \set{(\delta-\varepsilon) \E[\widetilde N], (\delta+\varepsilon) \E[\widetilde N]},
\end{equation}
for some constant $\bar c_{\rho, H}>0$ depending only on $\rho$ and $H$, and all large $n$. Since $\varepsilon >0$ is arbitrary and $\delta \mapsto I(\delta)$ is continuous on $(0,\infty)$, the desired conclusion then follows from \cite[Proposition~8.3]{HMS}.

Turning to prove \eqref{eq:fac-mom-bd}, as a preliminarily step, applying Lemmas \ref{lem:estimate t factorial moment lower bound} and \ref{lem:estimate t factorial moment} with $t=1$, we deduce that 
\begin{multline}\label{eq:t-1-bd}
\E[\wn] = (1+o(n^{-\bar c_{\rho, H}})) \left(\frac{n^{v_{H}}p^{e_{H}}\cdot D_{H}}{\tu{Aut}(H)}\right) \\
= (1+o(n^{-\bar c_{\rho, H}})) \cdot (1-O(d^{-1})) \left(\frac{n^{v_{H}}p^{e_{H}}}{\tu{Aut}(H)}\right) = (1+o(1))\left(\frac{n^{v_{H}}p^{e_{H}}}{\tu{Aut}(H)}\right),
\end{multline}
for some constant $\bar c_{\rho, H}$. To obtain the second equality in \eqref{eq:t-1-bd} we have used that any graph $H$ with minimum degree two has $e_H \geq v_H$, which together with the lower bound $n^{v_H} p^{e_H} \gg 1$ implies that $d \gg 1$, and the latter in turn implies $D_{H}= 1-O(d^{-1})$. \dteal{Upon using} the fact that $d \gg 1$ once more yields the rightmost equality.

Next, as a consequence of \eqref{eq:t-1-bd}, any positive integer $t$ satisfying $t\leq (\delta+\varepsilon)\E[\wn]$ also satisfies the hypotheses of Lemmas \ref{lem:estimate t factorial moment lower bound} and \ref{lem:estimate t factorial moment}, with $C_0$ and $\rho$  as specified in our assumptions. Hence, uniformly over all positive integers $t\leq (\delta+\varepsilon) \E[\widetilde N]$, we have
\[ 
\left|\E[\wn^{\underline{t}}]-\left(\frac{n^{v_{H}}p^{e_{H}}\cdot D_{H}}{\tu{Aut}(H)}\right)^{t}\right| \leq n^{-\bar c_{\rho, H}}\cdot \left(\frac{n^{v_{H}}p^{e_{H}}\cdot D_{H}}{\tu{Aut}(H)}\right)^{t} \leq n^{-{\bar c_{\rho, H}}}\cdot \left(1+o(n^{-\bar c_{\rho, H}})\right)^t \E[\widetilde N]^t,
\]
where we also used the first equality in \eqref{eq:t-1-bd}. Since, by our assumption $t \lesssim (\log n)^{O(1)}$ the required bound \eqref{eq:fac-mom-bd} follows from above. This completes the proof. 
\end{proof}
    
Next we turn to the proof of Lemma \ref{lem:estimate t factorial moment lower bound}. We need some more notation. For 
$t\in \N$, define 	
    \begin{equation*}
        \begin{aligned}
            \mathcal{H}(t) \coloneqq \big\{\underline{I}=(I_1, I_2, \ldots, I_t):\ & I_{1},I_{2},\ldots, I_{t}\tu{ are distinct copies of }H\tu{ in }K_n\big\}. 
        \end{aligned}
    \end{equation*}
Notice that $\wn= \sum_{I\in \mathcal{H}(1)}\mathds{1}_{(I\subseteq \Gd)}$. For $\underline{I}=(I_1, \ldots, I_{t})\in \mathcal{H}(t)$, write {\dteal{$G({\underline{I}})$} to denote the subgraph \dteal{spanned by the edges} $\cup_{i=1}^{t}I_{i}$}. We will write $V(\underline I)$ and $E(\underline I)$ to denote the vertex set and the edge set of $G(\underline I)$, respectively. Armed with this notation, obrseve that, for any $t\in \N$,   {$\wn^{\underline{t}} = \sum_{\underline{I}\in \mathcal{H}(t)}\mathds{1}_{(G(\underline{I})\subseteq \Gd)}$.} 
Therefore, to obtain a good lower bound on $\E[\wn^{\underline t}]$ we need to start with the same on $|\mathcal{H}(t)|$ as well. 

Turning to do that, upon using $\log(1-a) \geq -2a$ for $a \in [0,1/2]$ we find that  
\begin{equation}
x^{y}\cdot \exp(-y^{2}/x) \leq x^y \cdot \prod_{i=0}^{y-1} \left(1-\frac{i}{x}\right) \leq x^{\underline{y}}\leq x^y, \qquad \text{ for all } x \geq 2y,
\end{equation}		
Using the above bound, we \dteal{find} that  		
\[
|\mathcal{H}(1)| = n^{\underline{v_{H}}}\cdot (\tu{Aut}(H))^{-1} \geq n^{v_{H}}e^{-v_{H}^{2}/n}\cdot (\tu{Aut}(H))^{-1}.
\]
The lower bound on $|\mathcal{H}(1)|$, for $t\leq |\mathcal{H}(1)|/2$, further yields that
    \begin{equation}
        \label{eq:t-tuple count}
        |\mathcal{H}(t)| = |\mathcal{H}(1)|^{\underline{t}} \geq \left(\frac{n^{v_{H}}}{\tu{Aut}(H)}\right)^{t}\exp\left(-\frac{tv_{H}^{2}}{n}-\frac{t^{2}}{|\mathcal{H}(1)|}\right) = \left(\frac{n^{v_{H}}}{\tu{Aut}(H)}\right)^{t}\cdot e^{-O(\frac{t}{n}+\frac{t^{2}}{n^{v_{H}}})}.
    \end{equation} 		

Equipped with \eqref{eq:t-tuple count} we move to the proof of Lemma \ref{lem:estimate t factorial moment lower bound}.
    \begin{proof}[Proof of Lemma \ref{lem:estimate t factorial moment lower bound}]
        
        
{To obtain the necessary lower bound on the factorial moment, we rely on \eqref{eq:t-tuple count} to \dteal{lower bound} 
the number of $t$-tuples of copies of $H$ that are pairwise vertex disjoint. This counting estimate, together with a uniform lower bound on the probability of the event $\{G(\underline I) \subseteq \Gd\}$ for any such $t$-tuples $\underline I$, completes the argument.} Below we carry out the details.
        
Consider tuples  $\underline{I} = (I_{1},\ldots, I_{t})\in \mathcal{H}(t)$ consisting of copies of $H$ that are not all pairwise vertex disjoint, so that  $V(\underline{I})\leq tv_{H}-1$. 
        In this case, there exists a pair $(I_i, I_j)$ with $V(I_i)\cap V(I_j)\neq \emptyset$. We count such $t$-tuples by first choosing a pair of copies that intersect. There are at most ${t\choose 2}$ choices for $\{i,j\}$, at most $n^{v_{H}}\cdot\tu{Aut}(H)^{-1}$ choices for $I_i$, and at most $v_{H}^{2} n^{v_{H}-1}\cdot (\tu{Aut}(H))^{-1}$ choices for $I_j$ intersecting $I_i$ (the factor $v_H^2$ is due to the choice of the common vertex of $I_i$ and $I_j$ \dteal{from $V(I_i)$ and the choice of the vertex from $V(H)$ to be mapped to that common vertex}). The remaining $t-2$ copies can be chosen in at most $n^{(t-2)v_{H}}\cdot (\tu{Aut}(H))^{-(t-2)}$ ways. 
Subtracting this upper bound for the number of tuples $\underline I$ consisting of copies of $H$ that are not all pairwise vertex disjoint from \eqref{eq:t-tuple count}, and setting 
\[
\widehat{\mathcal{H}}(t):= \{\underline{I}\in \mathcal{H}(t): |V(\underline{I})|=tv_{H}\}, 
\]		
we obtain
        \begin{equation}
            \label{eq:disjoint t copies}
            |\widehat{\mathcal{H}}(t)| \geq \left(\frac{n^{v_{H}}}{\tu{Aut}(H)}\right)^{t}\left(e^{-O(\frac{t}{n}+\frac{t^{2}}{n^{v_{H}}})}-\frac{t^{2}v_{H}^{2}/2}{n}\right)\geq \left(\frac{n^{v_{H}}}{\tu{Aut}(H)}\right)^{t}(1-n^{-c_H}),
        \end{equation}
    for some $c_{H}\in (0,1)$, where in the last step we have used $t\lesssim (\log n)^{O(1)}$.

    Now, fix a $\underline{I} = (I_{1},\ldots, I_{t})\in \widehat{\mathcal{H}}(t)$, 
Define $\widetilde{J}_{1}:= \emptyset$ and $\widetilde{J}_{i} := \cup_{j=1}^{i-1} I_j$, for $i\in \{2,\ldots,t+1\}$. As $I_{i}$'s are pairwise edge disjoint, we have $|E(I_i)|=e_{H}$ and $|E(\widetilde{J}_i)|=(i-1)e_{H}$. We may therefore apply Corollary \ref{cor:almost independence lemma(poisson)} successively to each $I_i$, conditioning on $\{G(\widetilde{J_i})\subseteq \Gd\}$. {Therefore, applying Corollary  \ref{cor:almost independence lemma(poisson)} we get 
        \begin{equation}
            \label{eq:bound on each copy}
            \P(I_{i}\subseteq \Gd| \widetilde{J}_{i}\subseteq \Gd) \geq  p^{e_{H}} \cdot D_{I_{i},\widetilde{J}_i}\cdot f^{-e_H}_{ie_H}.
        \end{equation}
        Moreover, since the copies are vertex disjoint, we observe that 
        $D_{I_i, \widetilde{J}_i}  = D_H$ for all $i \in \set{t}$ \dteal{(recall \eqref{eq:def of D_H})}.  
Further note that $f_{ie_{H}}\leq f_{te_H}$ for all $i\in \set{t}$ (recall \eqref{eq:error bound on independence}). 
Plugging these bounds in \eqref{eq:bound on each copy}  we deduce

        \begin{equation}
            \label{eq:lower bound on vertex disjoint copies}
            \P(G(\underline{I})\subseteq \Gd) = \prod_{i=1}^{t}\P(I_i \subseteq \Gd|\widetilde{J}_i\subseteq \Gd) \geq p^{te_{H}}\cdot D_{H}^{t} \cdot f_{te_{H}}^{-te_{H}}. 
        \end{equation}
        Next, we analyze the asymptotic behavior of the correction term in \eqref{eq:lower bound on vertex disjoint copies}. From \eqref{eq:error bound on independence} we have $ f_{te_{H}}^{te_{H}} = \exp(O(tp+t^{2}/nd))$. Using the upper bounds $n^{v_{H}}p^{e_{H}}, t\lesssim (\log n)^{O(1)}$ yields  
        \begin{equation*}
            tp+\frac{t^{2}}{nd}\ \lesssim\  (n^{-v_{H}/e_{H}}+n^{-1})(\log n)^{O(1)}\ \lesssim\ n^{-\widehat{c}},
        \end{equation*}
    where $\widehat{c}$ is some positive constant depending only on $H$. As for $0\leq a< \dteal{1/2}$, $e^{a}\leq1+2a\leq (1-\dteal{2a})^{-1}$ holds,  we conclude from the above bound that 
        \begin{equation}
            \label{eq:f bound error}
            f_{te_{H}}^{te_{H}} \leq 1+O(n^{-\widehat{c}}) = \left(1-O(n^{-\widehat{c}})\right)^{-1}.
        \end{equation}

        The proof of the lower bound is now completed by noting that $\E[\wn^{\underline{t}}] \geq  \sum_{\underline{I}\in \widehat{\mathcal{H}}(t)} \P(G(\underline{I})\subseteq \Gd)$, and using the estimates \eqref{eq:disjoint t copies}, \eqref{eq:lower bound on vertex disjoint copies},and \eqref{eq:f bound error}.}
    \end{proof}


For the lower bound on $\E[\wn^{\underline{t}}]$ we had focused only on vertex non-overlapping \dteal{(and thus edge non-overlapping)} copies of $H$. 	
To obtain the same asymptotic upper bound on $\E[\wn^{\underline{t}}]$, we must also account for contributions from edge overlapping $t$ many copies of $H$. Since such overlapping copies are no longer \dteal{edge disjoint}, care is needed to deal with them.  A natural route is to partition these overlapping copies into clusters so that within a given cluster copies overlap, while copies across clusters are edge disjoint.  This procedure \dteal{is} formalized as below:
For an integer $s\geq 2$, define
\begin{equation}\label{eq:def-Fs}
\begin{aligned}
\mathfrak{F}_{s} := \{\{I_1,\ldots, I_{s}\}: I_{1}, \ldots, I_{s} & \tu{ are distinct copies of $H$ such that } \forall \emptyset \neq S \subsetneq \set{s}\\&
\exists i\in \set{S}, \tu{ and } j\in \set{s}\setminus S 
\tu{ with } E(I_{i})\cap E(I_{j})\neq \emptyset\}.
\end{aligned}
\end{equation}
For ease in writing, we extend the above definition to $\mathfrak{F}_{1}$, which is naturally identified with $\mathcal{H}(1)$.  It is worth highlighting that the definition \eqref{eq:def-Fs} ensures that any $\underline I \in \mathfrak{F}_s$, for some $s \geq 2$, cannot be decomposed further into $\underline I_1$ and $\underline I_2$ such that $\underline I_1\in \mathfrak{F}_{s_1}$ and $\underline I_1\in \mathfrak{F}_{s_2}$, with $s_1+s_2=s$, such that \dteal{$E(G({\underline I_1})) \cap E(G({\underline I_2}))=\emptyset$}.

Next, for a cluster $A = \{I_{1}, \ldots, I_{s}\}\in \mathfrak{F}_{s}$, we write $G(A)$, for convenience, to  
denote the subgraph spanned by the the union of the edges of $\{I_j\}_{j \in \set{s}}$, and denote $E(A)$ and $V(A)$ to be the edge set and vertex set of $G(A)$, respectively. 
For $\ell \geq 0$ and a sequence $1\leq s_{1}\leq s_{2}\leq \cdots \leq s_{\ell}$, we define the family of edge disjoint clusters by
\begin{equation*}
    \mathfrak{F}_{s_1, s_{2}, \ldots, s_\ell} \coloneqq \{(A_1, \ldots, A_{\ell}): A_{i}\in \mathfrak{F}_{s_{i}}\ \forall i\in \set{\ell}\tu{ and } \forall i,j\in \set{\ell},\ i\neq j,E(A_{i})\cap E(A_{j})=\emptyset\}.
\end{equation*}
Similar to the above, for an $\underline A = (A_1, \ldots, A_{\ell}) \in \mathfrak{F}_{s_1, s_{2}, \ldots, s_\ell}$ we set $G(\underline A):= G(\cup_{i=1}^\ell A_i)$ to be \dteal{the subgraph spanned by the} edge disjoint union of $\{G(A_i)\}_{i \in \set{\ell}}$. 


The key step in establishing the upper bound on $\E[\wn^{\underline t}]$ is to control the contribution of non-trivial clusters. The next lemma provides precisely the required conditional estimate in the presence of other clusters, as long as their union contains at most poly-logarithmically many edges.

\begin{lemma}
    \label{lem:cluster expectation}
    Fix some $C, \beta_1, \beta_{2}>0$. Assume that $p\in (0,1)$, satisfies  $n^{v_{H}}p^{e_{H}}\leq \beta_1(\log n)^{\beta_2}$. Assume $0< \rho <\rho^{\star}_{C}$ (recall \eqref{eq:def of constants}). Uniformly over all $G_{\star}\in \widehat{\Xd}$ (recall \eqref{eq:def-Xd}) satisfying $e(G_{\star}) \leq \beta_1(\log n)^{\beta_2}$, and all integers $2\leq s\leq \rho (\log n)^{\alpha^*_H/(\alpha^*_H-1)}$, we have for \dteal{all} sufficiently large $n$, 
    \begin{equation}
        \label{eq:1st estimate cluster expectation}
        \sum_{A\in \mathfrak{F}_{s}:\ E(A)\cap E(G_{\star})= \emptyset}  \P\left({G(A)} \subseteq \Gd\ \big| G_{\star}\subseteq \Gd \right)\leq C^{-s} \exp\left(-\frac{\theta_{\rho}}{4e_{H}} s^{1/\alpha}\log n \right),
    \end{equation}
    where $\theta_{\rho}\coloneqq\frac{\phi_{H}}{2}-M_{C}\cdot \rho^{(\alpha^*_H-1)/\alpha^*_H}$. 
\end{lemma}
The precise value of the constant in the exponent in \eqref{eq:1st estimate cluster expectation} will not be important for the proof of Lemma \ref{lem:estimate t factorial moment}; any positive constant would suffice. We will prove Lemma \ref{lem:cluster expectation} later, but for now we use this estimate to complete the proof of Lemma \ref{lem:estimate t factorial moment}.   To improve notational clarity, during the rest of the proofs in this section we will use $\alpha^*$ instead of $\alpha_{H}^{*}$. 

    \begin{proof}[Proof of Lemma \ref{lem:estimate t factorial moment}]
Consider an arbitrary sequence $\underline{I}=(I_1,\ldots, I_{t})\in \mathcal{H}(t)$. The copies of $H$ in $\underline{I}$ can be partitioned into groups based on their connectivity: some copies are isolated (i.e.~they do not share any edges with other copies of $H$ present in $\underline I$), while the rest form clusters of specific sizes. Hence, for any $\underline I \in \mathcal{H}(t)$, there exist integers $\ell,k,s_1, s_2,\ldots, s_\ell\geq 0$ satisfying $\ell\leq \frac{1}{2}(t-k)$, $k+s_1+\cdots+ s_\ell=t$ and $2\leq s_1\leq \cdots\leq s_\ell$ such that we can partition $\underline{I}$ as 
        \begin{equation*}
            \bigg(\underbrace{\{I_{i_1}\}, \{I_{i_2}\}, \ldots, \{I_{i_k}\}}_{k\tu{ isolated copies}}, \underbrace{\{I_{i_{k+1}}, \ldots, I_{i_{k+s_1}}\}}_{\tu{cluster of size }s_1}, \ldots, \underbrace{\{I_{i_{k+s_1+\cdots+s_{\ell-1}+1}}, \ldots, I_{i_{k+s_1+\cdots+s_{\ell}}}\}}_{\tu{cluster of size }s_\ell}\bigg) \in \mathfrak{F}_{1,1,\ldots, 1, s_{1}, \ldots, s_{\ell}},
        \end{equation*}
        where $i_{1}\leq i_{2}\leq \cdots \leq i_{k}$ and  there are exactly $k$ many $1$'s in the subscript of $\mathfrak{F}$. 
        
        Conversely, with $\ell,k, s_{1}, \ldots, s_{\ell}$  satisfying the above conditions, let 
        \[
\underline{A}=(A_{1}, A_{2}, \ldots,A_k, A_{k+1},\ldots, A_{k+\ell})\in \mathfrak{F}_{1, \ldots, 1,s_{1}, \ldots, s_\ell}. 
\]
We claim that the number of sequences $\underline{I}\in \mathcal{H}(t)$ that can be partitioned to yield this specific $\underline{A}$ is at most $t^{s_{1}+\cdots+s_{\ell}}$. This is because each element in $A_{i}$, for  $i> k$, has at most $t$ choices for its position, and there are $s_{1}+\cdots+s_\ell$ many such elements in $A_{k+1}\cup \cdots \cup A_{\ell}$. The positions of the rest of $k$ isolated elements of $A_{1}\cup \cdots \cup A_{k}$ are already fixed by our choice of the partitioning process. This yields the following bound:
        \begin{equation}
            \label{eq:rewrite t factorial moment}
\E[\wn^{\underline t}] =	\sum_{\underline{I}\in \mathcal{H}(t)} \P(G(\underline{I}) \subseteq\Gd)\ \leq\ \sum_{k=0}^{t} \sum_{\ell=0}^{\lfloor\frac{t-k}{2}\rfloor} \sum_{\stackrel{s_{1}+\ldots+s_{\ell}=t-k}{2\leq s_{1}\leq s_{2}\leq \cdots \leq s_{\ell}}} \sum_{\stackrel{(A_{1},\ldots,A_k, \ldots, A_{k+\ell})}{\in \mathfrak{F}_{1,\ldots, 1,s_{1}, \ldots, s_{\ell}}}} \P(G(\cup_{i=1}^{k+\ell}A_i) \subseteq \Gd) \cdot t^{\sum_{i=1}^{\ell}s_i}.
        \end{equation}  
        
        An upper bound on the RHS of  \eqref{eq:rewrite t factorial moment} will be obtained in stages. 
First, consider the contribution from the isolated copies. For any $(A_{1}, \ldots, A_{k})\in \mathfrak{F}_{1,\ldots,1}$, each $A_{i}$ is a copy of $H$ and $\{A_i\}_{i\in \set{k}}$ are pairwise edge disjoint. {Applying the upper bound in Corollary \ref{cor:almost independence lemma(poisson)} iteratively $k$ times, we get}
        \begin{equation*}
            {\P(G(\cup_{i=1}^k A_i)  \subseteq \Gd) \leq p^{ke_{H}}\cdot D_{H}^{k}\cdot f_{ke_{H}}^{ke_{H}},}
        \end{equation*}
        where we have used that for each $i\in \set{k}$ and any $G_\star \in \widehat{\Xd}$, we have $D_{A_i,G_\star}\leq  D_H$. The number of elements in $\mathfrak{F}_{1,\ldots, 1}$ (with $k$ many $1$'s in the subscript) is at most $|\mathcal{H}(1)|^{k}\leq n^{kv_{H}}\tu{Aut}(H)^{-k}$. Hence,
\begin{equation} \label{eq:bound on isolated copies}	
\sum_{(A_1, A_2, \ldots, A_k) \in \mathfrak{F}_{1,\ldots,1}} \P(G(\cup_{i=1}^k A_i)  \subseteq \Gd) \leq \left(\frac{n^{v_H}p^{e_H}D_H}{\tu{Aut}(H)}\right)^k \cdot {f_{ke_H}^{ke_H}}.	
\end{equation}	
        
        Now fix a tuple $(A_{1}, \ldots, A_k) \in \mathfrak{F}_{1, \ldots, 1}$. We recall that for any $(A_{k+1},\ldots, A_{k+\ell})\in \mathfrak{F}_{s_1, \ldots, s_\ell}$, $\{G(A_{k+j})\}_{j\in \set{\ell}}$ are pairwise edge disjoint. Notice that  
\begin{equation*}
\begin{aligned}
\P\left(G(\cup_{j=1}^{\ell}A_{k+j})\subseteq \Gd\ \mid  G(\cup_{i=1}^{k}A_{i})\subseteq \Gd\right) & = \prod_{j=1}^\ell  \P\left(G(A_{k+j})\subseteq \Gd\ \mid  G(\cup_{i=1}^{k+j-1}A_{i})\subseteq \Gd\right)\\
& \leq \prod_{j=1}^\ell \left(\sup_{\substack{ G_{\star}\in \widehat{\Xd},e(G_{\star})\leq te_{H}, \\ E(G_\star) \cap E(A_{k+j})=\emptyset}} \P(G(A_{k+j}) \subseteq \Gd|G_{\star}\subseteq \Gd)\right).
\end{aligned}
\end{equation*}
Therefore, for any $(A_1,A_2, \ldots, A_k) \in \mathfrak{F}_{1,1,\ldots, 1}$ and $\{s_i\}_{i \in \set{\ell}}$, with $2\leq s_1\leq\cdots\leq s_\ell\leq t \leq \rho (\log n)^{\alpha^*/(\alpha^*-1)}$, upon applying Lemma \ref{lem:cluster expectation} with $C=C_0$ we obtain 
        \begin{equation}
            \label{eq:bound on multiple clusters}
            \begin{aligned}
            & \sum_{\stackrel{A_{k+1}, \ldots, A_{k+\ell}\in \mathfrak{F}_{s_1, \ldots, s_\ell}}{E(A_{k+j})\cap E(\cup_{i=1}^{k}A_i)=\emptyset,\ \forall j\in \set{\ell}}}  \P\left(G(\cup_{j=1}^{\ell}A_{k+j})\subseteq \Gd\ \mid  G(\cup_{i=1}^{k}A_{i})\subseteq \Gd\right)\\ 
\leq & \hskip10pt \prod_{j=1}^{\ell}\ \bigg(\sup_{\substack{ G_{\star}\in \widehat{\Xd},\\ e(G_{\star})\leq te_{H}}}\sum_{\substack{A\in \mathfrak{F}_{s_{j}}\\ E(A)\cap E(G_{\star})=\emptyset}}\P(A\subseteq \Gd|G_{\star}\subseteq \Gd)\bigg) \leq C_{0}^{-\sum_{j=1}^{\ell}s_j}\exp\left(-c_{\rho, H}\cdot\log n \sum_{j=1}^{\ell}s_{j}^{1/\alpha^*}\right),
            \end{aligned}
        \end{equation}
for some constant $c_{\rho, H} >0$, depedning only on $\rho$ and $H$.			


On the other hand, since $\alpha^*>1$ (as $v_H >2$ and $\alpha^*_H \geq v_H/2$), 
and $s_{i}\geq 2$ for all $i\in \set{\ell}$, using that the map $(x,y) \mapsto (x+y)^{1/\alpha^*} - x^{1/\alpha^*} - y^{1/\alpha^*}$ is strictly decreasing on $(0,\infty)^2$ we have 
\begin{equation}\label{eq:varkappa}
\sum_{i=1}^{\ell}s_{i}^{1/\alpha^*}- \left(\sum_{i=1}^{\ell} s_{i}\right)^{1/\alpha^*}\geq \varkappa(\ell-1), \quad \text{ where } \varkappa :=2\cdot 2^{1/\alpha}-4^{1/\alpha}>0.
\end{equation}
Further for a fixed $k$ and an $\ell \geq 1$, the number of ordered tuples $(s_1, \ldots, s_\ell)$ summing to $t-k$, such that $t \lesssim (\log n)^{O(1)}$, is at most ${t-k-\ell-1\choose \ell-1}\leq t^{\ell-1}=\exp((\ell-1)\cdot o( \log n))$. Hence, for any $k$ such that $k \leq t-2$, 
        \begin{equation}
            \label{eq:bound for non-zero cluster}
            \sum_{\substack{s_{1}+\cdots+s_{\ell}=t-k\\ 2\leq s_{1}\leq \cdots\leq s_{\ell}}}\exp\left(-c_{\rho, H}\log n\sum_{j=1}^{\ell}s_{j}^{1/\alpha^*}\right)
        \leq n^{-c_{\rho, H}\cdot(t-k)^{1/\alpha^*}}\cdot n^{-(\varkappa c_{\rho, H}-o(1))(\ell-1)}\leq	n^{-c_{\rho, H}\cdot2^{1/\alpha^*}}.
        \end{equation}
                
Plugging the bounds \eqref{eq:bound on isolated copies},\eqref{eq:bound on multiple clusters}, and \eqref{eq:bound for non-zero cluster} in \eqref{eq:rewrite t factorial moment}, recalling that $\sum_{j=1}^{\ell}s_j=t-k$,  and noting $k$ never equals $t-1$, we deduce
        \begin{equation*}
            \begin{aligned} 
\E[\wn^{\underline t}] &\leq\sum_{k=0}^{t}\sum_{\ell= 0}^{\lfloor\frac{t-k}{2}\rfloor} \sum_{\stackrel{s_{1}+\ldots+s_{\ell}=t-k}{2\leq s_{1}\leq \cdots \leq s_{\ell}}}\left(\frac{n^{v_{H}}p^{e_{H}} D_{H}}{\tu{Aut}(H)}\right)^{k} \cdot {f_{ke_{H}}^{ke_{H}}}\cdot  \left(\frac{t}{C_0}\right)^{t-k}\exp\left(-c_{\rho, H}\log n\sum_{j=1}^{\ell}s_j^{1/\alpha^*}\right)\\[-5pt]
                &\leq \left(\frac{n^{v_{H}}p^{e_{H}} D_{H}}{\tu{Aut}(H)}\right)^{t}{f_{te_{H}}^{te_{H}}}\left(1+\sum_{k=0}^{t-2}\sum_{\ell=1}^{\lfloor\frac{t-k}{2}\rfloor}\sum_{\substack{s_{1}+\cdots + s_{\ell}=t-k\\ 2\leq s_{1}\leq \cdots\leq s_{l}}} \exp\big(-c_{\rho, H}\log n \sum_{j=1}^{\ell}s_{j}^{1/\alpha^*}\big)\right)\\
                & \leq \left(\frac{n^{v_{H}}p^{e_{H}} D_{H}}{\tu{Aut}(H)}\right)^{t} \cdot{(1+O(n^{-\widehat{c}}))}  \cdot \left(1+t^2 n^{-c_{\rho, H}\cdot2^{1/\alpha^*}}\right), 
            \end{aligned}
        \end{equation*}
        where in the second step we use the bounds $t\leq C_{0}\cdot n^{v_{H}}p^{e_{H}}D_{H}\cdot \tu{Aut}(H)^{-1}$ and  {$f_{ke_{H}}^{ke_{H}}\leq f_{te_{H}}^{te_{H}}$, and in the last step we used that $ f_{te_{H}}^{te_{H}} \leq (1+O(n^{-\widehat{c}}))$ obtained in \eqref{eq:f bound error} (since we have $n^{v_{H}}p^{e_{H}},t\lesssim (\log n)^{O(1)}$)}. {Further  using $t \lesssim (\log n)^{\alpha^*/(\alpha^*-1)}$, the desired upper bound on $\E[\wn^{\underline t}]$ now follows from above. This completes the proof.}
\end{proof}			
%
%
%
%
%
        
We now return to the proof of Lemma \ref{lem:cluster expectation}, which was deferred earlier.
\begin{proof}[Proof of Lemma \ref{lem:cluster expectation}]
 We begin by fixing an $A\in \mathfrak{F}_{s}$, and setting $m\coloneqq e(G(A))$,  $k\coloneqq|V(G(A))|$. Since $A$ consists of $s\geq 2$ copies of $H$, we have $e_{H}+1\leq m\leq e_{H}s$ and $v_{H}\leq k\leq v_{H}s$. In particular, the assumed upper bound on $s$ implies $m,k\lesssim (\log n)^{O(1)}$. Applying Corollary \ref{cor:almost independence lemma(poisson)} with $I=A$, we obtain
\begin{equation}
    \label{eq:cluster prob bound}
    \P(G(A)\subseteq \Gd\ \big|\ G_{\star}\subseteq \Gd)\leq p^{m}\cdot D_{A, G_{\star}}\cdot f_{m+e(G_{\star})}^{m}\leq (1+o(1))\cdot p^{m},
\end{equation} 
where we used $D_{A, G_{\star}}\leq 1$, and since, $m+e(G_{\star}) \lesssim  (\log n)^{O(1)}$, {using the upper bound on $p$, we have $f_{m+e(G_{\star})}^{m}= \exp(O(p+n^{-1}) (\log n)^{O(1)})=1+o(1)$. }

Next, let $N(H, k,m)$ denote the maximum number of copies of $H$ in a graph with $k$ vertices and $m$ edges. For fixed $k$ and $m$, we may bound the number of possible $A\in \mathfrak{F}_{s}$ by first choosing the $k$ vertices, then the $m$ edges, and finally the $s$ copies of $H$. Thus,
\begin{equation}
    \label{eq:number of clusters}
    \left|\{A\in \mathfrak{F}_s:\ e(G(A)) = m,\ |V(G(A))| = k\}\right|\leq n^{k}\cdot {k^{2}\choose m} {N(H, k, m)\choose s}\leq n^{k}\cdot {k^{2}\choose m} {(2m)^{\alpha^*}\choose s},
\end{equation}
where the last inequality follows from applying \cite[Theorem~5.7]{HMS} to upper bound $N(H,k,m)$. Notice that it also follows from above that unless $m\geq \frac{1}{2}s^{1/\alpha^*}$ the LHS \eqref{eq:number of clusters} zero. We claim that, uniformly for $2\leq s\leq t$, $\max\{\frac{1}{2}s^{1/\alpha^*}, e_{H}+1\}\leq m\leq e_{H}s$, and $v_{H}\leq k\leq v_{H}s$, 
\begin{equation}\label{eq:log-psi-bd}
\log \Psi_{s,k,m} \leq -(1-o(1))(\theta_{\rho}/2e_{H})s^{1/\alpha^*}\log n, 
\end{equation}
where
\begin{equation}\label{eq:def-psi}
    \Psi_{s,k,m}\coloneqq C^{s}n^{k}p^{m}{k^{2}\choose m} {(2m)^{\alpha^*}\choose s}.
\end{equation}
Equipped with the claim \eqref{eq:log-psi-bd}, upon using a union bound  and combining \eqref{eq:cluster prob bound} and \eqref{eq:number of clusters} it therefore follows that
\begin{multline*}
\sum_{A\in \mathfrak{F}_s: E(A)\cap E(G_{\star})=\emptyset}  C^s \P({G(A)}\subseteq \Gd|G_{\star}\subseteq \Gd)   \leq \sum_{\substack{v_{H}\leq k\leq v_{H}s\\ e_{H}+1\leq m\leq e_{H}s}}2\cdot\Psi_{s,k,m}\\
  \leq 2v_{H}e_{H}s^{2}\cdot\exp\left(-(1-o(1))\frac{\theta_{\rho}}{2e_{H}}s^{1/\alpha^*}\log n\right) \leq \exp\left(-\frac{\theta_{\rho}}{4e_{H}}s^{1/\alpha^*}\log n\right),
\end{multline*}
for all large $n$, where in the last step we used that $s \lesssim (\log n)^{O(1)}$. This proves the desired conclusion.	

Therefore, it remains to establish \eqref{eq:log-psi-bd}. \dteal{It} requires the following bound (recall \eqref{eq: defn of M_C} and \eqref{eq:def of constants}): For any $A=\{I_1,I_2, \ldots, I_s\} \in \mathfrak{F}_s$ with $s \geq 1$ and $m = e(G(A))$ and $k=|V(G(A))|$ we have 	
\begin{equation}
    \label{eq: SB inequality}
    m-e_{H}\geq \lambda r_{H} (k-v_{H}).
\end{equation} 
We prove \eqref{eq: SB inequality} by an induction on $s$. The bound \eqref{eq: SB inequality} trivially holds when $s=1$. Fix $s \geq 2$ and an $A \in \mathfrak{F}_s$. We claim that there exists some $i \in \set{s}$ such that $A':=A\setminus \{I_i\} \in \mathfrak{F}_{s-1}$. To see this, define $\widehat A:= \{I_1, I_2, \ldots, I_{s-1}\}$. If $\widehat A \in \mathfrak{F}_{s-1}$ then there is nothing to be proved. So, assume the contrary. In this case, we can decompose $\widehat A$ into edge disjoint clusters and write $\widehat A=(\widehat A_1, \widehat A_2, \ldots, \widehat A_t) \in \mathfrak{F}_{\widehat s_1, \widehat s_2, \ldots, \widehat s_t}$ for some integer $t \geq 2$ and some collection of positive integers $\{\widehat s_j\}_{j \in \set{t}}$. As $A \in \mathfrak{F}_s$ it follows that there exists some collection of indices $\{a_j\}_{j \in \set{t}}\subseteq \set{s-1}$ such that
\begin{equation}\label{eq:a-j}
I_{a_j} \in \widehat A_j \qquad \text{ and } \qquad E(I_{a_j})\cap E(I_s) \neq \emptyset, \quad \text{ for all } j \in \set{t}. 
\end{equation}
Set $i=a_1$. With this choice, using \eqref{eq:a-j}, as $\{\widehat A_b\}_{b \in \set{t}}$ is a collection of clusters, we observe that $A'$ is indeed is a cluster of size $s-1$.

Now we return to the proof of \eqref{eq: SB inequality}. We note from above that $E(A') \cap E(I_i) \neq \emptyset$. Hence, there exists some $\emptyset \neq {J}\subseteq {H}$ such that the subgraph spanned by $E(I_i) \cap E(A')$ is isomorphic to $J$. As $e(G(A))=m$ and $|V(G(A))|=k$, the last observation implies that 
\[
m -e_{H}= e(A')-e_{{J}} \qquad \text{ and } \qquad  k-v_{H}= |V(G(A'))|-v_{{J}}.
\]
Applying the induction hypothesis to ${A}'$ we deduce from above 
\[
m-e_{H}=e(A')-e_{H}+e_{H}-e_{{J}}\geq  \lambda r_{H}(|V(G(A'))|-v_{H})+\lambda r_{H}(v_{H}-v_{\tu{J}}) = \lambda r_{H}(k-v_{H}),
\]
where in the penultimate step we used
\[
e_{H}-e_{{J}}\geq \lambda r_{H}(v_{H}-v_{{J}}),
\]
which is a consequence of $H$ being strictly balanced (see \cite[Section~5]{BK25} for a proof). This proves \eqref{eq: SB inequality}.  	




We return to the proof of \eqref{eq:log-psi-bd}. Applying 
\eqref{eq: SB inequality}, upon using the assumption $n^{v_{H}}p^{e_{H}}\leq \beta_{1}(\log n)^{\beta_2}$, and noting that $np^{\lambda r_H}= (n p^{e_H/v_H})^{r_H} \cdot n^{1-r_H}$ (recall \eqref{eq:def of constants}) we obtain 
\begin{equation}\label{eq:n-k-p-m}
n^{k}p^{m} = n^{v_{H}}p^{e_{H}} \cdot n^{k-v_{H}}p^{m-e_{H}}\leq n^{v_{H}}p^{e_{H}} \cdot (np^{\lambda r_{H}})^{(m-e_{H})/\lambda r_{H}}\leq n^{-(m-e_{H})\phi_{H}}(\beta_{1}(\log n)^{\beta_2})^{m/e_{H}}.
\end{equation}
Moreover, as $k\leq v_{H}s$ and $s$ is at most poly-logarithmic, we obtain 
\begin{equation}\label{eq:log-k-choose-m}
\log {k^{2}\choose m}\leq 2m\log k = O(m\log \log n). 
\end{equation}
Using Stirling's approximation, as $m \lesssim s$, we also have 
\begin{equation}\label{eq:log-stirling}
\log {(2m)^{\alpha^*}\choose s } = s\cdot  f\left(\frac{(2m)^{\alpha^*}}{s}\right)+O(\log s) = s\cdot  f\left(\frac{(2m)^{\alpha^*}}{s}\right)+O(\log \log n),
\end{equation}
where $f(\cdot)$ is as in \eqref{eq: defn of M_C}. Combining \eqref{eq:n-k-p-m}, \eqref{eq:log-k-choose-m}, and \eqref{eq:log-stirling}, and recalling \eqref{eq:def-psi} we therefore have	
\begin{equation}
    \label{eq: Poisson exponent asymp} 
    \begin{aligned}
    \log \Psi_{s,k,m} & = \phi_{H}e_{H}\log n - \phi_{H}m\log n+ s\cdot f\left(\frac{(2m)^{\alpha^*}}{s}\right)+s\log C+O(m\log\log n )\\
    & \leq \phi_{H}e_{H}\log n +2m\log n \left(M_{C}\frac{s^{1-1/\alpha^*}}{\log n} - \frac{\phi_{H}}{2}\right)+O(m\log \log n)\\
    & \leq (\phi_{H}e_{H}-(\theta_{\rho}-o(1))2m)\log n \leq (\phi_{H}e_{H}-(\theta_{\rho}-o(1))s^{1/\alpha^*})\log n.
    \end{aligned}
\end{equation}
where in the second step we use \eqref{eq: defn of M_C}), in the penultimate step we used $s\leq \rho(\log n)^{\alpha^*/(\alpha^*-1)}$ and the definition of $\theta_{\rho}$, and in the last step we used $\theta_\rho >0$ and $2m\geq s^{1/\alpha^*}$. 

To complete the proof, we set $s_{0}\coloneqq \lceil (2\phi_{H}e_{H}/\theta_{\rho})^{\alpha^*}\rceil$ and treat the regimes $s\geq s_{0}$ and $s\leq s_{0}-1$ separately. If $s\geq s_{0}$, then $\phi_{H}e_{H}\leq \frac{1}{2}\theta_{\rho}s^{1/\alpha^*}$ and therefore the desired upper bound is immediate from \eqref{eq: Poisson exponent asymp}.


On the other hand, for $2\leq s\leq s_{0}$, we have $m,s= O(1)$. Since $m\geq e_{H}+1$ and $\phi_{H}>0$, by \eqref{eq: Poisson exponent asymp} (notice the first equality there) 
\[
\log \Psi_{s,k,m}= -\phi_{H}(m-e_{H})\log n + O(\log 
\log n)\leq -(1-o(1))\phi_{H}\log n.
\]
As $s \leq s_0$, we have $\frac{\theta_{\rho}}{2e_{H}}s^{1/\alpha^*}\leq \phi_{H}$. Therefore, substitution this bound in the above inequality gives the claim.  
\end{proof}

        

We end the section with the proof of Theorem \ref{thm:conditional-poisson}.

    \begin{proof}[Proof of Theorem \ref{thm:conditional-poisson}]
        Set $\beta \coloneqq \lceil \chi n^{v_{H}}p^{e_{H}}\tu{Aut}(H)^{-1} \rceil$.
        Let $Z_{H}$ denote the number of edge non-disjoint (unlabelled) copies in $\Gd$. We shall prove that there exists a constant $c_{H}>0$ such that
        \begin{equation}
            \label{eq: non-isolated copies prob}
            \P(Z_{H}\geq \beta)\leq \exp(-c_{H}(\chi n^{v_{H}}p^{e_{H}})^{1/\alpha^*}\cdot \log n).
        \end{equation}
    By Theorem \ref{thm:poisson}, we have $-\log \P(\UTH) \asymp n^{v_{H}}p^{e_{H}}$. Under the assumption $n^{v_{H}}p^{e_{H}}\ll (\log n)^{\alpha^*/(\alpha^*-1)}$, \eqref{eq: non-isolated copies prob} implies 
    \[
    \P(Z_{H}\geq \beta\ |\ \UTH )\leq \exp(-c_{H}^{\prime}(\chi n^{v_{H}}p^{e_{H}})^{1/\alpha^*}\log n),
    \]
    for some constant $c_{H}^{\prime}>0$. Since, 
    $N_{\tu{disj}}(H, \Gd) = (\wn - Z_{H})\cdot \tu{Aut}(H)$, the desired result follows from above.

    So, it remains to prove \eqref{eq: non-isolated copies prob}. A key to \eqref{eq: non-isolated copies prob} is the following claim:  
    \begin{equation}\label{eq:claim-conditional}
    \{Z_H \geq \beta\} \subseteq \bigcup_{\ell \in \set{\beta}}\bigcup_{\substack{s_1, s_2, \ldots, s_\ell\\ 2 \leq \min_i s_i \leq \max_i s_i \leq \beta\\ \beta \leq \sum_{i=1}^{\ell}s_{i}\leq 2\beta}}\left\{\exists \underline A=(A_1,A_2,\ldots, A_\ell) \in \mathfrak{F}_{s_1, \ldots, s_\ell} \tu{ such that } G(\cup_{i=1}^{\ell}A_{i})\subseteq \Gd\right\}. 
    \end{equation}


Before proving the claim \eqref{eq:claim-conditional}, we make the following observation which will be essential: Let $\underline I = \{I_1, I_2, \ldots, I_t\}$ be a cluster of size $t$, for some $t \in \N$. Fix $j \in \set{t}$ and let $\{I_{t+i}\}_{i \in \set{t'}}$ be a collection of distinct copies of $H$ so that $E(I_j) \cap E(I_{t+i}) \neq \emptyset$ for all $i \in \set{t'}$, for some $t' \in \N$. Then  $\underline I' = \{I_1, I_2, \ldots, I_{t+t'}\}$ is also a cluster.

To prove that $\underline I' $ is a cluster of size $(t+t')$, we need to show that give any $\emptyset \neq T \subsetneq \set{t+t'}$, there must exist $i \in T$ and $i' \in \set{t+t'} \setminus T$ so that $E(I_i) \cap E(I_{i'})\neq \emptyset$. To prove this, we need to consider a couple of subcases. First, let $T$ be such that $\min\{|T \cap \set{t}|, |\set{t}\setminus T|\} \geq 0$. Then, as $\underline I$ is a cluster, setting $T':= T \cap \set{t} \subsetneq \set{t}$ we see that there indeed exist some $i \in T' \subseteq T$ and $i' \in \set{t} \setminus T' \subseteq \set{t+t'} \setminus T$ so that $E(I_i) \cap E(I_{i'})\neq \emptyset$. In the other case, where either $T \supseteq \set{t}$ or $T \subseteq \set{t+t'}\setminus \set{t}$, the desired conclusion follows from the assumption that $I_j$ shares at least an edge with $I_{t+i}$ for every $i \in \set{t'}$. 

     
Now we return to the proof of \eqref{eq:claim-conditional}. 
We will prove \eqref{eq:claim-conditional} by providing an algorithm to construct an $\underline A$ with the required properties. Start with any edge non-disjoint unlabelled copy $I_1$ of $H$ in $\Gd$. Let $\{I_a\}_{a=2}^{s}$ be the collection of all distinct copies of $H$ in $\Gd$ that share at least one edge with $I_1$. Since $I_1$ is an edge non-disjoint copy, we have $s \geq 2$. As $E(I_a) \cap E(I_1)\neq \emptyset$ for all $a \in \{2, 3, \ldots, s\}$, it follows that, for any $\widetilde s \leq s$, the tuple $(I_1, I_2, \ldots, I_{\widetilde s})$ is a cluster. If $s \geq \beta$ then we set $\ell=1$ and $\underline A= A_1=\{I_1, I_2, \ldots, I_\beta\}$. Therefore, the claim \eqref{eq:claim-conditional} holds in this case with $\ell=1$ and $s_1=\beta$. 
     
So, let us assume $s \leq \beta -1$. On the event $\{Z_H \geq \beta\}$ let us look for edge non-disjoint copies of $I_2$ that are distinct from $\{I_a\}_{a \in \set{s}}$. Let $\{I_{s+a}\}_{a \in \set{s'}}$ be the collection of such copies. If $s+s' \geq \beta$ then we again set $\underline A= A_1 = \{I_1,I_2, \ldots, I_\beta\}$. By the observation made above, $A_1$ is a cluster of size $\beta$. On the other hand, if $s+s' \leq \beta -1$ then we note that $\{I_1,I_2, \ldots, I_{s+s'}\}$ is a cluster of size $(s+s')$ (again by the observation made above). In this case, we move to $I_3$ and repeat the same process. We continue this process until the cluster thus formed is of size $\beta$, or is of size $s_1 \leq \beta-1$ and all other copies of $H$ are edge disjoint from the current cluster. 

In the latter case, as $Z_H \geq \beta$ there must exist at least one edge non-disjoint copy of $H$, say $I'$. We repeat the entire process starting with $I'$.   This process of forming an edge disjoint clusters $\underline A \in \mathfrak{F}_{s_1, s_2, \ldots, s_\ell}$ is continued until $s_1+\cdots+s_\ell \geq \beta$ for some $\ell \in \N$ (the lower bound on the sum if guranteed on the event $\{Z_H \geq \beta\}$). 

Moreover, by the discussion above, we also have $2 \leq \min_i s_i \leq \max_i s_i \leq \beta$. By the stopping rule just mentioned, we also have $s_1+\cdots+s_{\ell-1} \leq \beta -1$. As $s_\ell \leq \beta$, this also gives the upper bound $s_1+\cdots+s_\ell \leq 2\beta$. This upper bound together with the lower bound on $\min_i s_i$ further yields the upper bound  on $\ell$. This finally proves the claim \eqref{eq:claim-conditional}. 
     
%
%
%
    
    
\dteal{Now we are nearly done. Equipped with \eqref{eq:claim-conditional} and applying a union bound,} we derive 
    \begin{equation}
        \label{eq:cluster expansion}
        \begin{aligned}
        \P(Z_{H}\geq \beta)& \leq \sum_{k=\beta}^{2\beta}\sum_{\ell=1}^{\beta}\sum_{\substack{s_1, \ldots, s_\ell \geq2, \\ s_{1}+\cdots+s_\ell =k}} \ \sum_{\substack{A_{1},\ldots, A_\ell\\ \ \in \mathfrak{F}_{s_1, \ldots, s_\ell}}}\P(G(\cup_{i=1}^{\ell}A_{i})\subseteq \Gd)\\
        & \leq \sum_{k=\beta}^{2\beta}\sum_{\ell= 1}^{\beta}\sum_{\substack{s_1,\ldots, s_\ell\geq 2 \\ s_{1}+\cdots+s_\ell=k}}\exp\left(-2c_{H}\log n\sum_{i=1}^{\ell}s_{i}^{1/\alpha^\star}\right)\\
& \leq \sum_{k=\beta}^{2\beta} \sum_{\ell=1}^{\beta} \exp\left(-2c_{H}k^{1/\alpha}\log n-(\ell-1)\left(\varkappa -\frac{\log k}{\log n}\right)\log n \right) \\
& \leq \beta^{2}\cdot \exp(-2c_{H}\beta^{1/\alpha} \log n) \leq \exp(-c_{H}  (\chi n^{v_{H}}p^{e_{H}})^{1/\alpha} \log n),
        \end{aligned}
    \end{equation}
for all large $n$ and some positive constant $c_H >0$, depending only on $H$, where in the second step using the fact that $e(\cup_{i=1}^\ell A_i) \lesssim n^{v_H}p^{e_H} \ll (\log n)^{\alpha^*/(\alpha^*-1)}$ we argued similarly as in \eqref{eq:bound on multiple clusters}, in the second last step we used \eqref{eq:bound for non-zero cluster} and argued similarly as in the paragraph following \eqref{eq:bound for non-zero cluster}, 
in the penultimate step we used $\log k  \lesssim \log\beta \ll\log n$, $k \geq \beta$, and $\ell\geq 1$, 
and in the last step we used the definition of $\beta$. This completes the proof. 
\end{proof}
%

\section{Localized regime: the lower bound}
\label{sec:lower bound}

\dteal{In this section we prove the lower bound on the upper tail large deviation event, suitable for the localized regime. The following is the main result of this section. Recall \eqref{eq:rate function defn}.} 
\begin{proposition} 
\label{prop:Lower bound}
Fix $\delta>0$.
    Let $\Delta\geq 2$ be an integer and $H$ a fixed $\Delta$-regular  graph. Assume $p=p_{n}\in [0,1]$ satisfy $p\ll1$ and $np^{\Delta/2}\gg1$. Then we have
    \begin{equation}
        \label{eq:lower bound}
        \liminf_{n\rightarrow \infty}\frac{\log\P\left(N(H, \Gd)\geq (1+\delta)n^{v_{H}}p^{e_{H}}\right)}{n^{2}p^{\Delta}\log(1/p)} \geq \dteal{-\Psi_H(\delta)}.
\end{equation}
\end{proposition}


\dteal{To prove Proposition \ref{prop:Lower bound} we plant a subgraph $G_\star$ inside $\Gd$ with a small number of edges that possesses 
the excess copies of $H$, and show that  
conditional on planting $G_\star$, the remainder of $\mathbb G_n^d$ still contains approximately the typical number of copies of $H$ that avoid the planted edges with probability close to one, and that the probability of the event $G_\star \subseteq \Gd$ is essentially the desired lower bound of Proposition \ref{prop:Lower bound}.}

To implement this strategy, \dteal{as just mentioned, we} need to control the \dteal{probability of the existence of typically many} 
\dteal{copies of $H$ in the} part of $\mathbb G_n^d$ outside the planted graph. 
The following lemma gives a more \dteal{general 
statement}, \dteal{applicable for irregular $H$}: it applies whenever $p\gg n^{-1/m_H}$, where $m_H:=\max\{e_{J}/v_{J}: J\subseteq H, e_J>0\}$, denotes the maximum density of $H$. 

\dteal{For ease in writing, we introduce the following notation: For any $G\in \widehat{\Xd}$ (recall \eqref{eq:def-Xd}), we let $N(H, \Gd, \overline{G})$ to be the number of labelled copies of $H$ in $\Gd$ that does not use any edge from $G$.}

\begin{lemma}
\label{lem:minimum subgraph count}
Assume $H$ is a fixed graph with maximum degree $\Delta$ and $p\gg n^{-1/m_{H}}$.
Fix $\varepsilon\in (0,1)$, and let $G\in \widehat{\Xd}$ with $e(G)=o(nd)$. 
Then, \dteal{for all sufficiently large $n$}, 
\begin{equation*}
    \P\left(N(H, \Gd, \overline{G})\geq (1-\varepsilon)n^{v_{H}}p^{e_{H}}\ \big|\ G\subseteq \Gd\right) \geq 1-5\varepsilon.
\end{equation*}
\end{lemma}

\dteal{Observe that for any $\Delta$-regular graph, we have $m_H=\Delta/2$. Therefore, the sparsity regime in  Lemma \ref{lem:minimum subgraph count} for regular graphs coincides with the sparsity regime in Proposition \ref{prop:Lower bound}.}

We defer the proof of Lemma \ref{lem:minimum subgraph count} and first use it to prove Proposition \ref{prop:Lower bound}.

\begin{proof}[Proof of Proposition \ref{prop:Lower bound}]
\dteal{We notice that} it suffices to construct a deterministic graph $G_{\star}\in \widehat{\Xd}$ such that 
\begin{equation}\label{eq:copies-in-G-lbd}
N(H,G_{\star})\geq (\delta+\varepsilon) n^{v_{H}}p^{e_{H}}\quad \tu{and }\quad e(G_{\star})=o(nd),
\end{equation}
\dteal{and show that
\begin{equation}\label{eq:G-star-cost}
\P(G_\star \subseteq \Gd) \geq \exp\left(-(1+o(1))
\dteal{\Psi_H(\delta+2\varepsilon)} n^{2}p^{\Delta}\log(1/p)\right).
\end{equation}
Indeed,} by Lemma \ref{lem:minimum subgraph count}, \dteal{
\begin{equation}\label{eq:good event}
    \P(\mathcal{A}_0)\geq (1-5\varepsilon)\cdot \P(G_{\star}\subseteq \Gd).
\end{equation}
where
\[
\mathcal{A}_0:= \left\{N(H, \Gd, \overline{G_{\star}})\geq (1-\varepsilon)n^{v_{H}}p^{e_{H}}\right\}\ \bigcap \left\{ G_{\star}\subseteq\Gd\right\}.
\]
On $\mathcal{A}_0$, by \eqref{eq:copies-in-G-lbd},  we have } 
\[N(H,\Gd)\geq N(H, G_{\star})+ N(H, \Gd, \overline{G_{\star}})\geq (1+\delta) n^{v_{H}}p^{e_{H}}.\]
\dteal{Therefore, the required lower bound is immediate from \eqref{eq:G-star-cost} and the continuity of the map $\delta \mapsto \Psi_H(\delta)$.} 

\dteal{Hence, it remains to prove \eqref{eq:copies-in-G-lbd} and \eqref{eq:G-star-cost}.}
We construct $G_{\star}$ \dteal{and lower bound the probability cost of its planting} separately for $\Delta=2$ and $\Delta\geq 3$.

Let us first consider $\Delta\geq 3$. Fix $\varepsilon\in (0,1)$ sufficiently small, and  set $s:=\lceil (\delta+2\varepsilon)^{1/v_{H}}np^{\Delta/2}\rceil$. Let $G_{\star}\in \widehat{\Xd}$ be the clique on the first $s$ vertices, that is,
\begin{equation*}
    E(G_\star) = \left\{(i,j):\ 1\leq i<j\leq s\right\}.
\end{equation*}
Since every injection of $V(H)$ into $\set{s}$ gives a labelled copy of $H$ in $G_{\star}$, we have 
\[
N(H, G_{\star})\geq {s\choose v_{H}}(v_{H}!)= (1-o(1))s^{v_{H}}\geq (\delta+\varepsilon)n^{v_{H}}p^{e_{H}}, 
\]
for all sufficiently large $n$, \dteal{where in the penultimate step we have used $np^{\Delta/2} \gg 1$}. 

Since $\Delta\geq 3$ and $p=o(1)$, we get $\tu{deg}_{G_{\star}}(i)\leq s = o(d)$ for any $i\in \set{s}$. Moreover, \dteal{as $\Delta \geq 3$,}
\[
e(G_{\star})={s\choose 2} = (1+o(1))(\delta+2\varepsilon)^{2/v_{H}}n^{2}p^{\Delta}/2 = o(nd). 
\]
\dteal{Therefore, Corollary \ref{cor:almost independence lemma(simple)} applies.}   
\dteal{Applying} Corollary \ref{cor:almost independence lemma(simple)} by sequentially exposing edges in $G_{\star}$, with $\eta = o(1)$, \dteal{and upon using that $p \ll 1$}, yields the probability bound
\begin{equation}
\label{eq:prob of clique}
\P(G_\star\subseteq \Gd) \geq (1-o(1))\cdot p^{(1+o(1))e(G_{\star})}\geq \exp(-(1+o(1))\frac{1}{2}(\delta+2\varepsilon)^{2/v_{H}}n^{2}p^{\Delta}\log(1/p)).
\end{equation} 
\dteal{Thus we have \eqref{eq:copies-in-G-lbd} and \eqref{eq:G-star-cost} for $\Delta \geq 3$.}

Now, let $\Delta=2$. \dteal{In this case $s  > d$ for $\delta  \geq 1$. Hence, due to the degree upper bound on graphs belonging to $\widehat{\Xd}$ the choice of $G_\star$ from above does not work. So, it needs some change. 

We therefore now set} $G_{\star}$ to be the union of several vertex disjoint cliques. Choose $\varepsilon\in(0,1)$ sufficiently small so that $\{\delta\}+2\varepsilon<1$ and set $s_0 \coloneqq \lfloor (\{\delta\}+2\varepsilon)^{1/v_{H}}d\rfloor$. Let $G_{\star}= G_{\star}^{1}\cup \cdots \cup G_{\star}^{\lfloor \delta \rfloor}\cup G_{\star}^{\lfloor \delta \rfloor+1}$, where $G_{\star}^{1}, \ldots, G_{\star}^{\lfloor \delta\rfloor}$ are copies of $K_{d+1}$ and $G_{\star}^{\lfloor \delta\rfloor +1}$ is a copy of $K_{s_0+1}$. More precisely, 
\begin{equation*}
E(G_\star^{k}) = \left\{(i,j):(k-1)(d+1)+1\leq i<j\leq k(d+1)\right\}, \tu{ for } k\in\set{\lfloor\delta\rfloor },
\end{equation*}
and
\begin{equation*}
E(G_\star^{\lfloor \delta\rfloor+1}) = \left\{(i,j): (d+1)\lfloor\delta\rfloor+1\leq i<j\leq (d+1)\lfloor \delta\rfloor+s_0+1\right\}.
\end{equation*}

\dteal{Next notice} that every injection of $V(H)$ into \dteal{the} vertex set of one of the cliques gives a copy of $H$. Therefore, \dteal{as $p \ll 1$},  
\begin{equation*}
    N(H, G_{\star})= \sum_{i=1}^{\lfloor \delta\rfloor +1}N(H, G_{\star}^{i})\geq \lfloor\delta\rfloor{d+1 \choose v_{H}}v_{H}!+{s_0+1\choose v_{H} }v_{H}!\ \geq\ (\delta+\varepsilon)\cdot n^{v_{H}}p^{e_{H}}.  
\end{equation*}
Further, \dteal{as $d= np \gg 1$}, we have 
\[
e(G_{\star}) = \lfloor\delta\rfloor {d+1 \choose 2}+ {s_0+1\choose 2}= (1+o(1))(\lfloor \delta \rfloor+ \{\delta+2\varepsilon\}^{2/v_{H}})d^{2}/2 = o(nd).
\]
\dteal{It remains to prove the probability lower bound \eqref{eq:G-star-cost} for this choice of $G_\star$. Turning to do the same},  
set $\mathcal{G}_{0}:=\emptyset$ and $\mathcal{G}_{i} := \cup_{j=1}^{i-1} G_{\star}^{j}$ for each $i\in \set{\lfloor \delta+1\rfloor}$. \dteal{We} apply Corollary \ref{cor:almost independence lemma(poisson)} iteratively to get
\begin{equation}
\label{eq:Delta=2, clique planting}\P(G_{\star}\subseteq\Gd) = \prod_{i=1}^{\lfloor\delta \rfloor +1} \P(G_{\star}^{i} \subseteq \Gd | \mathcal{G}_{i}\subseteq \Gd)\geq  \prod_{i=1}^{\lfloor\delta \rfloor+1}p^{e(G_{\star}^{i})}\cdot D_{G_{\star}^{i},  \mathcal{G}_i}\cdot (f_{e(G_{\star}^{i}\cup\mathcal{G}_{i})})^{-e(G_{\star}^{i})}. 
\end{equation}
The vertex-disjointness of $G_{\star}^{i}$ and $\mathcal{G}_i$ implies $D_{G_{i}^{\star}, \mathcal{G}_{i}} = D_{G_{\star}^{i},\emptyset}$. As $G_{\star}^{i}$ are cliques, we deduce
\[D_{G_{\star}^{i},\emptyset} = \left(\frac{d!}{d^{d}}\right)^{d+1},\  \forall i\in \set{\lfloor \delta \rfloor},\quad  \tu{and} \quad D_{G_{\star}^{\lfloor \delta\rfloor +1}, \emptyset} = \left(\frac{d!}{d^{s_0}((d-s_{0})!)}\right)^{s_0+1}.\]
Since $s_0\leq d$, Stirling's approximation yields $|\log D_{G_{\star}^{i},\emptyset}| = O(d^{2})$ uniformly over all $i\in\set{\lfloor\delta\rfloor+1}$. 

\dteal{Next, as} $f_{k}$  is increasing in $k$ (see definition in \eqref{eq:error bound on independence}), we obtain
\[\prod_{i=1}^{\lfloor \delta \rfloor +1}(f_{e(G_{\star}^{i}\cup \mathcal{G}_i)})^{e(G_{\star}^{i})}\leq (f_{e(G_{\star})})^{e(G_{\star})} = \exp\left(O\left(p\cdot e(G_{\star})+\frac{e(G_{\star})^{2}}{nd}\right)\right) = \exp(O(pd^2)),\]
because $e(G_{\star})=O(d^{2})$ and $d=np$. 

Substituting these bounds into \eqref{eq:Delta=2, clique planting} and using $p=o(1)$, we derive
\begin{equation*}
\P(G_{\star}\subseteq \Gd)\geq \exp(-e(G_{\star})\log (1/p) -O(d^{2}+pd^{2})) = \exp(-(1+o(1))
\dteal{\Psi_H(\delta+2\varepsilon)}d^{2}\log(1/p)).
\end{equation*}
\dteal{This yields \eqref{eq:G-star-cost} for $\Delta=2$. The proof of the proposition is now complete.} 
\end{proof}
Now, we finish the proof of Lemma \ref{lem:minimum subgraph count}. \dteal{We will assume that $H$ is not a forest. Otherwise, it is straightforward to obtain a deterministic bounds on $N(H, G\setminus G')$ in the leading order, for any $d$-regular graph $G$ on $n$ vertices and any graph $G'$ with $e(G') \ll nd$.}  
\begin{proof}[Proof of Lemma \ref{lem:minimum subgraph count}]
\dteal{In a nutshell the proof follows upon applying the Chebychev's inequality together with tight upper and lower bounds on the second and first moments of $N(H, \Gd, \overline{G})$, respectively. To obtain these bounds we will repeatedly use Corollary \ref{cor:almost independence lemma(simple)}. 

However, notice that to apply Corollary \ref{cor:almost independence lemma(simple)} we need to ensure that the endpoints of the edge to be added does not have a large degree in $G$. One approach would be to discard vertices that have `large' degrees in $G$, focus only on labelled copies that do not use any vertex of large degree, and prove the probability lower bound for this subset of labelled copies. 

To carry out this strategy, we need the following definitions:} 
Set \[\eta \coloneqq  \min\left\{\frac{\varepsilon}{8e_{H}}, \frac{\varepsilon^{3}}{4e_{H}}\right\}\quad \tu{and} \quad 
S_{\eta} \coloneqq \{v\in \set{n}: \tu{deg}_{G}(v)\leq \eta d\}.\]
Define
\begin{equation*}
        \mathcal{F}_{H} \coloneqq \mathcal{F}(H,S_{\eta},\overline{G})\coloneqq \left\{I:\ I \tu{ is a labelled copy of $H$ in $K_n$ such that } V(I)\subseteq S_{\eta} \tu{ and }E(I)\cap E(G)=\dteal{\emptyset}\right\}
\end{equation*}
\dteal{and
\[
X_{H}\coloneqq X_{H}(\Gd, S_\eta, \overline{G})\coloneqq \sum_{I\in \mathcal{F}_{H}}\mathds{1}(I\subseteq \Gd).
\]
We claim that 
\begin{equation}\label{eq:F-H-lbd}
    |\mathcal{F}_{H}|\geq (1-o(1))n^{v_H}.
    \end{equation}
The proof of the lower bound in \eqref{eq:F-H-lbd} is quite straightforward. From the definition of $S_\eta$ we find}
$|S_{\eta}^{c}|\leq \frac{2e(G)}{\eta d} =o(n)$, \dteal{as $e(G) \ll n d$}. Hence the number of labelled copies of $H$ that contains at least one vertex \dteal{belonging to} $S_{\eta}^{c}$ is $o(n^{v_{H}})$. 
The number of copies of $H$ \dteal{that contains at least one edge of $G$} is $O(e(G)\cdot (n-2)^{v_{H}-2}) =o(n^{v_{H}})$, where we used $e(G)=o(nd)$ \dteal{again. Combining the last two observations we deduce \eqref{eq:F-H-lbd}.} 


We next obtain a uniform lower bound for the conditional probability of \dteal{$I \subseteq \Gd$ for each} $I\in \mathcal{F}_{H}$. We fix such an $I\in \mathcal{F}_{H}$, enumerate its edges in $I$ as $e_{1}, e_{2},\ldots, e_{e_{H}}$. We expose these edges one at a time. At the $i$-th step, we condition on $\dteal{\widetilde G_i \coloneqq} G\cup \dteal{\{e_1,\ldots, e_{i-1}\}\subseteq \Gd}$. \dteal{Since every vertex  $v \in V(I)$ belongs to $S_{\eta}$, $I$ is a copy of $H$, and the maximum degree of $H$ is $\Delta$, 
we find that $\deg_{\widetilde G_i}(v) \leq \eta d+\Delta$. Since $H$ is not a forest it must contain a cycle, and therefore $m_H \geq 1$. This in turn yields that $d \gg 1$.} Thus, by iteratively applying Corollary \ref{cor:almost independence lemma(simple)}
\begin{equation*}
\P(I\subseteq \Gd\big| G\subseteq \Gd) = \prod_{i=1}^{e_{H}} \P\left(\{{e_{i}}\}\subseteq \Gd\ \big| \ \dteal{\widetilde G_i} 
\right)\geq \left(1-\eta-o(1)\right)^{2e_H} \cdot p^{e_{H}}.
\end{equation*} 
\dteal{Hence}
\begin{equation}
\label{eq:expectation lower bound}
\begin{aligned}
    \E[X_{H}\ \big|\ G\subseteq \Gd]\ \geq\ |\mathcal{F}_{H}|\cdot \left(1-\eta-o(1)\right)^{2e_{H}}\cdot p^{e_{H}} &\geq  (1-2e_{H}\eta-o(1))\cdot n^{v_{H}}p^{e_{H}}\\[-4pt]
    &>(1-\varepsilon)n^{v_{H}}p^{e_{H}},	  
\end{aligned}
\end{equation}
where the last inequality follows from our choice of $\eta<\varepsilon/(8e_{H})$. 

It remains to show that $X_{H}$ is sufficiently concentrated under the conditional law. 
Applying Chebyshev's inequality with respect to the conditional measure gives
\begin{equation}
\label{eq:chebyshev}
\P\left(X_{H} \leq (1-\varepsilon)n^{v_{H}}p^{e_{H}}\big| G\subseteq \Gd\right) \leq \frac{\E\left[X_{H}^{2}\big|G\subseteq \Gd\right]-\left(\E\left[X_{H}\big|G\subseteq \Gd\right]\right)^{2}}{\left(\E\left[X_{H}\big|G\subseteq \Gd\right]-(1-\varepsilon)n^{v_{H}}p^{e_{H}}\right)^{2}}.
\end{equation}
\dteal{To proceed from here} we need an upper bound on $\E[X_{H}^{2}]$. \dteal{To this end,} expanding the second moment, \dteal{we obtain}
\begin{equation}
\begin{aligned}
\label{eq:l2_norm}
\E\left[X_{H}^{2}\big| G\subseteq \Gd\right]  & = \sum_{I,I^{'}\in \mathcal{F}_{H}} \P(I\cup I^{'}\subseteq \Gd|G\subseteq \Gd) \\
 & \dteal{\leq (1+o(1)) \left[
 \sum_{k \in \set{v_H}\cup\{0\}} \sum_{I, I^{'}\in \mathcal{F}_{k}} p^{e(I\cup I^{'})} + \sum_{\emptyset \neq J \subseteq H} \sum_{I, I^{'}\in \widehat{\mathcal{F}}_J} p^{e(I\cup I^{'})}\right],}
\end{aligned}
\end{equation}
\dteal{where for any $k \in \set{v_H} \cup\{0\}$ the set $\mathcal{F}_k \subseteq \mathcal{F}_H\times \mathcal{F}_H$ is the collection of all pairs $(I, I')$ such that $V(I) \cap V(I')=k$ and $E(I) \cap E(I')=\emptyset$, for $\emptyset \neq J \subseteq H$ the set $\widehat{\mathcal{F}}_J \subseteq \mathcal{F}_H\times \mathcal{F}_H$ is the collection of all pairs $(I, I')$ such that $I \cap I'=J$, and in the last step above} 
%
%
we used Corollary \ref{cor:almost independence lemma(simple)} and the fact that $I\cup I^{'}$ has at most $O(1)$ many edges. 

\dteal{It remains to simplify the sum in the RHS of \eqref{eq:l2_norm}.} 
%
%
%
%
%
\dteal{Turning to do the same, we note that for any pair $(I, I^{'}) \in \mathcal{F}_k$ for some 
$k\in \set{v_H}\cup\{0\}$, we have $e(I\cup I^{\prime})= 2e_{H}$. Further, it is clear that 
$|\mathcal{F}_k| \leq {n\choose k}{n-k\choose {v_{H}-k}} {n-v_H\choose {v_{H}-k}}(v_{H}!)^{2} \leq n^{2v_{H}-k/2}$, for all large $n$. Therefore, 
\begin{equation}\label{eq:l-2-F-k}
\sum_{k \in \set{v_H}\cup\{0\}} \sum_{I, I^{'}\in \mathcal{F}_{k}} p^{e(I\cup I^{'})} = (1+o(1))n^{v_H}p^{e_H}.
\end{equation}
On the other hand, for any pair $(I,I')\in \widehat{\mathcal{F}}_J$ for some $\emptyset \neq J \subseteq H$ we have $e(I\cup I^{'}) = 2e_{H}-e_{J}$. Moreover,  
for any such $J$, $|\widehat{\mathcal{F}}_J| \leq {n\choose v_{J}} {n-v_{J}\choose 2v_{H}-2v_{J}}(v_{H}!)^{2}$. Therefore
%
\begin{equation}\label{eq:l-2-F-J}
\begin{aligned}
\sum_{\emptyset \neq J \subseteq H} \sum_{I, I^{'}\in \widehat{\mathcal{F}}_J} p^{e(I\cup I^{'})} \leq \sum_{\emptyset \neq J \subseteq H} 	{n\choose v_{J}} {n-v_{J}\choose 2v_{H}-2v_{J}}(v_{H}!)^{2} \cdot p^{2e_{H}-e_{J}}  & \lesssim\max_{\emptyset \neq J \subseteq H} n^{2v_H - v_{J}} p^{2e_{H}-e_{J}} \\ &= o(n^{2v_{H}}p^{2e_{H}}),
\end{aligned}
\end{equation} 
where last inequality follows from $p\gg n^{-1/m_{H}}$. Plugging the bounds \eqref{eq:l-2-F-k} and \eqref{eq:l-2-F-J} in \eqref{eq:l2_norm} we deduce} 
\begin{equation}
\label{eq:l2_norm upper bound}
\E[X_{H}^{2}\big| G\subseteq \Gd] \leq (1+o(1))\cdot n^{2v_{H}}p^{2e_{H}}.
\end{equation}
\dteal{Now, substituting} both the expectation lower bound \eqref{eq:expectation lower bound} and the second moment estimate \eqref{eq:l2_norm upper bound} into \eqref{eq:chebyshev}, we obtain
\begin{equation}
\label{eq:complement bound}
\P\left(X_{H}\leq (1-\varepsilon)n^{v_{H}}p^{e_{H}}\big| G\subseteq \Gd\right) \leq \frac{4e_{H}\eta +o(1)}{(\varepsilon - 2e_{H}\eta -o(1))^{2}} \leq 4\varepsilon+o(1)< 5\varepsilon,
\end{equation}
where the penultimate inequality follows from our assumption $\eta\leq \varepsilon^{3}/4e_{H}$. Using \eqref{eq:complement bound} and the fact that $N(H,\Gd, \overline{G})\geq X_{H}(\Gd)$ yields the desired conclusion.
\end{proof}

\section{Localized regime: the upper Bound}\label{sec:loc-regime-ubd}
\dteal{The main goal of this section is to prove the necessary large deviation upper bound in the localized regime. In particular we will show that, for any $\Delta$-regular connected graph $H$ and any $\varepsilon >0$,
\begin{equation}\label{eq:ubd-with-mathfrak-f}
\liminf_{n\rightarrow \infty}-\frac{\log \P(\UTH)}{n^{2}p^{\Delta}\log (1/p)}\geq \Psi_{H}(\delta)\cdot(1-\widehat{\mathfrak{f}}_{H}(\varepsilon)),
\end{equation}
for some nonnegative function $\widehat{\mathfrak{f}}_H(\cdot)$ such that $\lim_{\varepsilon \downarrow 0} \widehat{\mathfrak{f}}_H(\varepsilon)=0$.} 

\dteal{As already mentioned in Section \ref{sec:proof-outline}, the proof starts by identifying a class of subgraphs with a small number of edges so that each member of that class contains almost all the excess copies of $H$ and has the property that each edge in any such graph participates in  non negligible many copies. Following \cite{ABRB}, this class of graphs will be termed `core' graphs (definition differs from that defined in \cite{HMS}; explained below). To achieve this we need to introduce another class of subgraphs with a small number of edges, to be termed as `pre-seed' graphs (following \cite{ABRB} again), whose presence substantially increases the conditional expectation of the  number of copies of $H$. Recall \eqref{eq:def-Xd}.} 


\begin{definition}[Pre-seed]Let $\varepsilon>0$ be sufficiently small \dteal{and} $\cbar\coloneqq \cbar(\varepsilon, \delta)$ be a sufficiently large constant.
A graph $G\in \widehat{\Xd}$ is called a pre-seed if the following holds:  
\begin{itemize}
    \item[(PS1)]$\E\left[N(H, \Gd)\big|G\subseteq \Gd\right]\geq (1+\delta(1-\varepsilon)) \cdot n^{v_{H}}p^{e_{H}}$,
    \item[(PS2)]$e(G)\leq \cbar n^{2}p^{\Delta}\log(1/p)$.
\end{itemize}
\end{definition}

\dteal{For graphs $H$ and $G$, and any edge $e \in E(G)$ we write $N(H,G,e)$ to denote the number of labelled copies of $H$ in $G$ that uses the edge $e$.}

\begin{definition}[Core]
A graph $G\in \widehat{\Xd}$ is called a core if the \dteal{followings hold:}  
\begin{itemize}
    \item[(C1)]$N(H, G)\geq \delta(1-3\varepsilon) \cdot n^{v_{H}}p^{e_{H}}$,
    \item[(C2)]$e(G)\leq \cbar n^{2}p^{\Delta}\log(1/p)$,
\end{itemize}
and,
\begin{itemize}
    \item[(C3)]$\min_{e\in E(G)}N(H, G, e)\geq \delta \varepsilon n^{v_{H}}p^{e_{H}}/\left(\cbar n^{2}p^{\Delta}\log(1/p)\right)$.
\end{itemize}
\end{definition}
\dteal{The definitions of pre-seed and core graphs depend implicitly on $\varepsilon$. For ease in writing, we suppress this dependency. The first reduction towards the proof of \eqref{eq:ubd-with-mathfrak-f} is the following lemma.}
\begin{lemma}
\label{lem:existence of a core}
Assume $np^{\Delta/2}\gg 1$ and $p\ll 1$. For sufficiently large $n$, we have
    \begin{equation}\label{eq:contains-core}
        \P\left(\UTH\right)\leq (1+\varepsilon)\cdot \P(\Gd \tu{ contains a core}).
\end{equation}
\end{lemma}
\dteal{The proof of Lemma \ref{lem:existence of a core} is done in two stages. In the first stage, it is shown that \eqref{eq:contains-core} continues to hold if `core' in its RHS is replaced by `pre-seed'. The second stage involves in showing that any pre-seed graph must contain a core subgraph, which will be an immediate byproduct of our next lemma.} 
%

\begin{lemma}
\label{lem:expectation difference upper bound}
Fix $\Delta\geq 2$ and let $H$ be a connected $\Delta$-regular graph. Assume $1\ll d=np\ll n$ and let $G\in \widehat{\Xd}$ satisfy $e(G) \leq \cbar n^{2}p^{\Delta}\log(1/p)$. For sufficiently large $n$, we have
\begin{equation}
    \label{eq:expectation difference upper bound}
    \E[N(H, \Gd)\big| G\subseteq \Gd] - \E[N(H,\Gd)] \leq (1+o(1))\cdot N(H, G) + o(n^{v_{H}}p^{e_{H}}).
\end{equation}
\end{lemma}
Its proof relies on expanding the left hand side of \eqref{eq:expectation difference upper bound} and identifying the terms that \dteal{have} non negligible contribution. These turn out to come only from copies of $H$ already contained in $G$, giving the contribution $N(H,G)$. \dteal{Let us mention in passing that \eqref{eq:expectation difference upper bound} does not hold in the Erd\H{o}s-R\'enyi setting for $p \gg n^{-1/\Delta}$ (see the proof of \cite[Lemma 7.3]{HMS}), and it only holds when $p \ll n^{-1/\Delta} (\log n)^{-O(1)}$ (see the proof of \cite[Lemma 3.2]{ABRB}). This results in the difference in the notion of core graphs of \cite{HMS} and the ones here and in \cite{ABRB}.} 


\dteal{The proof of Lemma \ref{lem:expectation difference upper bound} requires the following result.} 
\begin{lemma}
\label{lem:contribution zero} Fix some constant \dteal{$C\in (0, \infty)$}.
Assume $p=o(1)$ and $G\in \widehat{\Xd}$ with $e(G) \leq C n^{2}p^{\Delta}\log(1/p)$. Let $\Delta\geq 2$ and $H$ be a connected $\Delta$-regular graph. Then for any non-empty proper subgraph \dteal{$J\subsetneq  H$} without  isolated vertices, we have
\begin{equation*}
    N(J, G) = o(n^{v_{J}}p^{e_{J}}).
\end{equation*} 
\end{lemma}

\dteal{Before proceeding to prove Lemma \ref{lem:contribution zero} let us recall the following standard definitions: For a graph $F$, a subset of $S \subseteq V(F)$ is said to be its vertex cover if $S$ intersects every edge of $F$. We denote $\vartheta(F)$ to be the minimum size of a vertex cover of $F$. Next we recall that matching is a subset of edges with no two edges sharing a vertex. We let $\nu(F)$ to be the matching number of $F$, i.e.~the size of the largest matching of $F$. Further recall that by K\H{o}nig's theorem, for any bipartite graph $F$ we have $\vartheta(F)=\nu(F)$.} 

\begin{proof}[Proof of Lemma \ref{lem:contribution zero}]
\dteal{Let \dteal{$\emptyset \neq J\subsetneq H$} be a nonempty proper subgraph with no isolated vertices. By \cite[Lemma~5.2]{HMS}, $\Delta(v_{J}-\alpha^{*}_{J})-e_{J}\geq 0$, where we recall that $\alpha^*_J$ is the fractional independence number of $J$.} 	
We consider two cases based on the value of $\Delta(v_{J}-\alpha^{*}_{J})-e_{J}$.\\
\textbf{Case 1: $\Delta(v_{J}-\alpha^{*}_J)-e_{J}>0$.}
By applying the bound from \cite[Theorem~5.4]{HMS}, \dteal{as $p \ll 1$}, we obtain 
\begin{equation*}
    N(J, G)\leq (2e(G))^{v_{J}-\alpha_{J}^{*}}\cdot n^{2\alpha^{*}_{J}-v_{J}} \leq (2C\log (1/p))^{v_{J}-\alpha_{J}^{*}}\cdot  n^{v_{J}} p^{e_{J}}\cdot p^{\Delta(v_{J}-\alpha_{J}^{*})-e_{J}}=o(n^{v_{J}}p^{e_{J}}).
\end{equation*}
\textbf{Case 2: $\Delta(v_{J}-\alpha^{*}_{J})-e_{J}=0$.} In this case, \cite[Lemma~5.2]{HMS} implies that $J$ is a bipartite graph with vertex bipartition $V(J) = A\cup B$ such that $\tu{deg}_{J}(a) = \Delta$ for all $a\in A$. Furthermore, we must have $|A|<|B|$. Otherwise $J$ would be a $\Delta$-regular proper subgraph of $H$, thereby contradicting the connectedness of $H$.


Suppose $\widetilde{V}\subseteq V(J)$ is a vertex cover of $J$. Since every edge of $J$ is incident to \dteal{some vertex in} $\widetilde{V}$, it follows that:
\begin{equation*}
    \Delta \cdot |A| = e_{J}\leq  \sum_{v\in \widetilde{V}} \tu{deg}_{J}(v) \leq \Delta \cdot |\widetilde{V}|,
\end{equation*} 
\dteal{and consequently $|A| \leq |\widetilde V|$}, where the final inequality holds because $J$ has maximum degree $\Delta$. 
Since $A$ itself is a vertex cover (\dteal{as} every edge has an endpoint in $A$), it is \dteal{a} minimum vertex cover.
Therefore, \dteal{by K\H{o}nig's theorem, $\nu(J)=|A|$. Fix any largest matching of $J$, say $M$.} 
Because each edge in the matching $M$ has to have one endpoint in $A$ and other in $B$, this implies that $M$ covers all vertices in $A$. 

We now count the number of labelled copies of $J$ in $G$ \dteal{as follows}: First we choose the images of the \dteal{edges in $M$. As $|M|=|A|$}, there are at most $(2e(G))^{|A|}$ choices for these images. The remaining $|B|-|A|$ vertices in $B$ are not covered by the matching $M$. Each such vertex is adjacent to some vertex in $A$ (since $J$ has no isolated vertices). \dteal{Since} 
the images of the vertices in  $A$ \dteal{have already been} fixed, the image of each remaining vertex in $B$ is constrained to be a neighbor of the image of its counterpart in $A$. \dteal{As} $G$ has maximum degree $d$, there are at most $d^{|B|-|A|}$ choices for these remaining vertices. 

Combining these counts and using the assumption $e(G)\lesssim n^{2}p^{\Delta}\log(1/p)$, we get
\begin{equation*}
    N(J, G)\leq (2e(G))^{|A|}\cdot d^{|B|-|A|}\leq n^{|A|+|B|}\cdot p^{\Delta |A|}\cdot (2C\log (1/p))^{|A|}\cdot  p^{|B|-|A|} = o(n^{v_{J}}p^{e_{J}}), 
\end{equation*}
where the last step follows since $e_{J} = \Delta |A|$,  $v_{J} = |A|+|B|$, $|B|>|A|$, \dteal{and $p \ll 1$}. This concludes the proof.
\end{proof}

\dteal{We are now ready to prove Lemma \ref{lem:expectation difference upper bound}.}

\begin{proof}[Proof of Lemma \ref{lem:expectation difference upper bound}]
Using linearity of  expectation, we rewrite the left hand side of \eqref{eq:expectation difference upper bound} as  \dteal{the following sum:}
\begin{equation}
    \label{eq:expectation expansion}
    \E[N(H, \Gd)\big|G\subseteq \Gd] - \E[N(H, \Gd)] = \sum_{I}\dteal{\left[\P(I\subseteq \Gd|G\subseteq \Gd) -\P(I\subseteq \Gd)\right]},
\end{equation}
\dteal{where the sum over $I$ is} taken over all labelled copies of $H$ in $K_{n}$. We break the above sum into two cases. 

First consider labelled copies $I$  such that  $I\cap G = \emptyset$. Using Corollary \ref{cor:almost independence lemma(simple)} with $\eta =\Delta\cdot d^{-1}$, \dteal{as $d \gg 1$}, we have
\begin{equation*}
    \P(I\subseteq \Gd | G\subseteq \Gd) - \P(I\subseteq \Gd) \leq (1+o(1))^{e_{H}}\cdot p^{e_{H}} - (1-o(1))^{e_{H}}\left(1-(\Delta/d)\right)^{2e_{H}}\cdot p^{e_{H}} = o(p^{e_{H}}). 
\end{equation*}
As the number of labelled copies \dteal{of $H$ in $K_n$} is at most $n^{v_{H}}$, we obtain
\begin{equation}
    \label{eq:not intersecting copies}
    \sum_{I:\dteal{E(I)\cap E(G)} =\emptyset}\dteal{\left[\P(I\subseteq \Gd|G\subseteq \Gd) -\P(I\subseteq \Gd)\right]} = o(n^{v_{H}}p^{e_{H}}).
\end{equation}
On the other hand, consider $I$ where $I\cap G$ is isomorphic to a non-empty subgraph $J$ of $H$ without any isolated vertices. For such an $I$, we \dteal{employ Corollary \ref{cor:almost independence lemma(simple)}, applying it} $(|E(I)|-|E(J)|)$ \dteal{many} times to get
\begin{equation*}
    \P(I\subseteq \Gd|G\subseteq \Gd) = \P(I\setminus G\subseteq \Gd|G\subseteq \Gd)\ \leq\ (1+o(1))^{e_{H}-e_{J}} p^{e_{H}-e_{J}}.
\end{equation*}
Number of such $I$'s can be bounded by first choosing a copy of $J$ in $G$ and then each remaining $v_{H}-v_{J}$ \dteal{many} vertices gets at most $n$ choices. This leads to the estimate:
\begin{equation}\label{eq:intersecting copies}
\begin{aligned}
    \sum_{I:\dteal{E(I)\cap E(G)} \neq \emptyset} \dteal{\left[\P(I\subseteq \Gd|G\subseteq \Gd) -\P(I\subseteq \Gd)\right]}  & \leq (1+o(1))\cdot \sum_{\emptyset\neq J\subseteq H}N(J, G)\cdot n^{v_{H}-v_{J}}\cdot p^{e_{H}-e_{J}}\\
& \dteal{\leq (1+o(1))\cdot N(H,G) + o( n^{v_{H}} p^{e_{H}})},
\end{aligned}		
\end{equation}
where \dteal{in the first step} the sum 
is over non empty subgraphs $J$ of $H$ without any isolated vertices, \dteal{and in the final step we used} 
Lemma \ref{lem:contribution zero}. 
\dteal{Now, adding} up the bounds in \eqref{eq:intersecting copies} and \eqref{eq:not intersecting copies}, finishes the proof.
\end{proof} 

\dteal{We now return to the proof of Lemma \ref{lem:existence of a core}.}

\begin{proof}[Proof of Lemma \ref{lem:existence of a core}]
We first claim that 
\begin{equation}
    \label{eq:pre-seed existence}
    \P(\dteal{\UTH})\leq (1+\varepsilon)\cdot \P(\Gd\tu{ contains a pre-seed graph $G$}).
\end{equation}
Applying \cite[Lemma~3.7]{HMS} with $X=N(H, \Gd)$ and $\ell = \cbar n^{2}p^{\Delta}\log(1/p)$, give us  
\begin{equation}
    \label{eq:not pre-seed upper bound}
    \P(\UTH \cap \{\Gd \tu{ contains a pre-seed graph $G$}\}^{c})\leq \left(\frac{1+\delta(1-\varepsilon)}{1+\delta}\right)^{\frac{\cbar}{e_{H}}n^{2}p^{\Delta}\log(1/p)}.
\end{equation} 
Choosing $\cbar$ sufficiently large and comparing this estimate with the lower bound \eqref{eq:lower bound}, we obtain 
\begin{equation*}
    \P(\UTH \cap \{\Gd \tu{ contains a pre-seed}\}^{c})\leq (\varepsilon/(\varepsilon+1))\cdot \P(\UTH).
\end{equation*}
This gives the claim \eqref{eq:pre-seed existence}. It remains to show that every pre-seed contains a core. Let $G$ be a pre-seed. By (PS1), Lemma \ref{lem:expectation difference upper bound}, and the fact that $\E[N(H, \Gd)] = (1+o(1))n^{v_{H}}p^{e_{H}}$ \dteal{(one may use Corollary \ref{cor:almost independence lemma(simple)})}, we have
\begin{equation*}
    (1+\delta(1-\varepsilon))n^{v_{H}}p^{e_{H}}\leq \E[N(H, \Gd)|G\subseteq \Gd] \leq (1+o(1))\cdot n^{v_{H}}p^{e_{H}}+(1+o(1))\cdot N(H, G).
\end{equation*}
Hence, for large enough $n$, $N(H, G)\geq \delta(1-2\varepsilon)n^{v_{H}}p^{e_{H}}$. We now prune $G$ by repeatedly deleting any edge \dteal{$e \in E(G)$} that participate in strictly fewer than $\delta\varepsilon n^{v_H}p^{e_H}/\left(\cbar n^{2}p^{\Delta}\log(1/p)\right)$ labelled copies of $H$ in the current graph. Let $G_{\star}$ be the graph obtained at the end of this procedure. So, $G_{\star}$ satisfies (C3). Moreover, since $e(G_{\star})\leq \cbar n^{2}p^{\Delta}\log(1/p)$, the total number of copies of $H$ removed is less than $\delta \varepsilon n^{v_{H}}p^{e_{H}}$, thus implying condition (C1) for $G_{\star}$. Finally, the condition (C2) is automatically inherited as $G_\star \subseteq G$.
\end{proof}

\begin{remark}
\dteal{Note that by enlarging $\overline C$ in the proof of Lemma \ref{lem:existence of a core} one obtains
\[
\P(\UTH \cap \{\Gd \tu{ contains a pre-seed}\}^{c})\leq  \P(\UTH)^2.
\]
Hence, repeating the rest of the argument one arrives at the following stronger bound:
\begin{equation}\label{eq:cond-struc-starting}
\P(\{\Gd \text{ contains a core graph}\}\mid \UTH) =\frac{\P(\UTH \cap \{\Gd \text{ contains a core graph}\})}{\P(\UTH)}\geq 1 -\P(\UTH). 
\end{equation}
This stronger bound will be useful for the proof of Theorem \ref{thm:cond-loc}.} 
\end{remark}

Equipped with Lemma \ref{lem:existence of a core}, it remains to bound the probability that $\Gd$ contains a core. \dteal{This is the content of the next result.} 
    
\begin{proposition}
\label{prop:prob of existence of core}
Fix $\delta>0$ and a $\Delta$-regular connected graph $H$ with $\Delta\geq 2$. Then, for any $p\in (0,1)$ satisfying $(\log n)^{1/(v_{H}-2)}\ll np^{\Delta/2}\ll n$, and $\varepsilon\in (0,1)$ sufficiently small, we have
\[\liminf_{n\rightarrow \infty}-\frac{\log \P(\Gd \tu{ contains a core graph }G)}{n^{2}p^{\Delta}\log (1/p)}\geq \Psi_{H}(\delta)\cdot(1-\mathfrak{f}_{H}(\varepsilon)),\]
for some non-negative function $\mathfrak{f}_{H}(\cdot)$ such that $\lim_{\varepsilon\downarrow 0}\mathfrak{f}_{H}(\varepsilon) = 0$.
\end{proposition}
\dteal{Notice that Proposition \ref{prop:prob of existence of core} together with Lemma \ref{lem:existence of a core} yields the required bound \eqref{eq:ubd-with-mathfrak-f}.} 

\dteal{The proof of Proposition \ref{prop:prob of existence of core} needs several ingerdients, is long, and will be done in stages through the remainder of Section \ref{sec:loc-regime-ubd}. As indicated in Section \ref{sec:proof-outline} a key is to bound the number of subgraphs with certain specified properties. We begin with the required bounds for core graphs.} We formulate the required counting estimate for a slightly more general class of graphs. 
     

Let $Q\coloneqq Q_{n,p, \varepsilon}$ and $T\coloneqq T_{n,p, \varepsilon}$ be positive parameters, and define
\begin{equation}\label{eq:mathfrak-G}
    \mathfrak{G}(Q,T)\coloneqq \{G\in \widehat{\Xd}:\ e(G)\leq T,\ \tu{and }\ \tu{deg}_{G}(u)\cdot \tu{deg}_{G}(v)\geq Q,\ \forall (u,v)\in E(G)\}.
\end{equation}
    Given another parameter $R\coloneqq R_{n,p,\varepsilon}\geq 1$, define the set of low degree vertices
    \begin{equation}
        \label{eq:def of W_R(G)}
        \W_{R}(G)\coloneqq \{u\in V(G):\ \tu{deg}_{G}(u)< R\}.
    \end{equation}
 
The following lemma bounds the number of graphs in $\mathfrak G(Q,T)$ after fixing \dteal{the} number of edges and the size of their low degree vertex set.

\begin{lemma}
    \label{lem:entropic stability}
    Fix $\varepsilon\in (0,1)$ and assume $p=o(1)$. Let $Q,T,R>0$ depend only on $n,p$ and $\varepsilon$. Fix $\bm{v}\in \set{n}\cup \{0\}$ and $\bm{e}\in \set{\lfloor T \rfloor}$. Suppose that the following three conditions hold:
        \begin{equation}
            \label{eq:three condition}
            \frac{T}{Q}\cdot \max\{R\mathds{1}_{\bm{v}>0}, 1\}\cdot\log \frac{d}{R}\leq p^{-\varepsilon/3},\ n^{1/R}\leq p^{-\varepsilon/16}, \ \tu{ and }\ \dteal{E_{1,1}(G):=}\{(u,v)\in E(G):u,v\in \W_{R}(G)\} = \emptyset. 
        \end{equation}
\dteal{Further assume that $R \ll d$.} Then, for all sufficiently large $n$,
        \begin{equation}
            \label{eq:entropic stability}
            \left|\{G\in \mathfrak{G}(Q,T):\ |\mathcal{W}_{R}(G)|=\bm{v} \tu{ and } e(G)=\bm{e}\}\right| \leq {n\choose \bm{v}}\cdot \exp(\varepsilon\bm{e}\log(1/p) ).
    \end{equation}
\end{lemma}
    
\begin{proof}[Proof of Lemma \ref{lem:entropic stability}]
Let $G\in \mathfrak{G}_{Q,T}$ with $|\W_{R}(G)| = \bm{v}$ and $e(G) = \bm{e}$.
We begin by defining a collection of subsets of the vertex set $V(G)$: for each $k\in \{0,1,2, \ldots, \lfloor \log_{2}(d/R) \rfloor\}$, let
\begin{equation*}
    B_{k} \coloneqq \left\{u\in V(G):\ \tu{deg}_{G}(u)\geq 2^{k}\cdot R\right\}.
\end{equation*}
For ease of notation, set $k_{*} := \lfloor \log_{2}(d/R)\rfloor$ and $B_{-1} = V(G)$. We keep in note that $\W_{R}(G) = B_{-1}\setminus B_{0}$. This collection of $B_{i}$'s forms a nested sequence: 
\begin{equation*}
    \emptyset=B_{k_{*}+1}\subset B_{k_{*}}\subset \cdots \subset B_{1}\subset B_{0} \subset B_{-1},
\end{equation*}
\dteal{where the leftmost equality is due to the upper bound on the degrees and the choice of $k_*$.}	

\dteal{Next, for $i\in \{0,1,\ldots, k_{*}+1\}$ we let $\widehat{i}\coloneqq \widehat{i}(i)$ to be the smallest integer satisfying 
\begin{equation}
    \label{eq:i hat bound}
    2^{i+\widehat{i}+1}\cdot R^{2}\geq Q.
\end{equation}
Set $\widehat{j}\coloneqq \widehat{j}(i)\coloneqq \max\{\widehat{i}, 0\}$. We claim that the end points of any edge $(u,v) \in E(G)$ must belong to
\[
\mathfrak{M}  := \bigcup_{i=0}^{k_{*}+1} {\left((B_{i-1}\setminus B_{i})\times  B_{\widehat{j}(i)}\right)}.
\]
} 

\dteal{To see this, fix any 	
$u\in B_{i-1}\setminus B_{i}$ for some $i\in \{0,1,\ldots, k_{*}+1\}$. Let us assume that $\widehat i <0$. Consider any edge $(u,v)$ that is incident to $u$. If $v \in B_0 = B_{\widehat j(i)}$ then there is nothing to be proved. If $v \in B_{-1}\setminus B_0$ then as $E_{1,1}(G)=\emptyset$ we must have that $i \geq 1$. Hence, also in this case the end points of the edge $(u,v)$ belongs to $(B_{-1}\setminus B_0) \times B_0 \subseteq \mathfrak{M}$.
} 

\dteal{Now let us focus on the case $\widehat i \geq 0$. So, in this case $\widehat j =\widehat i$.} By definition, 
\dteal{As $u \in B_{i-1}\setminus B_i$, we find} $\tu{deg}_{G}(u)< 2^{i} R$. 
\dteal{Suppose} $v\notin B_{\widehat{j}}$. Then
\[\tu{deg}_{G}(u)\cdot  \tu{deg}_{G}(v) < 2^{i+\widehat{i}} R^{2}< Q,\] 
where the last inequality follows from the minimality of $\widehat{i}$. However, this contradicts our assumption that $G\in \mathfrak{G}_{Q,T}$. \dteal{Therefore, we have the claim.}

We now 
\dteal{proceed to} bound the number of graphs $G \in \mathfrak{G}_{Q,T}$ with $|\W_{R}(G)|=\bm{v}$ and  $e(G)=\bm{e}$ by a simple process:
\begin{itemize}
    \item[(1)] Choose the vertex sets $\{B_{i-1}\setminus B_{i}\}_{i=0}^{k_{*}+1}$.
    \item[(2)] Choose $\bm{e}$ edges from the collection of possible \dteal{end points, as given by $\mathfrak{M}$}:
    
\end{itemize}

As $|B_{-1}\setminus B_{0}| = {\bm v}$, there are ${n \choose \bm{v}}$ ways to choose its vertices. 
\dteal{On the other hand,} for $i\in \{0, \ldots, k_{*}\}$, the size of $B_{i}$ can be bounded as follows: 
\begin{equation}
    \label{eq:bound on subset of vertices}
    |B_{i}|\leq 2e(G)\cdot 2^{-i} R^{-1}. 
\end{equation}
\dteal{Therefore,} the number of choices of vertices in $\{B_{i-1}\setminus B_{i}\}_{i=1}^{k_{*}+1}$ \dteal{is bounded above by}
\begin{equation}
    \label{eq:bound on choice of vertices}
    \prod_{i=1}^{k_{*}+1}\left(\sum_{w_{i}=0}^{\lfloor 2^{-i+2}R^{-1}\cdot \bm{e}\rfloor}n^{w_{i}}\right)\leq \prod_{i=1}^{k_{*}+1} n^{8 R^{-1} \cdot 2^{-i}\cdot \bm{e}}\leq n^{8 R^{-1}\cdot \bm{e}}\leq p^{-(\varepsilon/2)\cdot \dteal{{\bm e}}},
\end{equation}
where in the last step we used the second assumption in \eqref{eq:three condition}.

We now bound the cardinality of $\mathfrak{M}$. We first consider the collection when $i=0$. Using the bounds $|B_{-1}\setminus B_{0}| = \bm{v}\leq \mathds{1}_{\dteal{{\bm v}}>0}\cdot \dteal{2} \bm{e}$, $\bm{e}\leq T$ and \eqref{eq:i hat bound}-\eqref{eq:bound on subset of vertices}, we obtain
\begin{equation}
    \label{eq:edge cardinal 1}
    |B_{-1}\setminus B_{0}| \cdot |B_{\widehat{j}(0)}| = \bm{v}\cdot  |B_{\widehat{j}(0)}| \leq \bm{v}\cdot \frac{2\bm{e}}{2^{\widehat{j}(0)}R}\leq \frac{\dteal{8} TR }{Q}\cdot \mathds{1}_{\bm{v}>0}\cdot  \bm{e} .  
\end{equation}
Again using \eqref{eq:i hat bound}-\eqref{eq:bound on subset of vertices} we deduce that for any $i\in \{1,\ldots,k_{*}+1\}$,
\begin{equation}
    \label{eq:edge cardinal 2}
    |B_{i-1}\setminus B_{i}|\cdot |B_{\widehat{j}(i)}| \leq |B_{i-1}|\cdot |B_{\widehat{j}(i)}|\leq \frac{4 \bm{e}^{2}} { 2^{i+\widehat{j}(i)-1}R^{2}} \leq \frac{16 T}{Q}\cdot \bm{e}.
\end{equation} 
Summing the bounds obtained in \eqref{eq:edge cardinal 1}-\eqref{eq:edge cardinal 2}, we get
\begin{equation}
    \label{eq:bound on set of edges}
    |\mathfrak{M}|\leq (k_{*}+1)\cdot \frac{16T }{Q} \max\{R\mathds{1}_{\bm{v}>0},1\}\cdot \bm{e}\leq \frac{32 T}{Q}\cdot \max\{R\mathds{1}_{\bm{v}>0},1\}\cdot \log\frac{d}{R}\cdot \bm{e}\leq 32\cdot  p^{-\varepsilon/3}\cdot \bm{e},
\end{equation}
where \dteal{in the penultimate step we used $d \gg R$ and} in the final step we used the definition of $k_{*}$ and the \dteal{left most}  assumption in \eqref{eq:three condition}.

Following the scheme as in steps (1)-(2) and applying the bounds in \eqref{eq:bound on choice of vertices} and \eqref{eq:bound on set of edges}, we derive
\begin{equation*}
\begin{aligned}
    |\{G\in \mathfrak{G}_{Q,T}:|\W(G)|= \bm{v}\tu{ and }e(G)= \bm{e}\}| \leq \dteal{{n \choose {\bm v}}} \cdot p^{-(\varepsilon/2)\cdot \bm{e}}\cdot {|\mathfrak{M}|\choose \bm{e}} & \leq  \dteal{{n \choose {\bm v}}} \cdot p^{-(\varepsilon/2)\bm{e}}\cdot (32ep^{-\varepsilon/3})^{\bm{e}}\\
    & \leq \dteal{{n \choose {\bm v}}} \cdot p^{-\varepsilon \cdot \bm{e}},
    \end{aligned}
\end{equation*}
where in the penultimate step we use the inequality ${|\mathfrak{M}|\choose \bm{e}}\leq (e|\mathfrak{M}|/\bm{e})^{\bm{e}}$ and in the last step we use $p=o(1)$. Hence, the proof is complete.
\end{proof}

\dteal{Next, we record the following useful implication of Lemma \ref{lem:entropic stability}.} 


\begin{corollary}
    \label{cor:entropic stability 2}
\dteal{Fix $\varepsilon\in (0,1)$ and assume $p=o(1)$. Let $Q,T,R>0$ depend only on $n,p$ and $\varepsilon$. 
Suppose
        \begin{equation}
            \label{eq:three condition 2}
             R < Q/d, \quad R \ll d, \quad \frac{T}{Q}\cdot \log \frac{d}{R}\leq p^{-\varepsilon/3},\quad \tu{and }\quad n^{1/R}\leq p^{-\varepsilon/16}. 
        \end{equation} 
Then, for any $\bm{e}\in \set{\lfloor T \rfloor}$ and all sufficiently large $n$,}
        \begin{equation}
            \label{eq:true entropic stability}
            |\{G\in \mathfrak{G}(Q,T):\ e(G)= \bm{e}\}|\leq \exp(\varepsilon \bm{e}\log (1/p)).
    \end{equation}
\end{corollary}

\begin{proof}
\dteal{It is a straightforward consequence of the degree bound and Lemma \ref{lem:entropic stability}. Indeed, for 
any $G\in \mathfrak{G}(Q,T)$ and any edge $(u,v)\in E(G)$, as the degree is upper bounded by $d$, we have
\[
\min\{\deg_G(u), \tu{deg}_{G}(v)\} \geq \frac{1}{d}\cdot \tu{deg}_{G}(u)\cdot \tu{deg}_{G}(v)\geq \frac{Q}{d}.
\]
As $R < Q/d$ the above implies that 
$\W_{R}(G)= \emptyset$. Hence, upon applying Lemma \ref{lem:entropic stability} the required upper bound follows.}
\end{proof}

\begin{remark}
\dteal{Let us observe that to apply either Lemma \ref{lem:entropic stability} or Corollary \ref{cor:entropic stability 2}, as $d=np\geq 1$, the condition $n^{1/R}\leq p^{-\varepsilon/16}$ implies $R\geq 16/\varepsilon$. Hence, to apply any of these two results we need this lower bound on $R$.}
\end{remark}

\begin{remark}
\dteal{Following \cite{HMS}, we refer to a family of subgraphs satisfying a bound of the form \eqref{eq:true entropic stability} as entropically stable. The rationale is that, for graphs with $\bm{e}$ edges, the logarithm of the number of such admissible configurations is at most $\varepsilon \bm{e}\log(1/p)$, whereas, thanks to the near independence of the edges, the cost (i.e.~negative of the $\log$-probability) of realizing any particular configuration is roughly $\bm{e}\log(1/p)$. 
This is precisely what allows a union bound over all possible cores while preserving the correct leading order asymptotics. As already indicated in Section \ref{sec:proof-outline}, and will be seen below, such a bound does {\em not} hold for the regime $np^{\Delta -1} \lesssim (\log n)^{v_H}$.}
\end{remark}
%
%


\dteal{To initiate the proof of Proposition \ref{prop:prob of existence of core}  we need two more ingreadients. First, we need} a deterministic estimate showing that the rate function \dteal{$\Psi_{H}$ (recall \eqref{eq:def-psi}) determines} the minimum number of edges required to support a prescribed number of copies of $H$. \dteal{This will provide a sharp lower bound on the minimum number of edges a core graph must possess. It is also worth highlighting that this lower bound matches with the corresponding bound in the Erd\H{o}s-R\'enyi setting for $\Delta \geq 3$, while it is strictly larger for $\Delta =2$ and $\delta >1$.} 


\begin{lemma}
\label{lem:minimum edge bound}
    Let $H$ be a connected $\Delta$-regular graph, and suppose that $G\in \widehat{\Xd}$ satisfies $N(H, G)\geq \gamma n^{v_{H}}p^{e_{H}}$ for some $\gamma>0$. Then 
    \begin{equation*}
        e(G)\geq \Psi_{H}(\gamma)\cdot n^{2}p^{\Delta}. 
    \end{equation*}
\end{lemma}
\begin{proof}
    First, let us consider $\Delta\geq 3$. \dteal{The graph} $H$ being regular implies that $\alpha_{H}^{*}= v_{H}/2$. Therefore, from \cite[Theorem~5.4]{HMS}, we have $N(H, G)\leq (2e(G))^{v_{H}/2}$. Combining this with the assumed lower bound on $N(H,G)$ yields the desired bound on $e(G)$.

    Now, suppose that $\Delta=2$. Since $H$ is connected, $H=C_{\ell}$, where $\ell=v_{H}=e_{H}$. Let \dteal{$A:=A(G)$} denote the adjacency matrix of $G$. \dteal{Since $A$ is  a symmetric  matrix with $\{0,1\}$-valued entries with each row sum at most $d$, by Perron-Frobenius theorem all its eigenvalues $\{\widetilde \lambda_{i}(G)\}$} 
are  in $[-d,d]$. Moreover, \dteal{$2e(G) = \tu{Trace}(A^2) = \sum_{i=1}^{n}\widetilde\lambda_{i}(G)^2$}, while
    \begin{equation*}
        \gamma d^{\ell}\ \leq\ N(C_{\ell}, G)\ \leq\ \tu{Hom}(C_{\ell}, G)\ =\ \dteal{\tu{Trace}(A^\ell)\ =\ \sum_{i=1}^{n} \widetilde \lambda_{i}(G)^{\ell}.}
    \end{equation*}
    Setting \dteal{$\eta_{i}:=\widetilde\lambda_{i}(G)/d$}, \dteal{it follows from the discussion above that}
    \begin{equation}
        \label{eq:optimization}
        e(G)\geq \frac{d^2}{2}\cdot \inf\left\{\sum_{i=1}^{n}\eta_{i}^{2}:\ |\eta_{i}|\leq 1,\ \sum_{i=1}^{n}\eta_{i}^{\ell}\geq \gamma \right\},
    \end{equation}
\dteal{for any $G$ under consideration. It is easy to note that the} optimal $\{\eta_{i}\}$ in \eqref{eq:optimization} may be taken to be nonnegative. Now, the claimed bound follows directly from \cite[Lemma~2.2]{BD} applied to optimization problem \eqref{eq:optimization}.
\end{proof}

\dteal{Next we need} a uniform lower bound on the product of the degrees at the endpoints of each core edge. Condition (C3) provides precisely the information needed to obtain such a bound.	
\begin{lemma}
    \label{lem:product of degrees}
Assume $G_\star$ is a core and let $(u,v)\in E(G_\star)$. Then 
        \begin{equation*}
            \tu{deg}_{G_\star}(u)\cdot \tu{deg}_{G_\star}(v)	\geq \frac{\widetilde{c}_{0}\cdot n^{2}p^{\Delta}}{\left(\log(1/p)\right)^{v_{H}}},
        \end{equation*}
        for some constant $\widetilde{c}_{0}\coloneqq \widetilde{c}_{0}(H, \varepsilon, \delta)>0$.
\end{lemma}
\dteal{The proof of Lemma \ref{lem:product of degrees} being identical to that of \cite[Lemma~3.6]{ABRB} the details are omitted.} 


\begin{proof}[Proof of Proposition \ref{prop:prob of existence of core} when $np^{\Delta-1}\gg (\log n)^{v_{H}}$]
\dteal{We will show that
\begin{equation}
\label{eq:entropic stability of cores large p}
|\{G\in \widehat{\Xd}:\ G\tu{ is a core with }e(G) = \bm{e}\}|\leq \exp(\varepsilon \bm{e}\log(1/p)).
\end{equation}
To prove \eqref{eq:entropic stability of cores large p} we will apply Corollary \ref{cor:entropic stability 2}. Let us first introduce some notation:
\begin{equation}
    \label{eq:def Q,T core}
    Q_{\core} := \widetilde{c}_{0} n^{2}p^{\Delta}\cdot (\log(1/p))^{-v_{H}}\quad  \tu{and} \quad T_{\core} := \cbar n^{2}p^{\Delta}\log (1/p).
\end{equation}
Recall \eqref{eq:mathfrak-G}. By (C2) and Lemma \ref{lem:product of degrees}, every core graph $G$ belongs to $\mathfrak{G}(Q_{\core},T_{\core})$. Hence, Corollary \ref{cor:entropic stability 2} applies with $Q = Q_{\core}$, $T=T_{\core}$ and $R = \frac{Q_{\core}}{2d}$ if its hypotheses are met with these parameters.

By definition, $R<Q/d$. As $p \ll 1$ we have $R \ll d$. Further, notice that, for sufficiently large $n$,
    \[
\frac{T}{Q}\log \frac{d}{R}= \frac{T_{\core}}{Q_{\core}}\cdot \log \frac{d^2}{Q_{\core}} \lesssim (\log(1/p))^{v_{H}+2} \leq p^{-\varepsilon/3}.
\]
Moreover, using $np^{\Delta-1}\gg (\log n)^{v_{H}}$, we have
\[\log n^{2d/Q_{\core}}\lesssim  \frac{(\log n)^{v_H}}{np^{\Delta-1}} \cdot \log(1/p) \leq  \frac{\varepsilon}{16} \log(1/p).\]
Thus, all the hypotheses of Corollary \ref{cor:entropic stability 2} are met and so we have \eqref{eq:entropic stability of cores large p}.} 
    
\dteal{On the other hand, applying} Corollary \ref{cor:almost independence lemma(simple)} iteratively for each of its $e(G)$ many edges, we obtain 
    \begin{equation}\label{eq:single core probability}
        \P(G\subseteq \Gd)\leq p^{(1-o(1))e(G)}.
    \end{equation}
%
%
To complete the proof we note that by (C1) and Lemma \ref{lem:minimum edge bound}, for any core graph $G$ we have 
%
    \begin{equation}
        \label{eq:def of T_0}
        e(G)\ \geq\  n^{2}p^{\Delta} \cdot \Psi_{H}(\delta(1-3\varepsilon))\ \geq\ (1- \dteal{\mathfrak{g}}_{H}(\varepsilon)) \cdot \Psi_{H}(\delta) \cdot n^{2}p^{\Delta}\eqqcolon T_0,
    \end{equation}
where the second inequality follows from the  continuity of the function $\Psi_H(\cdot)$ and $\mathfrak{g}_{H}(\cdot)$ is some nonnegative function such that $\lim_{\varepsilon \downarrow 0}\mathfrak{g}_{H}(\varepsilon)=0$. \dteal{Finally} combining \eqref{eq:entropic stability of cores large p} with \eqref{eq:single core probability},
    \begin{equation}
            \dteal{\P(\Gd \tu{ contains a core}) \leq \sum_{\bm{e}\geq \lceil  T_{0}\rceil} p^{(1-o(1))\bm{e}} \cdot|\{G \in \widehat{\Xd}: G\tu{ is a core with }e(G)=\bm{e}\}|\leq 
p^{(1-2\varepsilon)T_0}.}
    \end{equation}
This finishes the proof.		
\end{proof}		
%
It remains to treat the complementary, sparser regime $np^{\Delta-1}\lesssim(\log n)^{v_H}$. In this range, $Q_{\core}/d$ need not diverge, so there may be no choice of $R$ satisfying simultaneously
$R<\frac{Q_{\core}}{d}$ and the hypotheses of Corollary \ref{cor:entropic stability 2}. Consequently, the low-degree set $\W_R(G)$ of a core may be nonempty, and the additional entropy associated with choosing these vertices can no longer be discarded. Thus the direct entropic-stability argument above does not by itself recover the correct upper tail exponent. We overcome this by first showing that cores with sufficiently many edges make an asymptotically negligible contribution to the upper-tail probability. This allows us to restrict attention to a more structured class of graphs, which we call strong cores, for which the remaining entropy can be controlled more sharply. The next section develops this refinement. As the following \dteal{set of} arguments continue to hold in a bigger regime of $p$, for sake of simplicity, we assume $np^{\Delta}\leq 1$.

\subsection{Reduction to strong cores when $np^{\Delta}\leq 1$}
To make the above idea precise, fix a sufficiently large constant $\cstar\coloneqq \cstar(\delta,H)>0$. The first lemma shows that cores above the cutoff $\cstar n^2p^\Delta$ have a negligible probability compared with the upper tail scale.
\begin{lemma}
    \label{lem:SC reduction step}
    Assume $(\log n)^{1/(v_{H}-2)}\ll np^{\Delta/2}\leq n^{1/2}$. For sufficiently large $n$, we have
    \[\P(\Gd \tu{ contains a core } G\tu{ with }e(G)\geq \cstar n^{2}p^{\Delta})\leq \exp\left(-\frac{\cstar}{(4\Delta)^{\dteal{4} v_{H}}} n^{2}p^{\Delta}\log(1/p)\right).\]
\end{lemma}

Compared to a core, we define a `strong core' to be a smaller \dteal{sub class of graphs} which \dteal{satisfy} a stronger local participation condition on every edge.

\begin{definition}[Strong core]
A graph $G\in \widehat{\Xd}$ is called a strong core if the \dteal{followings hold}:  
\begin{itemize}
    \item[(SC1)]$N(H, G)\geq \delta(1-4\varepsilon) \cdot n^{v_{H}}p^{e_{H}}$,
    \item[(SC2)]$e(G)\leq \cstar n^{2}p^{\Delta}$,
\end{itemize}
and,
\begin{itemize}
    \item[(SC3)]$\min_{e\in E(G)}N(H, G, e)\geq (\delta \varepsilon/\cstar) (np^{\Delta/2})^{v_{H}-2}$.
\end{itemize}
\end{definition}

\begin{lemma}
\label{lem:strong core main lemma}
Fix $\delta>0$ and a $\Delta$-regular connected graph $H$ with $\Delta\geq 2$. Then, for any $p\in (0,1)$ satisfying $(\log n)^{1/(v_{H}-2)}\ll np^{\Delta/2}\leq n^{1/2}$, and $\varepsilon\in (0,1)$ sufficiently small, we have
\[\liminf_{n\rightarrow \infty}-\frac{\log \P(\Gd \tu{ contains a strong core graph }G)}{n^{2}p^{\Delta}\log (1/p)}\geq \Psi_{H}(\delta)\cdot(1-\mathfrak{f}_{H}(\varepsilon)),\]
for some non-negative function $\mathfrak{f}_{H}(\cdot)$ such that $\lim_{\varepsilon\downarrow 0}\mathfrak{f}_{H}(\varepsilon) = 0$.
\end{lemma}

Assuming \dteal{the last two lemmas}, we can now complete the proof of Proposition \ref{prop:prob of existence of core}.

\begin{proof}[Proof of Proposition \ref{prop:prob of existence of core} when $(\log n)^{1/(v_{H}-2)}\ll np^{\Delta/2}\leq n^{1/2}$]
We choose $\cstar>(4\Delta)^{\dteal{4} v_H}\cdot 2\Psi_{H}(\delta)$. Then Lemma \ref{lem:SC reduction step} gives us 
\begin{equation}
    \label{eq:core with many edges simple}
    \P(\Gd \tu{ contains a core $G$ with }e(G)\geq \cstar n^{2}p^{\Delta})\leq \exp(-2\Psi_{H}(\delta)n^{2}p^{\Delta}\log(1/p)).
\end{equation}
Next, we claim that any core graph $G$ with $e(G)\leq \cstar n^{2}p^{\Delta}$, contains a strong core subgraph \dteal{$G^{\prime}\subseteq G$}. \dteal{The proof follows a pruning procedure similar to the one used to obtain a core subgraph of a pre-seed graph (see the proof of Lemma \ref{lem:existence of a core}).} 
%
Hence, it follows from the claim that
\begin{equation}
    \label{eq:core contains strong-core}
    \P(\Gd\tu{ contains a core graph $G$ with $e(G)\leq \cstar n^{2}p^{\Delta}$})
\leq \P(\Gd \tu{ contains a strong core graph $G$}).
\end{equation}
\dteal{Applying} Lemma \ref{lem:strong core main lemma} \dteal{to upper bound the RHS of} \eqref{eq:core contains strong-core}, and using \eqref{eq:core with many edges simple}, we obtain the required bound.  
\end{proof}

\subsection{Proof of Lemma \ref{lem:SC reduction step}}
We begin by recalling the definition of $\W_{R}(G)$ from \eqref{eq:def of W_R(G)} and \dteal{of} $T_{\core}$ \dteal{and} $Q_{\core}$ \dteal{from} \eqref{eq:def Q,T core}. We also set up some \dteal{notation that are relevant for the next couple of subsections.} 
Consider the following edge sets:
\begin{equation}
    \label{eq:def of edge sets}
    E_{1,1}(G) \coloneqq \{(u,v)\in E(G): u,v\in \W_{R}(G)\}, \quad E_{2,2}(G)\coloneqq \{(u,v)\in E(G): u,v\notin \dteal{\mathcal W}_{R}(G)\},
\end{equation}
and,
\begin{equation}
    \label{eq:def of edge sets 2}
    E_{1,2}(G)\coloneqq \{(u,v)\in E(G): \tu{one of }u \tu{ and }v \tu{ is in }\W_{R}(G)\}.
\end{equation}
Further, we write 
\begin{equation}\label{eq:e-112}
\dteal{e_{1,1}(G) \coloneqq |E_{1,1}(G)|, \quad e_{1,2}(G) \coloneqq |E_{1,2}(G)|, \quad \text{ and } \quad e_{2,2}(G) \coloneqq |E_{2,2}(G)|.}
\end{equation} 

\subsubsection{Sparsity sub regime $n^{2}p^{\Delta}\geq (\log n)^{2v_{H}}$} \dteal{The following result yields Lemma \ref{lem:SC reduction step} in the sub regime $n^2p^\Delta \geq (\log n)^{2v_H}$.}



\begin{lemma}
    \label{lem:core graph with many edges large p}
If $n^{2}p^{\Delta}\geq \dteal{(\log n)^{2v_{H}}}$ \dteal{ and $n p^\Delta \leq 1$}, then for any $\varepsilon\in (0,\frac{1}{8})$, 
    \[
    \dteal{\P\left(\bigcup_{e\geq \overline{C}_*n^2p^\Delta}\left\{\exists G\subseteq\Gd:G\tu{ is a core graph} \tu{ and }e(G)= \bm{e}\right\}\right)}  \leq  \exp\left(-\frac{\cstar}{16} n^{2}p^{\Delta}\log (1/p)\right).\]
\end{lemma}

\dteal{The proof of Lemma \ref{lem:core graph with many edges large p} follows a strategy similar to that used in the proof of \cite[Lemma 3.5]{ABRB}.} 
\begin{proof}[Proof of Lemma \ref{lem:core graph with many edges large p}]
\dteal{Set $R=16\Delta/\varepsilon$. We claim that, with this choice of $R$, $E_{1,1}(G)=\emptyset$ for any core graph $G$ under consideration.  To see this, fix 
any $(u,v)\in E(G)$ with $u\in\W_R(G)$. Recall \eqref{eq:def Q,T core}. By Lemma \ref{lem:product of degrees} we have
\[
\tu{deg}_{G}(v)\geq \frac{1}{R}\tu{deg}_{G}(u)\cdot \tu{deg}_{G}(v)\geq \frac{Q_{\core}}{R}\gg R,
\]
where the last step follows from our assumption $n^{2}p^{\Delta}\geq (\log n)^{2v_{H}}$. Thus $v\notin\W_R(G)$ and hence the claim. 

As $E_{1,1}(G)=\emptyset$ and $np^\Delta \leq 1$, applying Lemma \ref{lem:entropic stability} with 
with $T=T_{\core}$, $Q=Q_{\core}$, and $R= 16\Delta/\varepsilon$ (recall \eqref{eq:def Q,T core} and notice that all the hypotheses of Lemma \ref{lem:entropic stability} are met with these parameters) we find that 
\begin{equation}\label{eq:core-large-p-large-1}
|\{G\in \widehat{\Xd}: G\tu{ is a core, } |\W_{R}(G)| = \bm{v}, \tu{ and } e(G)=\bm{e} \}|\leq {n\choose v} \exp(\varepsilon \bm{e}\log(1/p)).
\end{equation}
Further, as the number of edges in a core graph $G$ is $o(nd)$, applying Corollary \ref{cor:almost independence lemma(simple)} iteratively it follows that 
\begin{equation}\label{eq:core-large-p-large-2}
\P(G\subseteq \Gd)\leq p^{(1-o(1))e(G)}, \qquad \text{ for any core graph } G,
\end{equation}
which differs from the corresponding probability in the Erd\H{o}s-R\'enyi setting by the negligible factor $p^{o(e(G))}$. Since, the proof of \cite[Lemma 3.5]{ABRB} only relies on the bound \eqref{eq:core-large-p-large-1} and the analog of \eqref{eq:core-large-p-large-2} for the Erd\H{o}s-R\'enyi setting, one can carry out the same approach here. We do not repeat it here. This completes the proof.} 
\end{proof}	

\subsubsection{Sparsity sub regime $n^{2}p^{\Delta}\leq (\log n)^{2v_{H}}$}
\dteal{Let us remind the reader that to prove an analog of Lemma \ref{lem:core graph with many edges large p} in the current sub regime, the approach taken in \cite{ABRB} was quite involved. One can indeed repeat the same argument. We provide an argument below that is much simpler and shorter. 

The crucial observation behind this simpler route is that the condition $E_{1.1}(G)=\emptyset$ used in the proof of Lemma \ref{lem:core graph with many edges large p} is not essential. One can essentially repeat the same argument  under the weaker assumption $e_{1,2}(G) + e_{2,2}(G) \gtrsim e(G)$ (recall \eqref{eq:e-112}), only at the cost of increasing $\cstar$ in Lemma \ref{lem:core graph with many edges large p}. Below we show that such a bound is possible by choosing $R=4\Delta$ in \eqref{eq:def of edge sets}-\eqref{eq:def of edge sets 2}, and then using that bound we complete the proof of Lemma \ref{lem:SC reduction step}. 
}
\begin{lemma}
\label{lem:control on E_11}
\dteal{Let $G$ be a core graph and set $R=4\Delta$. 
Then for $p\in (0,1)$ satisfying $np^{\Delta/2}\gg (\log n)^{1/(v_{H}-2)}$, we have}
\[e(G)\leq (4\Delta)^{2v_{H}} \cdot (e_{1,2}(G)+e_{2,2}(G)).\]
\end{lemma}
We defer its proof and first complete the proof of Lemma \ref{lem:SC reduction step}. 
\begin{proof}[Proof of Lemma \ref{lem:SC reduction step} when $(\log n)^{1/(v_H-2)}\ll np^{\Delta/2}\leq (\log n)^{v_H}$]
As $1\leq n^{2}p^{\Delta}\leq (\log n)^{2v_H}$, any core graph $G$ has $|V(G)|\leq e(G)\leq T_{\core}\leq  (\log n)^{3v_{H}}$, for large enough $n$ \dteal{(recall \eqref{eq:def Q,T core})}. Fixing some $\bm{v}\leq \bm{e}\leq (\log n)^{3v_H}$, we naturally bound the number of possible core graphs $G$ with $|V(G)|=\bm{v}$ and $e(G)= \bm{e}$ by first choosing the $\bm{v}$ vertices in at most $n^{\bm{v}}$ ways and then choosing $\bm{e}$ edges in at most ${\bm{v}^{2}\choose \bm{e}}$ ways. 
\dteal{As $n^{2}p^{\Delta}\leq (\log n)^{2v_H}$ and $\bm{v}\leq \bm{e}\leq (\log n)^{3v_{H}}$, we further get} 
\[
n^{\bm{v}}\leq \exp\left(\frac{\Delta}{2}\bm{v}\log (1/p)+v_{H}\bm{e}\log \log n\right) \quad  \text{ and } \quad {\bm{v}^{2}\choose \bm{e}}\leq \bm{v}^{2\bm{e}}\leq \exp(6v_H \bm{e}\log \log n). 
\]
Since $\log\log n =o(\log (1/p))$ in this regime, \dteal{we deduce from above} 
\begin{equation}
\label{eq:core bound small p}
|\{G\in \widehat{\Xd}: G\tu{ is a core with }|V(G)|=\bm{v}\tu{ and }e(G)=\bm{e} \}|\leq \exp\left(\left(\frac{\Delta}{2}\cdot \frac{\bm{v}}{\bm{e}}+o(1)\right) \bm{e} \log (1/p)\right).
\end{equation}
\dteal{To complete the proof we will argue that the RHS of \eqref{eq:core bound small p} is bounded above by $(1-\bar c) {\bm e} \log(1/p)$, for $\bar c = 1/(4 \cdot R^{3v_H})$. 
This will then allow to us to use the union bound over the set of all core graphs with $e(G) \geq \cstar n^2p^\Delta$.}

\dteal{Turning to prove the above}, we observe that any edge incident to some vertex in $V(G)\setminus \W_{R}(G)$ must either be in $E_{1,2}(G)$ or $E_{2,2}(G)$. \dteal{As any vertex in $V(G)\setminus \W_{R}(G)$ has degree at least $R= 4\Delta$, we obtain}
\[|V(G)\setminus \W_{R}(G)|\leq (2/R)\cdot(e_{1,2}(G)+e_{2,2}(G))\leq (1/2\Delta)\cdot(e_{1,2}(G)+e_{2,2}(G)).\]

On the other hand, for any core graph $G$, condition (C3) and $np^{\Delta/2}\gg (\log n)^{1/(v_{H}-2)}$ imply that each edge of $G$ participate in at least one copy of $H$. Subsequently, $H$ being a $\Delta$-regular graph implies that $\min_{u\in V(G)}\tu{deg}_{G}(u)\geq \Delta$. Now, from definition of $E_{1,1}(G)$ and $E_{1,2}(G)$ it follows that,
\[|\W_{R}(G)|\leq \frac{1}{\Delta} (2e_{1,1}(G)+e_{1,2}(G)).\]
Using \dteal{the last two bounds} and the identity $e(G) = e_{1,1}(G)+e_{1,2}(G)+e_{2,2}(G)$, we obtain
\begin{equation}
\label{eq:v/e bound small p}
\frac{\Delta}{2}|V(G)| \leq e_{1,1}(G) +\frac{3}{4}(e_{1,2}(G)+e_{2,2}(G)) =e(G)-\frac{1}{4}(e_{1,2}(G)+e_{2,2}(G)) \leq \left(1-\frac{1}{4\cdot (4\Delta)^{2v_H}}\right)e(G),
\end{equation}
where the last inequality follows from Lemma \ref{lem:control on E_11}. 

\dteal{Now}, using observation \eqref{eq:v/e bound small p} with the fixed core probability bound \eqref{eq:single core probability} and the counting estimate \eqref{eq:core bound small p}, we have
\[\P(G\subseteq \Gd: G \tu{ is a core with }|V(G)|=\bm{v},\ e(G)=\bm{e})\leq \exp\left(-(4\Delta)^{-3v_H}\cdot \bm{e}\cdot \log (1/p)\right).\]
Summing over all $\bm{v},\bm{e}$ satisfying $\bm{v}\leq \bm{e}\leq (\log n)^{3v_{H}}$ and \dteal{$\bm{e}\geq \cstar n^{2}p^{\Delta}$}, yield
\[\P(\Gd \tu{ contains a core $G$ with }e(G)\geq \cstar n^{2}p^{\Delta})\leq (\log n)^{6v_{H}}\cdot \exp\left(-(\cstar/(4\Delta)^{3v_{H}}) n^{2}p^{\Delta}\log(1/p)\right).\]
\dteal{Upon using $\log \log n = o(\log (1/p))$, the proof completes.} 
\end{proof}
We now prove Lemma \ref{lem:control on E_11}.
\begin{proof}[Proof of Lemma \ref{lem:control on E_11}]

Let $G_{1,1}\subseteq G$ be the subgraph induced by $E_{1,1}(G)$. We claim that for any edge $(x,y)\in E_{1,1}(G)$, \begin{equation}
    \label{eq:bound on number of copies of H in G_11}
    N(H, G_{1,1}, (x,y))\leq 2e_{H}\cdot R^{v_{H}-2}.\end{equation}
\dteal{To show this, we begin by noting that there are $2e_H$ many ways an edge in $H$ can be mapped to the edge $(x,y)$ (the factor two arises because of the choice of the end point of the chosen edge of $H$ to be mapped to $x$).  
Suppose $(a,b)\in E(H)$ is mapped to $(x,y)$.} 	
Fix $V_0\coloneqq\{a,b\}$ and define for each $i\in \set{v_H-2}$, \[V_{i}\coloneqq \{u\in V(H)\setminus \cup_{j=0}^{i-1}V_j: \exists v\in V_{i-1}\tu{ such that }(u,v)\in E(H)\}.\]
\dteal{Notice that there may exist some $i_0 < v_H-2$ such that $V_i =\emptyset$ for $i \geq i_0$.} Since $H$ is connected, the collection $\{V_i\}$ forms a partition of $V(H)$. \dteal{Next, as} maximum degree of $G_{1,1}$ is $R$, we observe that, recursively for each $i\in \set{v_{H}-2}$, having fixed the choices for vertices in $V_{i-1}$, the vertices in $V_{i}$ can be chosen in at most $R^{|V_{i}|}$ ways. This proves the \dteal{claimed upper bound \eqref{eq:bound on number of copies of H in G_11}.}

\dteal{Now, by our lower bound on $p$, and since $G$ is a core graph (recall (C3)), upon taking $R=4\Delta$ and applying \eqref{eq:bound on number of copies of H in G_11}, we find that, for all sufficiently large $n$, 
\[
N(H, G, (x,y))- N(H, G_{1,1}, (x,y))\gg 1, \qquad \text{ for any edge } (x,y) \in E_{1,1}(G).
\]
This, in particular, implies that, given any edge $(x,y) \in E_{1,1}(G)$,} there exists at least one labelled copy of $H$ which uses the edge $(x,y)$ and some other edge $(x^{\prime}, y^{\prime})\in E_{1,2}(G)$. Thus, for each edge $(x,y)\in E_{1,1}(G)$, we can choose \dteal{(at least)} an edge $(x^{\prime}, y^{\prime})$ such that $x^{\prime}\in \W_{R}(G)$ \dteal{and} $y^{\prime}\in V(G)\setminus \W_{R}(G)$, 
with the additional property that there exists a path $P$ of length at most $v_{H}-1$ consisting entirely of edges from $E_{1,1}(G)$,  whose one endpoint is $x^{\prime}$ and \dteal{the} other endpoint belongs to $\{x,y\}$. \dteal{The last property stems from the fact that both $(x,y)$ and $(x',y')$ belong to a copy of $H$.} 
Set $Y((x,y))\coloneqq (x^{\prime}, y^{\prime})$. \dteal{In case there are more than one choice for the edge $(x',y')$, pick any one of them arbitrarily.} 

Now, given \dteal{an edge} $(x^{\prime}, y^{\prime})\in E_{1,2}(G)$ with $x^{\prime}\in \W_{R}(G)$, there are at most $\sum_{k=0}^{v_{H}-1}R^{k}\leq R^{v_{H}}$ many paths $P$ with length at most $v_{H}-1$ with one endpoint at $x^{\prime}$ and consisting entirely of edges from $E_{1,1}(G)$. \dteal{Therefore, for any $(x',y') \in E_{1,2}(G)$, the number of edges $(x,y) \in E_{1,1}(G)$ such that $Y((x,y))=(x',y')$ is bounded above by $R^{v_H}$.} 
Equipped with this, a double counting argument gives us
\[e_{1,1}(G) 
=\sum_{\substack{(x,y)\in \\  E_{1,1}(G)}}\sum_{\substack{(x^{\prime},y^{\prime})\in \\  E_{1,2}(G)}}\mathds{1}_{Y((x,y))=(x^{\prime}, y^{\prime}))}
\ \dteal{=} \sum_{\substack{(x^{\prime},y^{\prime})\in \\  E_{1,2}(G)}}\sum_{\substack{(x,y)\in \\  E_{1,1}(G)}}\mathds{1}_{Y((x,y))=(x^{\prime}, y^{\prime}))}\leq R^{v_{H}}\cdot e_{1,2}(G).\]
\dteal{Plugging in the above upper bound we derive
\[
e(G)= e_{1,1}(G)+e_{1,2}(G)+e_{2,2}(G) \leq (R^{v_{H}}+1) e_{1,2}(G)+e_{2,2}(G),
\]	
completing the proof.}
\end{proof}

\subsection{Proof of Lemma \ref{lem:strong core main lemma}}\label{sec:pf-lm-513}
\dteal{In this subsection we prove Lemma \ref{lem:strong core main lemma}. We will consider three cases separately: (i) $\Delta =2$, (ii) $\Delta \geq 3$ and $H$ non-bipartite, and (iii) $\Delta \geq 3$ and $H$ bipartite. The proofs of the first two cases are simpler, while case (iii) needs some refinements of the arguments in \cite{ABRB}.}

To obtain sufficiently sharp counting estimates for strong core graphs, we first record some structural properties of such graphs. Following \cite{ABRB}, we decompose a strong core graph into subgraphs consisting of edges with unusually large product \dteal{of the degrees of its end points} and the remaining edges, and isolate a small collection of `bad' edges whose copies of $H$ necessarily interact with the former \dteal{collection of edges}. This decomposition will allow us to retain almost all copies of $H$ while gaining additional control over the structure of the graph.

\begin{definition}[$G_{\high},G_{\low}, G_{\bad}$]
\label{def:low, bad and high}
Given a graph $G$ and a positive constant $C_{0}\in (0,\infty)$, let $G_{\high}$ to be the subgraph of $G$ spanned by edges $(u,v)\in E(G)$ satisfying 
\begin{equation}
    \label{eq:deg prod upper bound}
    \tu{deg}_{G}(u)\cdot \tu{deg}_{G}(v)\geq C_{0}\cdot n^{2}p^{\Delta}.
\end{equation}
Set $G_{\low} \coloneqq G\setminus G_{\high}$. Further, let $G_{\bad}$ denote the subgraph of $G$ spanned by edges that participate in at least one copy of $H$ that also uses an edge from $G_{\high}$. 
\end{definition} 

The following lemma, adapted from \cite[Lemma~4.2]{ABRB}, shows that the degree products of the endpoints of every edge of a strong core are uniformly bounded from below, while edges with exceptionally large degree product form only a small portion of the graph.

\begin{lemma}[{\cite[Lemma~4.2]{ABRB}}]
\label{lem:deg product bound strong core}
Fix a $\Delta$-regular graph $H$ with $\Delta\geq 2$. Let $G$ be a strong core graph. Then for any $\varepsilon>0$, there exist $0<c_{0}(\varepsilon), C_{0}(\varepsilon)<\infty$, such that the \dteal{followings hold}:
\begin{itemize}
    \item[(a)] For every edge $(u,v)\in E(G)$, 
    \begin{equation}
        \label{eq:deg prod lower bound SC}
        \tu{deg}_{G}(u)\cdot \tu{deg}_{G}(v)\geq c_{0}(\varepsilon)\cdot n^{2}p^{\Delta}.
    \end{equation}
    \item[(b)] Having defined $G_{\high}, G_{\low}, G_{\bad}$ with $C_{0}(\varepsilon)$ as in Definition \ref{def:low, bad and high}, we have 
    \begin{equation}
        \label{eq:G-low contains all excess copies}
        N(H, G_{\low})\geq N(H, G)-\varepsilon^{2}\delta n^{v_{H}}p^{e_{H}} \qquad \dteal{\text{ and } \qquad e(G_{\high})\leq e(G_{\bad})\leq \varepsilon^{2} \cdot e(G).}
    \end{equation}
\end{itemize}
\end{lemma}
Compared with \cite[Lemma~4.2]{ABRB}, part (b) is stated here with the slightly stronger error $\varepsilon^2$. This follows by choosing the threshold constant $C_0(\varepsilon)$ sufficiently larger.

\dteal{Next, for ease in witing, set} 
\begin{equation}
\label{eq:def of Q-sc and T-sc}
T_{\sc}\coloneqq \cstar n^{2}p^{\Delta}\quad \tu{and}\quad Q_{\sc} = c_{0}(\varepsilon)\cdot n^{2}p^{\Delta}.
\end{equation}
\dteal{Recall \eqref{eq:mathfrak-G}.} By Lemma \ref{lem:deg product bound strong core}(a) and condition (SC2), 
\begin{equation}
\label{eq:SC in G(Q,T)}
\{G\in \widehat{\Xd}:\ G\tu{ is a strong-core graph}\} \subseteq \mathfrak{G}(Q_{\sc}, T_{\sc}).
\end{equation}
Further, by (SC1) and Lemma \ref{lem:minimum edge bound}, we get
\begin{equation}
\label{eq:def of T_1}
e(G)\ \geq\  n^{2}p^{\Delta} \cdot \Psi_{H}(\delta(1-4\varepsilon))\ \geq\ (1-\mathfrak{f}_{H}(\varepsilon)) \cdot \Psi_{H}(\delta) \cdot n^{2}p^{\Delta}\eqqcolon T_1,
\end{equation}
where $\mathfrak{f}_{H}(\cdot)$ is a nonnegative function  satisfying $\lim_{\varepsilon \downarrow 0}\mathfrak{f}_{H}(\varepsilon)=0$



\dteal{We now prove Lemma \ref{lem:strong core main lemma} for $\Delta=2$. Compared to the Erd\H{o}s-R\'enyi setting the proof in this case turns out to be much simpler. It stems from the fact in the Erd\H{o}s-R\'enyi setting one may have the maximum degree of a core graph be as large as $\Theta(d^2)$, while due to the degree upper bound that is not permitted in our setting.}

\begin{proof}[Proof of Lemma \ref{lem:strong core main lemma} for $\Delta=2$]
\dteal{The proof is another application of Corollary \ref{cor:entropic stability 2} and a union bound. To carry out the details, 
we begin by noticing that any} connected $2$-regular graph $H$ is a cycle $C_{\ell}$, for some $\ell\geq 3$. Hence, the \dteal{sparsity regime simplifies to the following:}  
\begin{equation*}
    \label{eq:regime cycle SC}
    (\log n)^{1/(\ell-2)}\ll d=np\leq \sqrt{n}.
\end{equation*}
\dteal{We now proceed to} verify the hypotheses of Corollary \ref{cor:entropic stability 2} \dteal{with $Q=Q_{\sc}$, $T=T_{\sc}$, and $R=16\Delta/\varepsilon$ (recall \eqref{eq:def of Q-sc and T-sc})}.
\dteal{Since $d\gg1 $, we have $R < Q_{sc}/d$, for sufficiently large $n$, and $R \ll d$}. 
%
\dteal{As $p=d/n\leq n^{-1/2}$ 
we get}
\[
\dteal{\frac{T}{Q}\cdot \log \frac{d}{R} \lesssim \log n \lesssim \log(1/p)\leq p^{-\varepsilon/3}.}
\]
Similarly, by another application of $p\leq n^{-1/2}$ and \dteal{as} $R=16\Delta/\varepsilon$, we have
\[\frac{1}{R}\log n\leq \frac{\Delta}{R} \cdot \log(1/p) \leq  \frac{\varepsilon}{16} \log(1/p).\]	
Thus, by Corollary \ref{cor:entropic stability 2} 
and \eqref{eq:SC in G(Q,T)}, we obtain 
\begin{equation}
    \label{eq:entropic stability of SC-cycles}
    |\{G\in \widehat{\Xd}:\ G\tu{ is a strong-core graph with }e(G) = \bm{e}\}|\leq p^{-\varepsilon \bm{e}}.
\end{equation}
\dteal{Combining \eqref{eq:core-large-p-large-2} \dteal{(although stated for core graphs it continues to hold for strong core graphs, and will be used later as well)}, \eqref{eq:entropic stability of SC-cycles}, and \eqref{eq:def of T_1} we therefore deduce} 
\begin{equation}
    \label{eq:SC cycles}
    \begin{aligned} 
        \P(\Gd \tu{ contains a strong-core}) &\leq \dteal{\sum_{\bm{e}\geq \lceil  T_{1}\rceil}} p^{(1-2\varepsilon)\bm{e}} \leq p^{(1-3\varepsilon)T_{1}},
    \end{aligned}
\end{equation}
for all sufficiently large $n$. 
Substituting the definition of $T_1$ from \eqref{eq:def of T_1} yields the desired bound and completes the proof. 	
\end{proof}

\dteal{
\begin{proof}[Proof of Lemma \ref{lem:strong core main lemma} for non-bipartite $\Delta$-regular graphs with $\Delta \geq 3$]
Repeating the same arguments in the proof of \cite[Proposition 3.3]{ABRB} for non-bipartite graphs (see also \cite[Eqn.~(4.8)]{ABRB} it follows that
\begin{equation}\label{eq:entropic stability of SC-cycles-1}
\left|\{G\in \widehat{\Xd}:\ G\tu{ is a strong-core graph with }e(G) = \bm{e} \}\right| \leq \exp\left(5 \Delta \varepsilon {\bm e} \log(1/p)\right).
\end{equation}
Now, applying \eqref{eq:core-large-p-large-2} and a union bound, and noticing that the rate function in the Erd\H{o}s-R\'enyi setting and in the random regular graph setting are the same for $\Delta \geq 3$, the required bound follows. 
\end{proof}
}

\subsubsection{$\Delta \geq 3$ and $H$ bipartite}
\dteal{In this case, one can attempt to repeat the arguments of \cite{ABRB}. However, the arguments in \cite{ABRB} are not sharp enough to claim that the probability of the existence of a strong core graph with $(1+C\varepsilon)T_1$, for some large constant $C$ depending only on $\delta$ and $H$, is negligible in the large deviation scale. This enhanced probability bound will be a crucial input in the proof of Theorem \ref{thm:cond-loc}. Hence, we need some refinements of the arguments in \cite{ABRB}.}

\dteal{To this end, we begin with an improvement of \cite[Lemma 4.4]{ABRB}. It requires the following assumption.}


\begin{assumption}
\label{asmpt:bipartite}
Let $\overline{G}$ be a bipartite graph with two parts $U_1$ and $U_2$, i.e., $E(\overline{G})\subset U_1\times U_2$ for some disjoint set of vertices $U_{1}$ and $U_2$. Assume that there exists $U_{1,1}\subseteq U_1$ such that 
\begin{equation}
    \label{eq:min degree requirement}
    \dteal{\min_{u\in U_{1,1}}\tu{deg}_{\overline{G}}(u)\geq \Delta, \quad \text{ for some } \Delta \geq 2, \quad \text{ and } \quad |U_1\setminus U_{1,1}|\leq \varepsilon_{2} e(\overline{G}),}
\end{equation}
for \dteal{some 
$\varepsilon_{2}\in (0,1)$.}
\end{assumption}

\begin{lemma}
\label{lem:bipartite edge,vertex bound}
Fix a $\Delta$-regular connected bipartite graph $H$ with $\Delta\geq 2$. Let $G$ be a bipartite graph with parts $U_{1}$ and $U_{2}$ satisfying Assumptions \ref{asmpt:bipartite}. Then, \dteal{for $\varepsilon_2 \in (0,1)$ sufficiently small, depending only on $\Delta$}, 
\begin{equation}\label{eq:bipartite-improved}
e(\overline{G})- \frac{\Delta}{2}\cdot |U_1|\geq (1-2\Delta \varepsilon_{2})\cdot \left(\frac{\Delta}{2}\right)^{2/v_{H}}\cdot \frac{1}{2}N(H, \overline{G})^{2/v_{H}}. 
\end{equation}
\end{lemma} 

\dteal{Let us point out that \cite[Lemma 4.4]{ABRB} holds with $2\Delta \varepsilon_2$ and $\Delta/2$ in \eqref{eq:bipartite-improved} replaced by $\Delta \varepsilon_2$ and $1$, respectively. Therefore, Lemma \ref{lem:bipartite edge,vertex bound} provides a better bound for $\Delta \geq 3$. This will be a key in improving the error bounds in the proof of the Lemma \ref{lem:strong core main lemma} for $\Delta \geq 3$.} 


\begin{proof}[Proof of Lemma \ref{lem:bipartite edge,vertex bound}]
\dteal{We borrow ideas from the proof of \cite[Lemma~4.4]{ABRB}. We will assume that $\Delta \geq 3$.} 

\dteal{We begin by noting that, $H$ being} a bipartite graph, we can partition its vertex set into $V^{(1)}(H)$ and $V^{2}(H)$, where both parts are independent \dteal{sets}. We also observe that $H$ being $\Delta$-regular, gives
\begin{equation}
    \label{eq:t geq Delta}
    |V^{1}(H)|=|V_{2}(H)| = \frac{v_{H}}{2}\eqqcolon t, \quad t\geq \Delta.
\end{equation}
\dteal{Without loss of generality, let $V^{1}(H) = \{x_1,x_2,\ldots,x_t\}$ and $V^{2}(H)=\{\bar{x}_1, \bar{x}_2, \ldots, \bar{x}_t\}$. By \cite[Lemma 4.6]{ABRB}, there exists a subgraph $\widehat M_1 \subseteq H$ that is the edge disjoint union of some perfect matching $M$ and a collection of $(\Delta -1)$ edges incident to the vertex $x_t \in V^1(H)$. Reversing the roles of $V^1(H)$ and $V^2(H)$ one also obtains a subgraph $\widehat M_2 \subseteq H$ with properties similar to $\widehat M_1$, with the vertex $x_t$ replaced by $\bar x_t$. 

Next, for $j \in \{1,2\}$, we set $N_j(H,\overline{G})$ to be the number of labelled copies of $H$ in $G$ such that the vertices of $V^j(H)$ are mapped to $U_1$. Fixing, $t$ distinct vertices $\bm{u} = (u_{1},u_{2},\ldots, u_{t})\in U_{1}^{t}$ we let  $N(\widehat{M}_1, \overline{G},\bm{u})$ and $N(\widehat{M}_2, \overline{G},\bm{u})$ to be the number of copies of $\widehat{M}_1$  and $\widehat M_2$ in $\overline{G}$, where the the vertices $x_i$ and $\bar x_i$ are mapped to the vertex $u_i$, respectively.} 

\dteal{Therefore, as $\overline{G}$ is bipartite, from the discussion above it follows
\begin{equation}\label{eq:H-ot-M}
N(H,\overline{G}) = \sum_{j=1}^2 N_j(H,\overline{G}) \leq \sum_{j=1}^2 \sum_{\bm{u}\in U_{1}^{t}} N(\widehat{M}_j, \overline{G}, \bm{u}),
\end{equation}
where the innermost sum in the RHS above is taken over tuples with distinct coordinates. We will show that for sufficiently small $\varepsilon_{2}\in (0,1)$, 
\begin{equation}
    \label{eq:bipartite improvement claim}
    \sum_{\bm{u}\in U_{1}^{t}} N_{j}(\widehat{M}, \overline{G}, \bm{u})\leq \frac {1}{\Delta}\cdot \left(\frac{2}{1-2\Delta\varepsilon_2}\right)^{t}\cdot \left(e(\overline{G})-\frac{\Delta}{2}|U_1|\right)^{t}, \quad \tu{for }j\in \{1,2\}.
\end{equation}
Notice that \eqref{eq:H-ot-M} and \eqref{eq:bipartite improvement claim} together imply the required bound.} 	


\dteal{So, it remains to  prove \eqref{eq:bipartite improvement claim}. To this end, arguing similarly as in the steps leading to \cite[Eqn.~(4.13)]{ABRB} we obtain 
%
\begin{equation}
    \label{eq:intermediate improved bound}
    \begin{aligned}
        \sum_{\bm{u}\in U_{1}^{t}}N(\widehat{M}_j, \overline{G}, \bm{u}) 
        & \leq\left(\sum_{u\in U_1}\deg_{\overline{G}}(u)\right)^{t-1}\cdot \left[\sum_{u\in U_1}(\deg_{\overline{G}}(u)-(\Delta-1))\right] \\
         & = e(\overline{G})^{t-1}\cdot (e(\overline{G})-(\Delta-1)|U_1|) = \mathfrak{h}^{t-1}(\mathfrak{h}-\Delta+1) \cdot |U_1|^t,
    \end{aligned}
\end{equation}
where $\mathfrak{h}:=e(\overline{G})/|U_1|$. To complete the proof we need to find an upper bound on the RHS of \eqref{eq:intermediate improved bound}. Turning to do this, define 
\[
g(x):= x^{t-1}(x-\Delta+1)\cdot (x-(\Delta/2))^{-t},\quad   x\geq \Delta -1.
\]
It is strightforward to note that derivative of $\log g(x)$ is strictly negative on $x\in(x_{H},\infty)$ where $x_{H} \coloneqq x_H(\Delta, t):=\frac{\Delta(\Delta-1)(t-1)}{\Delta t-2(\Delta-1)}$, and hence $g(x)$ is strictly decreasing on the same interval. Further noting that the map $(\Delta, t) \mapsto x_H(\Delta, t)$ is decreasing in $t$, and as $t= \frac{v_{H}}{2}\geq \Delta$, we find that $x_H(\Delta, t) \leq x_H(\Delta, \Delta) < \Delta$. This, in particular, implies that $x_{H}<(1-\Delta \varepsilon_{2})\Delta$, provided $\varepsilon_{2}$ is sufficiently small, depending only on $\Delta$. Therefore,  
\begin{equation}\label{eq:sup-g}
\sup_{x\geq (1-\Delta\varepsilon_2)\Delta}g(x)\leq g((1-\Delta\varepsilon_2)\Delta) = \frac{2^t}{\Delta}\cdot \frac{(1-\Delta\varepsilon_{2})^{t-1}\cdot (1-\Delta^{2}\varepsilon_2)}{(1-2\Delta\varepsilon_2)^t}\leq \frac{2^t}{\Delta\cdot (1-2\Delta\varepsilon_2)^{t}}.
\end{equation}
}
\dteal{Now we apply 
the assumption \eqref{eq:min degree requirement}
to derive
\begin{equation*}
    \Delta |U_1|\leq \sum_{u\in U_1}\deg_{\overline{G}}(u)+\Delta|U_1\setminus U_{1,1}|\leq (1+\Delta\varepsilon_2)e(\overline{G}),
\end{equation*}
which in turn implies that $\mathfrak{h} \geq \frac{\Delta}{1+\Delta\varepsilon_2}\geq (1-\Delta\varepsilon_2)\Delta$. Combining \eqref{eq:intermediate improved bound} with \eqref{eq:sup-g} we arrive at \eqref{eq:bipartite improvement claim}. This finishes the proof.}  
\end{proof}	

\dteal{The reader may have already anticipated that to prove Lemma \ref{lem:strong core main lemma} for $\Delta \geq 3$ and $H$ bipartite, ideas from the proofs of \cite{ABRB} will again be borrowed. Indeed, we will use several ingredients from the proofs of in \cite[Section 4.4]{ABRB}. To this end, we first need to introduce several notation.} 


Set 
\begin{equation*}
    C_{3}\coloneqq C_{3}(\varepsilon) = (t-1)\left(\left\lceil \frac{(2\cstar)^{t}}{\delta \varepsilon^{2}}\right\rceil+2\right),
\end{equation*}
\dteal{where $t=v_H/2$.} Take the constants $0<c_{0}(\varepsilon)<C_{0}(\varepsilon)<\infty$ as in Lemma \ref{lem:deg product bound strong core} for which the degree product bounds \eqref{eq:deg prod upper bound}-\eqref{eq:deg prod lower bound SC} and  bounds in \eqref{eq:G-low contains all excess copies} hold. For $i\in \set{C_3}$, define 
\begin{equation}
    \label{eq:def R_i}
    R_{i}\coloneqq R_{i}(\varepsilon)\coloneqq 
    \frac{16\Delta}{\varepsilon}\cdot \left(\frac{C_{0}(\varepsilon)}{c_{0}(\varepsilon)}\right)^{t-1}.
\end{equation}
Given a strong-core graph $G$, we consider the sets  $V_{i}\coloneqq \W_{R_i}(G)\cap V(G_\low)$ for $i\in \set{C_3}$. Let $G_{i}$ be the subgraph of $G_{\low}$ spanned by edges that are incident to some vertex in $V_{i}$, and set $\overline{G}_i\coloneqq G_{\low}\setminus G_{i}$. \dteal{Before moving to the proof of Lemma \ref{lem:strong core main lemma}, let us state the following observations that were made during the course of the proofs in \cite[Section 4.4]{ABRB} in the form of a lemma.} 
\dteal{	
\begin{lemma}
    \label{claim: i_star}
Let $G$ be any strong core graph. There exists $i_{\star}\in \set{C_3-t+1}$, such that the followings hold:
    \begin{itemize}
        \item[(a)] $N(H, G_{i_\star+t-1})-N(H, G_{i_{\star}})\leq \varepsilon^{2}\delta n^{v_{H}}p^{e_{H}}$.
        \item[(b)] Any copy of $H$ in $G_{\low}$ that uses an edge of $G_{i_{\star}}$ is contained in $G_{i_{\star}+t-1}$.
\item[(c)] $N(H, G_{i_{\star}})+N(H, \overline{G}_{i_\star})\geq \delta(1-5\varepsilon) n^{v_{H}}p^{e_{H}}$.			
    \end{itemize}
\end{lemma}
}	

\begin{proof}[Proof of Lemma \ref{lem:strong core main lemma} for bipartite $\Delta$-regular graphs, with $\Delta\geq 3$]
\dteal{Recalling \eqref{eq:def of Q-sc and T-sc} and \eqref{eq:def of T_1}, and applying Lemma \ref{claim: i_star}(c) we find that 
\begin{equation}
    \label{eq:SC tuple}
    \{\Gd\tu{ contains a strong core graph}\} \subseteq \bigcup_{i_\star=1}^{C_3-t+1}\bigcup_{e=\lceil T_1\rceil}^{\lfloor T_{\sc}\rfloor} \bigcup_{\substack{\bm{e}_{\#}\leq \bm{e},\\ \bm{v}_{\#}\leq \bm{v}\leq 2\bm{e}}}\{\exists G\subseteq \Gd:\ G\in \mathcal{I}_{i_\star, \bm{e}, \bm{e}_{\#}, \bm{v}, \bm{v}_{\#}}\},
\end{equation}
where   
\begin{equation}
    \begin{aligned}
        \mathcal{I}_{i_{\star}, \bm{e}, \bm{e}_{\#}, \bm{v}, \bm{v}_{\#}}\coloneqq \{G:\ G \tu{ is a strong core graph}& \tu{ with }e(G)=\bm{e},\ e(G_{i_\star})=\bm{e}_{\#}, |\W_{R_{i_{\star}}}(G)|=\bm{v}, |V_{i_{\star}}|=\bm{v}_{\#},\\
        &\tu{ and }N(H, G_{i_\star})+N(H, \overline{G}_{i_\star})\geq (1-5\varepsilon)\delta n^{v_{H}}p^{e_{H}} \}.
    \end{aligned}
\end{equation}
}	

\dteal{
As we have already seen a few times the approach would to be bound the cardinality of $\mathcal{I}_{i_{\star}, \bm{e}, \bm{e}_{\#}, \bm{v}, \bm{v}_{\#}}$. To this end, we claim that  
\begin{equation}
    \label{eq:2nd uniform bound on I}
    |\mathcal{I}_{i_{\star}, \bm{e}, \bm{e}_{\#}, \bm{v}, \bm{v}_{\#}}|\leq  {n\choose \bm{v}}\cdot p^{-\varepsilon \bm{e}} \leq    p^{-\frac{\Delta}{2}\bm{v} -2\varepsilon \bm{e}} \leq p^{-\frac{\Delta}{2}\bm{v}_{\#} -3\Delta\varepsilon \bm{e}}.
\end{equation}
}

To prove the leftmost inequality in  we apply Lemma \ref{lem:entropic stability} with $T=T_{\sc},Q=Q_{\sc}$ and $R=R_{i_\star}(\varepsilon)$. \dteal{Observe that, in the current regime $R \asymp 1 \ll d$. Since, $T/Q \asymp 1$ and $\log d \lesssim \log(1/p)$ the leftmost assertion in \eqref{eq:three condition} holds. As, $np^\Delta \leq 1$ and $R_{i_\star}\geq 16\Delta/\varepsilon$ the middle assertion in \eqref{eq:three condition} holds.} 

%
\dteal{Further}, if $(u,v)\in E(G)$ with $u,v\in \W_{R_{i_\star}}(G)$, then from Lemma \ref{lem:deg product bound strong core}(a), we have
\begin{equation}
    \label{eq:bipartiteness of G_i-star}
    \tu{deg}_{G}(v)\geq  \frac{1}{R_{i_\star}}\tu{deg}_{G}(u)\cdot \tu{deg}_{G}(v)\geq \frac{Q_{\sc}}{R_{i_\star}}\gg R_{i_\star}, 
\end{equation} 
where the last inequality is true since $np^{\Delta/2}\gg1$. So there are no edges with both endpoints in $\W_{R_{i_\star}}$. Hence, upon recalling \eqref{eq:SC in G(Q,T)} and applying Lemma \ref{lem:entropic stability} we obtain the leftmost inequality in \eqref{eq:2nd uniform bound on I}.
%
%

\dteal{Next, we aim to derive the middle inequality in \eqref{eq:2nd uniform bound on I}. If ${\bm v}=0$ then the inequality}
holds trivially. \dteal{So assume} $\bm{v}\geq 1$. Using the inequality ${a\choose b}\leq (ea/b)^{b}$, we can bound as below
\[{n\choose {\bm v}}\cdot p^{-\varepsilon \bm{e}}\leq \left(\frac{en}{\bm{v}}\right)^{{\bm v}} p^{-\varepsilon \bm{e}} = p^{-\frac{\Delta}{2}\bm{v} -\varepsilon \bm{e}}\cdot \left(\frac{enp^{\Delta/2}}{\bm{v}}\right)^{\bm{v}}\leq p^{-\frac{\Delta}{2}\bm{v} -\varepsilon \bm{e}}\cdot \exp(np^{\Delta/2})\leq p^{-\frac{\Delta}{2}\bm{v}-2\varepsilon \bm{e}},\]
where in the penultimate step we use the inequality $\left(\frac{a}{b}\right)^{b}\leq e^{a/e}$ and the final step follows from the observation that $np^{\Delta/2}\gg 1$ and $\bm{e}\geq T_{1}\asymp n^{2}p^{\Delta}$.

Now, turning to prove the rightmost inequality in \eqref{eq:2nd uniform bound on I} \dteal{recalling the definition of $V_{i_\star}$ 
we find that} $\W_{R_{i_\star}}(G)\setminus V_{i_\star}\subseteq V(G_{\high})$. So, using \eqref{eq:G-low contains all excess copies}, we have for any graph $G\in \mathcal{I}_{i_{\star}, \bm{e}, \bm{e}_{\#}, \bm{v}, \bm{v}_{\#}}$, 
\begin{equation*}
    0\leq \bm{v}-\bm{v}_{\#} \leq |V(G_{\high})|\leq 2e(G_{\high})\leq 2\varepsilon^{2} \bm{e}. 
\end{equation*}
\dteal{Substituting the above we obtain the rightmost inequality in \eqref{eq:2nd uniform bound on I}.}
%
%
\dteal{Now, combining the bounds \eqref{eq:2nd uniform bound on I} and \eqref{eq:core-large-p-large-2}, and} using $\bm{e}\leq T_{\sc}\leq \cstar n^{2}p^{\Delta}$, we get
\begin{equation}
    \label{eq:general prob bound Delta-3}
    \P(\exists G\subseteq \Gd: G\in \mathcal{I}_{i_{\star}, \bm{e}, \bm{e}_{\#}, \bm{v}, \bm{v}_{\#}})\leq \exp\left(-\left(\bm{e}-\frac{\Delta}{2}\bm{v}_{\#}\right) \log\frac{1}{p}+\varepsilon (4\Delta\cstar) n^{2}p^{\Delta}\log\frac{1}{p}\right).
\end{equation}

\dteal{Next}, the goal is to simplify the RHS. For this, we need to find a lower bound on $(\bm{e}-\frac{\Delta}{2}\bm{v}_{\#})$. 
\dteal{We} need to distinguish according to whether $N(H, G_{i_{\star}})$ is small or large.

\subsubsection*{Case 1: $N(H, G_{i_{\star}})\leq \varepsilon^{2}\cdot \delta n^{v_{H}}p^{e_{H}}$} 
It follows from Lemma \ref{claim: i_star}(a),  
\begin{equation}
    \label{eq:bound on G-i+t-1}
    N(H,G_{i_\star + t-1})\leq N(H, G_{i_{\star}})+\varepsilon^2 \delta n^{v_{H}}p^{e_{H}}\leq 2\varepsilon^{2}\delta n^{v_{H}}p^{e_{H}}.
\end{equation} 
Moreover, Lemma \ref{claim: i_star}(b) implies that every copy of $H$ passing through an edge of $G_{i_{\star}}$ must either be contained in $G_{i_\star+t-1}$ or does not use any edge from $G_{\low}$.  As lower bound (SC3)  holds for every edge $e\in E(G_{i_{\star}})$, 
we deduce
\begin{equation*}
    \begin{aligned} 
        e(G_{i_\star})\cdot (\delta\varepsilon/\cstar) (np^{\Delta/2})^{v_{H}-2}\leq \sum_{e\in E(G_{i_\star})}N(H,G,e)\leq& e_{H} \cdot (N(H, G_{i_\star+t-1}+N(H, G)-N(H, G_{\low})))\\[-10pt]
        &\leq 3\varepsilon^{2}e_{H}\delta n^{v_{H}}p^{e_{H}},
    \end{aligned}
\end{equation*}
where the final step follows from \eqref{eq:G-low contains all excess copies} and \eqref{eq:bound on G-i+t-1}. Thus, for any $G\in \mathcal{I}_{\tup}$, we obtain
\begin{equation}
    \label{eq:vertex cardinal of G_i-star}
    \bm{v}_{\#}\leq 2e(G_{i_\star})\leq \varepsilon (6\cstar e_{H}) n^{2}p^{\Delta}.
\end{equation}
Therefore, plugging the bound \eqref{eq:vertex cardinal of G_i-star} in \eqref{eq:general prob bound Delta-3}, \dteal{by a union bound, and using $C_3 \lesssim 1$ and $\log T_{\sc} \ll n^{2}p^{\Delta}\log(1/p)$, and recalling the definitions of $T_1$ and $\Psi_H(\delta)$ from \eqref{eq:def of T_1} and \eqref{eq:rate function defn}, respectively, we derive that} 
    \begin{equation}
        \label{eq:entropically stable}
        \begin{aligned}
            &\P\left(\bigcup_{i_\star=1}^{C_3-t+1}\bigcup_{e=\lceil T_1\rceil}^{\lfloor T_{\sc}\rfloor} \bigcup_{\substack{\bm{e}_{\#}\leq \bm{e},\\ \bm{v}_{\#}\leq \bm{v}\leq 2\bm{e}}}\left\{\exists G\subseteq \Gd: G\in \mathcal{I}_{\tup},\ N(H, G_{i_\star}) \leq \varepsilon^{2}\delta n^{v_{H}}p^{e_{H}}\right\}\right)\\ 
             \leq \ & \exp\left(-(1-\mathfrak{f}^{(1)}_{H}(\varepsilon))\Psi_{H}(\delta)n^{2}p^{\Delta}\log\frac{1}{p}\right),\\
        \end{aligned}
    \end{equation}
for some positive function $\mathfrak{f}^{(1)}_{H}(\varepsilon)$ 
satisfying $\lim_{\varepsilon\downarrow 0}\mathfrak{f}^{(1)}_{H}(\varepsilon)=0$.

    
\subsubsection*{Case 2: $N(H, G_{i_{\star}})\geq \varepsilon^{2}\delta n^{v_H}p^{e_{H}}$}
\dteal{The proof in this case is more elaborate.} To get a lower bound on $(\bm{e}-\frac{\Delta}{2}\bm{v}_{\#})$, we split it into two parts: $(\bm{e}-\bm{e}_{\#})$ and $(\bm{e}_{\#}-\frac{\Delta}{2}\bm{v}_{\#})$, \dteal{and unlike the previous case both these quantities may be non negligible.}

\dteal{Let} $G\in \mathcal{I}_{\tup}$ with $N(H, G_{i\star})\geq \varepsilon^{2}\delta n^{v_{H}}p^{e_{H}}$ and set 
\begin{equation}\label{eq:def-wh-x}
\x\coloneqq \min\left\{\frac{N(H, G_{i_\star})}{\delta n^{v_{H}}p^{e_{H}}}, 1-5\varepsilon\right\}.
\end{equation}
\dteal{By Lemma \ref{claim: i_star}(c)  we have the lower bound $N(H, \overline{G}_{i_\star})\geq \delta(1-5\varepsilon-\x)n^{v_{H}}p^{e_{H}}$. Therefore, using the lower bound on $N(H, G_{i_\star})$, applying Lemma \ref{lem:minimum edge bound}, and noting that $\Psi_H(\gamma)=\gamma^{2/v_H}/2 $, as $\Delta \geq 3$, we deduce  
    \begin{equation}
        \label{eq:e-e-sharp}
        \bm{e}-\bm{e}_{\#}\geq e(\overline{G}_{i_\star}) \geq (1-5\varepsilon -\x)^{1/t}\cdot \frac{1}{2}\delta^{1/t}\cdot n^{2}p^{\Delta} \qquad \text{ and } \qquad \bm{e}_{\#}\geq \varepsilon^{2/t}\cdot \frac{1}{2}\delta^{1/t} n^{2}p^{\Delta}.
    \end{equation}
}		
Next, we prepare to apply Lemma \ref{lem:bipartite edge,vertex bound}. As demonstrated in \eqref{eq:bipartiteness of G_i-star}, there are no edges with both endpoints in $\W_{R_{i_\star}}(G)$ and since $V_{i_\star}\subseteq  \W_{R_{i_\star}}(G)$, we note that $G_{i_\star}$ is a bipartite graph with one part $V_{i_\star}$. We further observe that for any $v\in V_{i\star}\setminus V(G_\bad)$, the vertex $v$ must participate  in at least one  copy of $H$ contained in $G_{\low}$. Thus implying $\deg_{G_{i_\star}}(v)\geq \Delta$. Also  by Lemma \ref{lem:deg product bound strong core}(b) and \dteal{as} $e(G)\leq T_{sc}=\cstar n^{2}p^{\Delta}$,  we have
    \[|V(G_{\bad})|\leq \dteal{2}\varepsilon^{2} e(G)\leq \dteal{2}\varepsilon^{2} \cstar n^{2}p^{\Delta} \leq K_0 \cdot \varepsilon^{2(1-1/t)}\cdot  \bm{e}_{\#},\]
    for some constant $K_0$ (depending only on $\delta$ and $H$), where the last inequality follows from \dteal{using the lower bound on ${\bm e}_{\#}$ obtained in \eqref{eq:e-e-sharp}}. Therefore, $G_{i_\star}$ satisfies Assumption \ref{asmpt:bipartite} with $\varepsilon_2 = K_0 \cdot \varepsilon^{2(1-1/t)}$, $U_1=V_{i_\star}$ and $U_{1,1}= V_{i_\star}\setminus V(G_{\bad})$. Lemma \ref{lem:bipartite edge,vertex bound} now gives us  
    \begin{equation}
        \label{eq:e-sharp-v-sharp}
        \bm{e}_{\#} -\frac{\Delta}{2}\bm{v}_{\#}\geq (1-2\dteal{K_0}\Delta \varepsilon^{2(1-1/t)})\cdot \left(\frac{\Delta}{2}\right)^{1/t}\frac{1}{2}N(H, G_{i_\star})^{1/t} \geq (1-2\Delta K_0\varepsilon^{2(1-1/t)})\cdot \left(\frac{3\x}{2}\right)^{1/t}  \frac{1}{2}\delta^{1/t}n^{2}p^{\Delta}, 
    \end{equation}
    where in the last step we used $\Delta\geq 3$ and $N(H, G_{i_\star})\geq \x \delta n^{v_{H}}p^{e_H}$. 
\dteal{Upon combining 
\eqref{eq:e-e-sharp} and  \eqref{eq:e-sharp-v-sharp}, and recalling \eqref{eq:rate function defn} we obtain that, for $\varepsilon >0$ sufficiently small,
    \begin{equation}
        \label{eq:concave function}
        \bm{e}-\frac{\Delta}{2}\bm{v}_{\#}\geq F(\x)\cdot \frac{1}{2}\delta^{1/t}n^{2}p^{\Delta} \geq \left((1+(9/8)^{1/t}\varepsilon^{2/t}\right)\Psi_H(\delta),
    \end{equation}
    where 
    \begin{equation*}
        F(x)\coloneqq (1-5\varepsilon-x)^{1/t}+(1-2\Delta K_0 \varepsilon^{2(1-1/t)})\left(\frac{3x}{2}\right)^{1/t}, \quad x \in[0, 1-5\varepsilon],
    \end{equation*}
and in the last step in \eqref{eq:concave function} we used that $F(\cdot)$ is  concave on its domain and $\x \in [\varepsilon^2, 1-5\varepsilon]$ (recall \eqref{eq:def-wh-x}).}

    
\dteal{Substituting the bound \eqref{eq:concave function} in \eqref{eq:general prob bound Delta-3}, and using $t\geq \Delta \geq 3$ and taking $\varepsilon$ sufficiently small. 
we deduce}
    \begin{equation*}
        \begin{aligned}
            &\P(\exists G\subseteq\Gd:G\in \mathcal{I}_{\tup}, N(H, G_{i_\star})\geq\varepsilon^{2}\delta n^{v_{H}}p^{e_{H}})\\ &\leq \exp\left(-\left(1+\left(\frac{9\varepsilon^{2}}{8}\right)^{1/t}-\varepsilon \frac{8\Delta\cstar}{\delta^{1/t}}\right)\frac{\delta^{1/t}}{2}n^{2}p^{\Delta}\log\frac{1}{p}\right)\leq \exp\left(-\left(1+\left(\frac{17\varepsilon^{2}}{16}\right)^{\frac{1}{t}}\right)\frac{\delta^{1/t}}{2}n^{2}p^{\Delta}\log\frac{1}{p}\right),
        \end{aligned}
    \end{equation*}
    
%
\dteal{Therefore by a union bound, we obtain} 
    \begin{equation}
        \label{eq:entropically expensive}
        \begin{aligned}
            &\P\left(\bigcup_{i_\star=1}^{C_3-t+1}\bigcup_{e=\lceil T_1\rceil}^{\lfloor T_{\sc}\rfloor} \bigcup_{\substack{\bm{e}_{\#}\leq \bm{e},\\ \bm{v}_{\#}\leq \bm{v}\leq 2\bm{e}}}\{\exists G\subseteq \Gd: G\in \mathcal{I}_{\tup},\ N(H, G_{i_\star}) 	\geq \varepsilon^{2}\delta n^{v_{H}}p^{e_{H}}\}\right)\\
            \leq \ & \exp\left(-\left(1+\left(\frac{33\varepsilon^{2}}{32}\right)^{\frac{1}{t}}\right)\Psi_{H}(\delta)n^{2}p^{\Delta}\log\frac{1}{p}\right).
        \end{aligned}
    \end{equation}
    
    To complete the proof, observe that the bound obtained in \eqref{eq:entropically expensive} is exponentially smaller than the bound in \eqref{eq:entropically stable}. 
This completes the proof.
\end{proof}

\dteal{We end the section with the following result which will be used in Section \ref{sec:conditional-loc-regime} to prove Theorem \ref{thm:cond-loc}.}

\newcommand{\lrge}{\tu{large}}
\begin{proposition}\label{prop:loc-cond}
\dteal{Consider the same setup as in Theorem \ref{thm:loc-regime}.} 
\dteal{Then, there exists a large constant $C\in (0,\infty)$ depending only on $\delta$ and $H$ such that for any $\varepsilon>0$ sufficiently small and sufficiently all large $n$, we have  

\[\P\left(\left\{\exists G\subseteq \Gd: G \in \mathcal{G}_0\right\}\big|\, \UTH \right)\geq 1-\P(\UTH)^{\varepsilon},\]
where 
\begin{equation}\label{eq:mathcal-G-0}
\mathcal G_0:= \mathcal G_0(C,\varepsilon):= \left\{ G \subseteq \widehat{\Xd}: N(H,G)\geq \delta(1-C\varepsilon)n^{v_{H}}p^{e_{H}} \tu{ and }e(G)\leq (1+C\varepsilon)\Psi_{H}(\delta)n^{2}p^{\Delta}\right\}.
\end{equation}}	
\end{proposition}

\begin{proof}
\dteal{By \eqref{eq:cond-struc-starting}, it suffices to prove that 
\begin{equation}\label{eq:mathcal-E-cond-prob}
\P(\mathcal{E} \mid \UTH)\leq \P(\UTH)^{2 \varepsilon}, 
\end{equation}
where $T_\lrge \coloneqq (1+C\varepsilon)\Psi_{H}(\delta)n^{2}p^{\Delta}$ and 
\[\mathcal{E} \coloneqq \{\Gd \tu{ contains a core graph }G \tu{ with }e(G)\geq T_\lrge\}.
\]
As in the proof of Proposition \ref{prop:prob of existence of core} we need to cosider a few cases:

\subsubsection*{Case 1: $np^{\Delta-1}\gg (\log n)^{v_{H}}$} 
Applying \eqref{eq:entropic stability of cores large p},\eqref{eq:single core probability}, the lower bound \eqref{eq:lower bound}, and a union bound we have 
\[
\P(\mathcal{E}) \leq \sum_{\bm{e}\geq \lceil  T_{\lrge}\rceil} p^{(1-2\varepsilon)\bm{e}} \leq  
p^{(1-3\varepsilon)T_\lrge} \leq \P(\UTH)^{1+2\varepsilon},
\]
whenever $C$ is a sufficiently large absolute constant. This yields \eqref{eq:mathcal-E-cond-prob}. 
}
\dteal{
\subsubsection*{Case 2: $(\log n)^{1/(v_H-2)}\ll np^{\Delta/2}\leq n^{1/2}$} By Lemma \ref{lem:SC reduction step} and the observation that any core graph $G$ with $e(G)\leq \cstar n^{2}p^{\Delta}$, contains a strong core subgraph $G^{\prime}\subseteq G$, we notice that it suffices to show
\begin{equation}\label{eq:wh-cE-prob-bound}
\P(\widehat{\mathcal{E}})\leq \P(\UTH)^{1+2 \varepsilon}, 
\end{equation} 
where
\[
\widehat{\mathcal{E}} \coloneqq \{\Gd \tu{ contains a string core graph }G \tu{ with }e(G)\geq T_\lrge\},
\]
for some large constant $C$ depending only on $\delta$ and $H$ to be determined below. We consider two further subcases.
}
\dteal{
\subsubsection*{Case 2(a): $\Delta =2$ or $H$ non-bipartite} By \eqref{eq:entropic stability of SC-cycles} and \eqref{eq:entropic stability of SC-cycles-1}, and arguing as in \eqref{eq:SC cycles} we get 
\begin{equation*}
        \P(\widehat{\mathcal E}) \leq \sum_{\bm{e}\geq \lceil  T_{1}\rceil} p^{(1-6\Delta\varepsilon)\bm{e}} \leq p^{(1-7\Delta \varepsilon)T_{\lrge}} \leq \P(\UTH)^{1+2\varepsilon},
\end{equation*}
for any $C$ sufficiently large depending only on $\Delta$. 	
}

%

\dteal{
\subsubsection*{Case 2(b): $H$ bipartite and $\Delta \geq 3$} Observe that as $H$ is bipartite and $\Delta$-regular with $\Delta \geq 3$, we have $v_H \geq 2\Delta \geq 6$. Consequently, $\varepsilon^{4/v_H}\gg\varepsilon$ as $\varepsilon\downarrow0$. Hence, the RHS of \eqref{eq:entropically expensive} is at most $\P(\UTH)^{1+3\varepsilon/2}$ for $\varepsilon >0$ sufficiently small. Now, recalling \eqref{eq:SC tuple} and \eqref{eq:entropically expensive}, and applying \eqref{eq:general prob bound Delta-3}, \eqref{eq:vertex cardinal of G_i-star}, and a union bound we obtain 
%
%
%
%
\begin{equation}
\label{eq:conditional SC large part 3}
    \P(\widehat{\mathcal E}) \leq
    \exp\left(-(T_\lrge-\varepsilon(8\Delta\cstar e_{H})n^{2}p^{\Delta})\log\frac{1}{p}\right)+ \P(\UTH)^{1+3\varepsilon/2} \leq \P(\UTH)^{1+2\varepsilon},
\end{equation}
where the last step holds for $C$ sufficiently large depending only on $\Delta$ and $v_H$.}
\dteal{The proof is now complete.}
\end{proof}


\section{Conditional structure in the localized regime}\label{sec:conditional-loc-regime}
In this section we prove Theorem \ref{thm:cond-loc}. A key to the proof of Theorem \ref{thm:cond-loc} is the following proposition. 
\begin{proposition}\label{prop:loc-structure}
Let $G=(V,E)$ be a graph with maximum degree at most $d$, 
such that 
\begin{equation}\label{eq:cycle-and-edge-bd}
N(C_\ell, G) \geq (\delta -\tau_\star) d^\ell \qquad \text{ and } \qquad 2 e(G) \leq \left( \lfloor \delta \rfloor + \{\delta\}^{2/\ell} +\tau_\star\right) d^2,
\end{equation}
for some $\delta >0$, $\tau_\star \in (0,1)$, and $\ell \geq 3$ an integer. Then, there exist some positive constants $\underline{c}$ and $\overline{c}$, depending only on $\ell$ and $\delta$, and some absolute constant $c \in (0,1)$ such that for any $\tau_\star \in (\underline{c}/d, \overline{c})$, 
the following properties hold:
\begin{enumerate}
\item[(i)] Let $\ell \in 2 \N +1$. There exist pairwise disjoint subsets $\{U_i\}_{i\in \set{\lceil \delta \rceil}}$ of $V$ such that
\[
\min_{i \in \set{\lfloor \delta \rfloor}} |U_i| \geq (1-\tau_\star^c)d, \quad |U_{\lceil \delta \rceil}| \geq (1-\tau_\star^c) \{\delta\}^{1/\ell} d, \text{ and } \min_{i \in \set{\lceil \delta \rceil}} \min_{u \in U_i} \deg_{G[U_i]}(u)/|U_i| \geq (1-\tau_\star^c).
\]
\item[(ii)] Let $\ell \in 2 \N$. There exist a nonnegative integer $k' \leq \lfloor \delta \rfloor/2$,  and a collection of pairwise disjoint subsets $\{U_i\}_{i \in \set{\lceil \delta'\rceil}}$, $\{V_i\}_{i \in \set{k'}}$, and $\{W_i\}_{i \in \set{k'}}$ of $V$, where $\delta':=\delta -2k'$, such that 
\begin{multline*}
\min\left\{\min_{i \in \set{\lfloor \delta'\rfloor}} |U_i|, \min_{i \in \set{k'}} \min\left\{ |V_i|, |W_i|\right\} \right\}\geq (1-\tau_\star^{c})d, \quad |U_{\lceil \delta'\rceil}| \geq (1-\tau_\star^{ c}) \{\delta\}^{1/\ell} d, \quad \text{ and }\\
\min\left\{ \min_{i \in \set{\lceil \delta '\rceil}} \min_{u \in U_i} \deg_{G[U_i]}(u)/|U_i|, \min_{i \in \set{k'}} \min_{u \in V_i} \deg_{G[V_i,W_i]}(u)/|W_i|, \min_{i \in \set{k'}} \min_{u \in W_i} \deg_{G[V_i,W_i]}(u)/|V_i| \right\} \geq (1-\tau_\star^{ c}).
\end{multline*}
\end{enumerate}
\end{proposition}

\dteal{Deferring the proof Proposition \ref{prop:loc-structure} for the moment let us provide the proof of Theorem \ref{thm:cond-loc}.}

\begin{proof}[Proof of Theorem \ref{thm:cond-loc}]
\dteal{First consider the case $\Delta \geq 3$. Recall \eqref{eq:rate function defn} and \eqref{eq:mathcal-G-0}. Observe that for any $G \in \mathcal G_0$ we have $N(H,G) \geq (1- 4 C v_H\varepsilon)(2 e(G))^{v_H/2}$, for $\varepsilon >0$ sufficiently small. Hence, applying \cite[Lemma A.2]{BK25} and Lemma \ref{lem:minimum edge bound} we deduce that for any such graph $G$ there exists a subgraph $G' \subseteq G$ such that $d_{\min}(G') \geq (1-8 C v_H \varepsilon^{1/2})\delta^{1/v_H} np^{\Delta/2}$, where $d_{\min}(\cdot)$ denotes the minimum degree in a graph. As $|V(G')| \geq d_{\min}(G')$ and $d_{\min}(G') \cdot |V(G')| \leq 2 e(G') \leq 2 e(G)$ and $G \in \mathcal G_0$, the required probability estimate follows from Proposition \ref{prop:loc-cond}. 

Next consider $\Delta =2$. Notice that any $G \in \mathcal G_0$ satisfies the assumption \eqref{eq:cycle-and-edge-bd} with $\tau_\star = \lceil \delta \rceil C \varepsilon$. Under our assumption $d \gg 1$. Therefore, by Propositions \ref{prop:loc-cond} and \ref{prop:loc-structure} the required probability lower bound is immediate, completing the proof.} 
%
\end{proof}

\dteal{We now return to the proof of Proposition \ref{prop:loc-structure}. It} will be done in a few stages. In the first step, we show that \eqref{eq:cycle-and-edge-bd} implies certain rigidity in the spectrum of the adjacency matrices associated with such graphs. In the second step we will show that if a graph, whose maximum degree is bounded by $d$, has eigenvalues close to $d$ then it must contain components that are nearly regular and {\em loosely connected}, \dteal{and have certain special structures,} 
and furthermore the number of such components is same as the algebraic multiplicity of eigenvalues that are near $d$ in the original graph. In the third step we show that a graph that is nearly regular and can have only two {\em macroscopically} large eigenvalues \dteal{only if they are close to each other in absolute values.} 

Combining these three pieces we will then complete the proof of Proposition \ref{prop:loc-structure}. 


\subsection{Rigidity in spectrum} For a graph $G$ we write $\{\widetilde \lambda_i(G)\}$ and $\{\lambda_i(G)\}$ to denote the eigenvalues of its adjacency matrix $A(G)$ arranged in a nonincreasing order, and in a nonincreasing order of their absolute values, respectively. The next lemma shows that any graph $G$ with maximum degree at most $d$ that satisfy \eqref{eq:cycle-and-edge-bd}, must have some near extremal structure in its spectrum.

\begin{lemma}\label{lem:loc-structure-spectral}
Let $G$ be as in Proposition \ref{prop:loc-structure}. Then, there exists an absolute constant $c_\star \in (0,1)$ such that for $\tau_\star$ sufficiently small, depending only on $\delta$ and $\ell$, the following properties hold:
\begin{enumerate}
\item[(i)] Assume $\ell \in 2\N+1$. Then 
\begin{equation}\label{eq:lambda-near-1-1}
d \geq \widetilde \lambda_1(G) \geq \widetilde\lambda_2(G) \geq \cdots \geq \widetilde\lambda_{\lfloor \delta \rfloor}(G) \geq (1 -  \tau_\star^{c_\star})d 
\end{equation}
and
\begin{equation}\label{eq:residue-eval}  
\sum_{i \geq \lceil \delta \rceil +1} \lambda_i^2(G) \leq  \tau_\star^{c_\star} d^2. 
\end{equation}
Additionally, for $\delta \notin \N$, 
\begin{equation}\label{eq:lambda-fraction-1-1}
\widetilde\lambda_{\lceil \delta \rceil}(G)/d \in [\{\delta\}^{1/\ell} - \tau_\star^{c_\star}, \{\delta\}^{1/\ell} + \tau_\star^{c_\star}].
\end{equation}
\item[(ii)] Assume $\ell \in 2\N$. Then
\begin{equation}\label{eq:lambda-near-1-2}
d \geq  |\lambda_1(G)| \geq |\lambda_2(G)| \geq \cdots \geq |\lambda_{\lfloor \delta \rfloor}(G)| \geq (1 -  \tau_\star^{c_\star})d 
\end{equation}
and \eqref{eq:residue-eval} holds. 
Additionally, for $\delta \notin \N$, 
\begin{equation}\label{eq:lambda-fraction-1-2}
|\lambda_{\lceil \delta \rceil}(G)|/d \in [\{\delta\}^{1/\ell} - \tau_\star^{c_\star}, \{\delta\}^{1/\ell} + \tau_\star^{c_\star}].
\end{equation}
\end{enumerate}
\end{lemma}

The proof of Lemma \ref{lem:loc-structure-spectral} relies on the following two results.  

\begin{lemma}\label{lem:optimization}
Fix  and integer $\ell \geq 3$, $c_0 \in (0,1]$, and $c_1 > 0$. Then, for any integer $s \geq \lceil c_1/c_0\rceil$,
\begin{equation}\label{eq:optimization0}
\inf_{(x_1,x_2,\ldots, x_s) \in \R^s} \left\{ \sum_{i\in \set{s}} x_i^{2/\ell} \;:\; \sum_{i \in \set{s}} x_i \geq c_1 \ \text{ and } \ x_i \in \left[(-1)^{{\bf 1}_{\{\ell \in 2\N+1\}}}, c_0\right] \right\}
= \left\lfloor \frac{c_1}{c_0} \right\rfloor c_0^{2/\ell}
+  \left\{ \frac{c_1}{c_0} \right\}^{2/\ell} c_0^{2/\ell}.  
\end{equation}
\end{lemma}

\begin{lemma}\label{lem:L}
Fix $k \in \N$,  and $h \in (0,1)$. For $\theta \in (-1,\infty)$, set $x^\star_\theta:=k(k+\theta)/(k+1)$ and define 
\begin{equation}\label{eq:define-L}
L_{k,\theta}(x):= \lfloor x \rfloor + \{x\}^h + \left(\frac{x}{k}\right)^h \cdot \left(\left\lfloor \frac{(k+\theta -x) k}{x}\right\rfloor + \left\{ \frac{(k+\theta -x) k}{x}\right\}^h \right), \quad \mbox{ for } x \in (0,\min\{k, k+\theta\}].  
\end{equation}
Then, for any $\theta \in (-k^{-1},1)$, $L_{k,\theta}(x) \geq L_{k,\theta}(x^\star_\theta)$ for all $x \in (0, x^\star_\theta]$. 
\end{lemma}

Deferring the proofs of Lemmas \ref{lem:optimization} and \ref{lem:L}, we now prove Lemma \ref{lem:loc-structure-spectral}.

\begin{proof}[Proof of Lemma \ref{lem:loc-structure-spectral}]
\dteal{For ease in writing, define} $\widetilde\eta_i(G):=(\widetilde\lambda_i(G)/d)^\ell$. For \dteal{a} later use, let us notice that $\{\widetilde \eta_i(G)\}$ are nonnegative for $\ell \in 2\N$, while for $\ell \in 2 \N+1$, a priori there is no such restriction on $\{\widetilde \eta_i(G)\}$. Let $\{\eta_i(G)\}_{i \in \set{|V|}}$ be an nonincreasing rearrangement of $\{\widetilde \eta_i(G)\}_{i \in \set{|V|}}$. Also notice that, by Perron-Frobenius theorem, $\max_i |\eta_i(G)| \leq 1$. 

\dteal{Recall ${\rm Hom}(H,G)$ is the number of homomorphisms of $H$ in $G$.} Since 
\[
\sum_{i\in \set{|V|}} \lambda_i(G)^\ell = {\rm Hom}(C_\ell, G) \geq N(C_\ell, G) \qquad \text{ and } \qquad 2 e(G)= \sum_{i\in \set{|V|}} \lambda_i^2(G),
\]
we have from \eqref{eq:cycle-and-edge-bd} that
\begin{equation}\label{eq:ineq-in-eta}
\sum_{i \in \set{|V|}} \eta_i(G) \geq \delta -\tau_\star=:  \delta_{\tau_\star} \qquad \text{ and } \qquad \sum_{i \in \set{|V|}} \eta_i^{2/\ell}(G) \leq \lfloor \delta\rfloor  +\{\delta\}^{2/\ell}+ \tau_\star.
\end{equation}
\dteal{Observe that, as $\eta_i(G)$ for $\ell \in 2 \N$, the quantity $\eta_i^{2/\ell}(G)$ is well defined.}

Using the bounds \eqref{eq:ineq-in-eta} we \dteal{will prove} that 
\begin{equation}\label{eq:eta-near-1}
1 \geq \eta_1(G) \geq \eta_2(G) \geq \cdots \geq \eta_{\lfloor \delta \rfloor}(G) \geq 1 -  \tau_\star^a,
\end{equation}
and, for $\delta \notin \N$,
\begin{equation}\label{eq:eta-fraction}
\eta_{\lceil \delta \rceil}(G) \in [\{\delta\} - \tau_\star^{\bar a}, \{\delta\} + \tau_\star^{\bar a}],
\end{equation}
for some absolute constants $a, \bar a \in (0,1)$ and $\tau_\star$ sufficiently small depending on $\delta$ and $\ell$.
%
Later we will argue that the desired bounds on $\{\lambda_i(G)\}_{i \in \set{|V|}}$ follow from \eqref{eq:eta-near-1}-\eqref{eq:eta-fraction}. 

For the rest of the proof, for ease in writing, we write $\eta_i$ instead of $\eta_i(G)$.
Postponing the proof \eqref{eq:eta-near-1} for \dteal{the moment}, we proceed to prove \eqref{eq:eta-fraction}.  To this end, by \eqref{eq:ineq-in-eta} and \eqref{eq:eta-near-1} we find that 
\begin{equation}\label{eq:eta-sum-residue}
\eta_{\lceil \delta \rceil}^{2/\ell} \leq \sum_{i \geq \lceil \delta \rceil} \eta_i^{2/\ell} \leq \lfloor \delta\rfloor  + \{\delta\}^{2/\ell} +  \tau_\star - \sum_{i \leq \lfloor \delta \rfloor} \eta_i^{2/\ell} \leq \{\delta\}^{2/\ell} +\lceil \delta \rceil \tau_\star^a.
\end{equation}
This implies that $\eta_{\lceil \delta \rceil} \leq |\eta_{\lceil \delta \rceil}| \leq \{\delta\} + \tau_\star^{a/2}$, for $\tau_\star$ sufficiently small, depending on $\delta$ and $\ell$, yielding the desired upper bound in \eqref{eq:eta-fraction}. 

Next, using the upper bound $\eta_i \leq 1$ we also obtain from \eqref{eq:ineq-in-eta} that
\[
\sum_{i \geq \lceil \delta \rceil} \eta_i \geq \{\delta\} - \tau_\star =: \widetilde \delta_{\tau_\star}>0,
\]
for $\tau_\star$ small. This, in particular, implies that $z:=\eta_{\lceil \delta \rceil} >0$. As $\eta_i \leq z$ for $i \geq \lceil \delta \rceil$, applying Lemma \ref{lem:optimization}, as $\ell \geq 3$, we find that 
\[
\sum_{i \geq \lceil \delta \rceil} \eta_i^{2/\ell} \geq z^{2/\ell} \cdot( \lfloor \widetilde \delta_{\tau_\star}/z \rfloor + \{ \widetilde\delta_{\tau_\star}/z\}^{2/\ell}) \geq z^{2/\ell-1} \widetilde\delta_{\tau_\star}. 
\]
The last bound together with \eqref{eq:eta-sum-residue} implies that $z^{2/\ell-1} \leq \{\delta\}^{2/\ell-1} (1+\tau_\star^{a/2})$. This yields the desired lower bound on $\eta_{\lceil \delta \rceil}$.

We now proceed to prove \eqref{eq:eta-near-1}. Assume $\delta \geq 1$, otherwise there is nothing to be proved. Let $\{\eta_i\}$ be such that $\sum_{i \leq \lfloor \delta \rfloor} \eta_i= x_\star$, for some $x_\star$. We claim that $x_\star >0$. If not, then there exists some $i \in \set{\lfloor \delta\rfloor}$ such that $\eta_i \leq 0$. Hence, $\eta_j \leq 0$ for all $j \geq \lceil \delta \rceil$, and therefore $\sum_i \eta_i \leq x_\star\leq 0$ contradicting \eqref{eq:ineq-in-eta}, for $\tau_\star$ small. \dteal{As $\eta_i \leq 1$, applying} Lemma \ref{lem:optimization} we find that
\begin{equation}\label{eq:eta-large-lbd}
\sum_{i \leq \lfloor \delta \rfloor} \eta_i^{2/\ell} \geq \lfloor x_\star \rfloor + \{x_\star\}^{2/\ell}. 
\end{equation}
To move forward, without loss of generality we may assume $x_\star \leq \lfloor \delta \rfloor -\tau_\star^{1/2}$, as otherwise \eqref{eq:eta-near-1} is immediate. The last assumption, together with \eqref{eq:ineq-in-eta}, imply that 
\[
\sum_{i \geq \lfloor \delta \rfloor+1} \eta_i \geq  \delta_{\tau_\star} - x_\star =: \delta_{\tau_\star, x_\star} >0 \qquad \text{ and } \qquad \eta_{\lfloor \delta \rfloor+1}  \leq \frac{1}{\lfloor \delta \rfloor} \sum_{i \leq \lfloor \delta \rfloor}\eta_i = \frac{x_\star}{\lfloor \delta \rfloor}=:\widehat x_\star.
\]
Applying Lemma \ref{lem:optimization} again, we find that
\begin{equation}\label{eq:eta-small-lbd}
\sum_{i \geq \lfloor \delta \rfloor+1} \eta_i^{2/\ell} \geq \widehat x_\star^{2/\ell} \left\lfloor \frac{\delta_{\tau_\star, x_\star}}{\widehat x_\star}\right\rfloor + \widehat x_\star^{2/\ell} \left\{ \frac{\delta_{\tau_\star, x_\star}}{\widehat x_\star}\right\}^{2/\ell}. 
\end{equation}
Combining \eqref{eq:ineq-in-eta}, \eqref{eq:eta-large-lbd}, and \eqref{eq:eta-small-lbd}, and recalling \eqref{eq:define-L} we obtain
\begin{equation}\label{eq:L-ubd}
L_{\lfloor \delta \rfloor, \theta_\delta}(x_\star) \leq \sum_{i \in \set{|V|}} \eta_i^{2/\ell} \leq \lfloor \delta\rfloor  +\{\delta\}^{2/\ell}+ \tau_\star,
\end{equation}
for any $x_\star$ such that $\sum_{i \leq \lfloor \delta \rfloor} \eta_i= x_\star$, 
where $\theta_\delta:= \{\delta\} -\tau_\star \in (-1/2\lfloor \delta \rfloor, 1)$, for all $\tau_\star$ sufficiently small. 

\dteal{To arrive at a contradiction we set
\[
\widetilde L(x):= \lfloor \delta \rfloor -1 + (x - \lfloor \delta \rfloor +1)^{2/\ell} +  (\delta_{\tau_\star} -x)^{2/\ell}, \qquad x \in [\lfloor \delta \rfloor -1, \delta_{\tau_\star}]
\]
and note that $L_{\lfloor \delta \rfloor, \theta_\delta}(\cdot) = \widetilde L(\cdot)$ on $[x^\star_{\theta_\delta}, \lfloor \delta \rfloor -{\tau_\star}]$, where we recall that $x^\star_{\theta_\delta}= \lfloor \delta \rfloor \delta_{\tau_\star}/(\lfloor \delta \rfloor +1)$. 

Therefore,}
\begin{multline}\label{eq:L-to-wtL}
\inf_{x \in (0, \lfloor \delta \rfloor -\tau_\star^{1/2}]} L_{\lfloor \delta \rfloor, \theta_\delta}(x) = \min \left\{ \inf_{x \in (0,x_{\theta_\delta}^\star]} L_{\lfloor \delta \rfloor, \theta_\delta}(x), \inf_{x \in [x_{\theta_\delta}^\star, \lfloor \delta \rfloor -\tau_\star^{1/2} ]} L_{\lfloor \delta \rfloor, \theta_\delta}(x)\right\} \\
= \min \left\{ L_{\lfloor \delta \rfloor, \theta_\delta}(x^\star_{\theta_\delta}), \inf_{x \in [x_{\theta_\delta}^\star, \lfloor \delta \rfloor -\tau_\star^{1/2} ]} \widetilde L(x)\right\} = \min \left\{ \widetilde L(x^\star_{\theta_\delta}), \widetilde L(\lfloor \delta \rfloor -\tau_\star^{1/2})\right\},
\end{multline}
\dteal{where the penultimate equality in \eqref{eq:L-to-wtL} follows from Lemma \ref{lem:L},  
and the last equality in \eqref{eq:L-to-wtL} follows upon noting that $\widetilde L(\cdot)$ is strictly concave on its domain.} 

We will show that  
\begin{equation}\label{eq:wtL-min}
\min \left\{ \widetilde L(x^\star_{\theta_\delta}), \widetilde L(\lfloor \delta \rfloor -\tau_\star^{1/2})\right\} \geq \lfloor \delta \rfloor + \{\delta \}^{2/\ell} + \tau_\star^{2/3}.
\end{equation}
Combining \eqref{eq:wtL-min} with \eqref{eq:L-to-wtL} and \eqref{eq:L-ubd} we find that it contradicts the assumption that $x_\star \leq \lfloor \delta \rfloor -\tau_\star^{1/2}$. Hence, we must have $x_\star > \lfloor \delta \rfloor - \tau_\star^{1/2}$, and therefore \eqref{eq:eta-near-1} must hold.

We now turn to the proof of \eqref{eq:wtL-min}. It is elementary. Indeed, using the concavity of the map $\lambda \mapsto (u+\lambda)^{2/\ell}$ we find that
\begin{equation}\label{eq:wt-L-right}
\widetilde L(\lfloor \delta \rfloor -\dteal{\tau_\star^{1/2}}) = \lfloor \delta \rfloor -1 + (1-\tau_\star)^{2/\ell} + (\{\delta\} - \tau_\star + \tau_\star^{1/2})^{2/\ell} \geq \lfloor \delta \rfloor + \{\delta \}^{2/\ell} + \tau_\star^{2/3}, 
\end{equation}
for $\tau_\star$ sufficiently small depending only on $\delta$ and $\ell$. On the other hand, as the map $x \mapsto (x+r)^h-x^h$ is strictly decreasing on $[0,\infty)$ for $r >0$ and $h \in (0,1)$ it follows that   
\begin{equation*}
\begin{aligned}
(1+\{\delta\} \lfloor \delta \rfloor)^{2/\ell} +  (\lfloor \delta \rfloor + \{\delta\})^{2/\ell}  & \geq (1+c_{\delta, \ell})\left( \{\delta \}^{2/\ell}  (\lfloor \delta \rfloor +1)^{2/\ell}+ (\lfloor \delta \rfloor +1)^{2/\ell} \right)\\
& = (1+c_{\delta, \ell}) (\{\delta \}^{2/\ell}+1)  (\lfloor \delta \rfloor +1)^{2/\ell},
\end{aligned}
\end{equation*}
for some constant $c_{\delta, \ell} >0$, depending on $\delta$ and $\ell$. \dteal{As
\[
\widetilde L(x_{\theta_\delta}^\star) = \lfloor \delta \rfloor -1 + (\lfloor \delta \rfloor +1)^{-2/\ell} \cdot \left( \left(1+(\{\delta\}-\tau_\star) \lfloor \delta \rfloor\right)^{2/\ell} + \left( \delta -\tau_\star\right)^{2/\ell}\right),
\]
the last observation implies that} there exists some $\tau_\star >0$, sufficiently small, depending only on $\delta$ and $\ell$ such that $\widetilde L(x^\star_{\theta_\delta}) \geq \lfloor \delta \rfloor + \{\delta \}^{2/\ell} + \tau_\star^{2/3}$. \dteal{Together with \eqref{eq:wt-L-right} this yield} \eqref{eq:wtL-min}. Thus, we now finally have \eqref{eq:eta-near-1}-\eqref{eq:eta-fraction}

%
%

It remains to use \eqref{eq:eta-near-1}-\eqref{eq:eta-fraction} and derive the desired bounds on the eigenvalues. It is immediate. Indeed, as $\eta_i(G)= (\widetilde \lambda_i(G)/d)^\ell$ for $\ell \in 2 \N+1$ and $\eta_i(G)= |\lambda_i(G)/d|^\ell$ for $\ell \in 2\N$, \dteal{the bounds in \eqref{eq:lambda-near-1-1}, \eqref{eq:lambda-fraction-1-}, \eqref{eq:lambda-near-1-2} and \eqref{eq:lambda-fraction-1-2} follow}. The upper bound in \eqref{eq:residue-eval} follows from the upper bound on $e(G)$ in \eqref{eq:cycle-and-edge-bd} and the lower bounds in \eqref{eq:eta-near-1}-\eqref{eq:eta-fraction}.
%
%
This completes the proof. 
\end{proof}

We next turn to the proofs of Lemmas \ref{lem:optimization} and \ref{lem:L}. 
The proof of Lemma \ref{lem:optimization} is straightforward and is a generalization of \cite[Lemma 2.2]{BD}. We include it for completeness. 

\begin{proof}[Proof of Lemma \ref{lem:optimization}] 
Fix $\ell \in 2 \N+1$. Since $\sum_i \max\{x_i,0\} \geq \sum_i x_i$ and $\sum_i (\max\{x_i,0\})^{2/\ell} \leq \sum_i x_i^{2/\ell}$, any $(x_1, x_2, \ldots, x_s)$ such that $\min_i x_i < 0$ cannot be an optimizer to the variational problem in \eqref{eq:optimization0}. Therefore, irrespective of $\ell$, it suffices to prove that 
\begin{equation}\label{eq:optimization1}
\inf_{(x_1,x_2,\ldots, x_s) \in \R^s} \left\{ \sum_{i\in \set{s}} x_i^{2/\ell} \;:\; \sum_{i \in \set{s}} x_i \geq c_1 \ \text{ and } \ x_i \in \left[0, c_0\right] \right\}
= \left\lfloor \frac{c_1}{c_0} \right\rfloor c_0^{2/\ell}
+ \left( c_1 - c_0 \left\lfloor \frac{c_1}{c_0} \right\rfloor \right)^{2/\ell}.
\end{equation}
As the map $(x_1,x_2, \ldots,x_s) \mapsto \sum_{i \in \set{s}} x_i^{2/\ell}$ is strictly concave for any pair of variables,  any optimizer can contain at most one non extremal coordinate. This completes the proof. 
\end{proof}

\dteal{To prove Lemma \ref{lem:L} we observe that the function $L_{k,\theta}(\cdot)$, although not concave on the whole domain, is concave on each block of a certain partition of the interval $(0, x_\theta^\star]$, where each block 
is again some interval. 
Hence, it boils down to comparing the values of $L_{k,\theta}(\cdot)$ at the end points of those sub intervals. However, this turns out to be tedious. On the other hand, we also notice that by enlarging the domain $L_{k,\theta}(x^\star_\theta)$ can be expressed as the sum $\sum_{i} \mathfrak{u}_i(x_\theta^\star)^h$, where $\mathfrak{u}(x):=\{\mathfrak{u}_i(x)\}$ is a certain transformation of $x$ such that each coordinate 
lie in the interval $[0,1]$, with at most two coordinates of $\{\mathfrak{u}_i(x_\theta^\star)\}$ non extremal. This suggests the following alternate route: view $L_{k,\theta}(\cdot)$ as a Schur-concave function of the transform $x \mapsto \mathfrak{u}(x)$, 
and show that $\mathfrak{u}(x^\star_\theta)$ majorizes $\mathfrak{u}(x)$ for all relevant $x$. We take this less tedious route. Below we carry out the details.
%
}

\begin{proof}[Proof of Lemma \ref{lem:L}]
%
%
\dteal{Recalling} $x^\star_\theta:=k(k+\theta)/(k+1)$, for $K \in \N$, set 
\[
\mathcal{T}_{K, \theta}:= \{ x \in (0, x^\star_\theta]: \lfloor \mathfrak{s}_{k,\theta}(x) \rfloor \leq K\}, \qquad \mbox{ where } \quad \mathfrak{s}_{k,\theta}(x):= \frac{(k+\theta -x)k}{x}. 
\] 
Observe that $\mathfrak{s}_{k,\theta}(x^\star_\theta)=1$. Therefore, it suffices to show that $L_{k,\theta}(x) \geq L_{k,\theta}(x^\star_\theta)$ for all $x \in \mathcal{T}_{K, \theta}$ and any $K \in \N $. 

To this end, for any $x \in \mathcal{T}_{K, \theta}$,  define
\[
\mathfrak{u}(x):= \left(\underbrace{1,1,\ldots,1}_{\lfloor x \rfloor}, \{x\}, \underbrace{x/k, x/k, \ldots, x/k}_{\lfloor{\mathfrak{s}_{k,\theta}(x)}\rfloor}, (x/k) \cdot \{\mathfrak{s}_{k,\theta}(x)\},\underbrace{0,0,\ldots,0}_{k+K- \lfloor x \rfloor -\lfloor \mathfrak{s}_{k, \theta}(x) \rfloor}\right) \in [0,1]^{k'},
\]
where, for brevity, we write $k':=k+K+2$. Notice that $L_{k,\theta}(x)=\Phi(\mathfrak{u}(x))$, where for any $\widehat{\mathfrak{u}}=(\widehat{\mathfrak{u}}_1, \widehat{\mathfrak{u}}_2, \ldots, \widehat{\mathfrak{u}}_{k'}) \in [0,1]^{k'}$, $\Phi(\widehat{\mathfrak{u}}):=\sum_{i\in \set{k'}} \widehat{\mathfrak{u}}_i^h$. Since the map $\widehat{\mathfrak{u}} \mapsto \Phi(\widehat{\mathfrak{u}})$ is permutation invariant and concave (as $h \in [0,1)$), it follows that it is also Schur-concave (see \cite[Proposition III.C.1]{MOA2011}). That is, for any $\widehat{\mathfrak{u}}, \widetilde{\mathfrak{u}}=(\widetilde{\mathfrak{u}}_1, \widetilde{\mathfrak{u}}_2, \ldots, \widetilde{\mathfrak{u}}_{k'}) \in [0,1]^{k'}$ such that $\widetilde{\mathfrak{u}} \prec \widehat{\mathfrak{u}}$ one has $\Phi(\widehat{\mathfrak{u}}) \leq \Phi(\widetilde{\mathfrak{u}})$, where

\[
\widetilde{\mathfrak{u}}\prec\widehat{\mathfrak{u}} \iff \sum_{i\in \set{\ell}} \widetilde{\mathfrak{u}}^{\downarrow}_{i}\leq \sum_{i\in \set{\ell}}\widehat{\mathfrak{u}}^{\downarrow}_{i}, \quad \mbox{ for } \ell \in \set{k'-1} \qquad \text{  and } \qquad \sum_{i \in \set{k'}}\widetilde{\mathfrak{u}}_{i}=\sum_{i\in \set{k'}}\widehat{\mathfrak{u}}_{i},
\]
and for any $\mathfrak{u}=(\mathfrak{u}_1, \mathfrak{u}_2, \ldots, \mathfrak{u}_{k'}) \in [0,1]^{k'}$, the coordinates of the vector $\mathfrak{u}^{\downarrow}= (\mathfrak{u}_1^\downarrow, \mathfrak{u}_2^\downarrow, \ldots, \mathfrak{u}_{k'}^\downarrow)$ are rearrangements of the coordinates of $\mathfrak{u}$ so that $\mathfrak{u}_1^\downarrow \geq \mathfrak{u}_2^\downarrow\geq  \cdots \geq \mathfrak{u}_{k'}^\downarrow$. 

Hence, it suffices to show that $\mathfrak{u}(x) \prec \mathfrak{u}(x^\star)$, for any $x \in \mathcal{T}_{K, \theta}$. 
Turning to prove the same, observe that, by definition $\sum_{i\in \set{k'}} (\mathfrak{u}(x))_i = k+\theta$ for all $x \in \mathcal{T}_{K, \theta}$. Next, as $(\mathfrak{u}(x))_i \in [0,1]$ for all $i \in \set{k'}$ and $x \in \mathcal{T}_{K, \theta}$, and as $\theta \in (- k^{-1},1)$ implies $\lfloor x^\star_\theta \rfloor=k-1$, 
we also have 
\[
\sum_{i \in \set{\ell}} (\mathfrak{u}(x))^{\downarrow}_{i} \leq \ell = \sum_{i \in \set{\ell}} (\mathfrak{u}(x^\star_\theta))^{\downarrow}_{i}, \qquad \text{ for any } \qquad \ell \in \set{k-1}.
\]
Since $\mathfrak{s}_{k, \theta}(x^\star_\theta)=1$ and $x_\theta^\star \notin \N$ (use $\theta \in(- k^{-1},1)$), we see that $\mathfrak{u}(x^\star_\theta)$ consists of only $k+1$ nonzero terms. Hence, we moreover have
\[
\sum_{i \in \set{\ell}} (\mathfrak{u}(x))^{\downarrow}_{i} \leq \sum_{i \in \set{k'}} (\mathfrak{u}(x))_{i} = k+\theta = \sum_{i \in \set{\ell}} (\mathfrak{u}(x^\star_\theta))^{\downarrow}_{i}, \qquad \text{ for any } \qquad \ell \in \set{k'}\setminus \set{k}.
\]
To complete the proof it remains to argue that 
\begin{equation}\label{eq:majorize}
\sum_{i \in \set{k}} (\mathfrak{u}(x))^{\downarrow}_{i} \leq \sum_{i \in \set{k}} (\mathfrak{u}(x^\star_\theta))^{\downarrow}_{i}.
\end{equation}
Towards obtaining \eqref{eq:majorize} we observe that   
\begin{equation*}
\sum_{i \in \set{k}} (\mathfrak{u}(x))^{\downarrow}_{i} \leq  \lfloor x \rfloor  +(k-1 -\lfloor x \rfloor) \cdot (x/k) + \max\{x/k, \{x\}\} \leq k-1 + \max\{x/k, \{x\}\}.
\end{equation*}
Since, the map $x \mapsto \max\{x/k, \{x\}\}$ is increasing for $x \in [k-1,k)$ the bound in \eqref{eq:majorize} follows from above, for any $x \in \mathcal{T}_{K, \theta}$ be such that $\lfloor x \rfloor = k-1$. Next, fix any $x \in \mathcal{T}_{K, \theta}$ such that 
$\lfloor x \rfloor < k-1$. In this case, we derive from above that  
\[
\sum_{i \in \set{k}} (\mathfrak{u}(x))^{\downarrow}_{i} \leq k-2 + x/k + \max\{x/k, \{x\}\} \leq k -1 +x/k < k-1+x^\star_\theta/k = \sum_{i \in \set{k}} (\mathfrak{u}(x^\star_\theta))^{\downarrow}_{i},
\]
yielding \eqref{eq:majorize}, where in the last step we also used $ \{x_\theta^\star\} \leq x_\theta^\star/k$. Hence, we have shown that $\mathfrak{u}(x) \prec \mathfrak{u}(x^\star_\theta)$, for any $x \in \mathcal{T}_{K, \theta}$ and $\theta \in (-k^{-1},1)$. This completes the proof of the lemma.
\end{proof}



\subsection{Structural rigidity via higher order Cheeger and dual Cheeger inequalities for weighted graphs} The following two results are the main results of this subsection. 

\begin{proposition}\label{prop:cheeger-applied}
Let $G=(V,E)$ be a graph with maximum degree at most $d$. Further assume there exists some $k \in \N$ such that
\begin{equation}\label{eq:lambda-assumption-cheeger}
d \geq \widetilde \lambda_1(G) \geq \widetilde \lambda_2(G) \geq \cdots \geq \widetilde \lambda_k(G) \geq (1 -\tau_0)d,
\end{equation}
for some $\tau_0 >0$. Then, there exist an absolute constant $\widehat C <\infty$, and a collection of pairwise disjoint subsets $\{S_i\}_{i \in \set{k}}$ of $V$ such that, for any $\tau_0$ sufficiently small,
\begin{equation}\label{eq:ho-cheeger}
2e (G[S_i]) \geq \left(1- \widehat C k^3 \sqrt{\tau_0}\right) d|S_i| \quad \text{ and } \quad e(G[S_i, V \setminus S_i]) \leq \widehat C k^3 \sqrt{\tau_0} d |S_i|, \quad \text{ for all } i \in \set{k}.
\end{equation}
\end{proposition}

\begin{proposition}\label{prop:dual-cheeger-applied}
Let $G=(V,E)$ be a graph with maximum degree at most $d$. Further assume there exists some $k \in \N$ such that
\begin{equation}\label{eq:lambda-assumption-cheeger-dual}
- d \leq \widehat \lambda_1(G) \leq \widehat \lambda_2(G) \leq \cdots \leq \widehat \lambda_k(G) \leq -(1 -\tau_0)d,
\end{equation}
for some $\tau_0 >0$, where for brevity we have used the notation $\widehat \lambda_i(G):= \widetilde \lambda_{|V(G)|-i+1}(G)$ for $i \in \set{|V(G)}$. Then, there exist an absolute constant $\widehat C <\infty$, and a collection of pairwise disjoint subsets $\{S_i\}_{i \in \set{2k}}$ of $V$ such that, for any $\tau_0$ sufficiently small,
\begin{multline}\label{eq:ho-dual-cheeger}
e (G[S_{2i-1},S_{2i}])/d \geq \left(1- \widehat C k^3 \sqrt{\tau_0}\right) \max\left\{|S_{2i-1}|,|S_{2i}|\right\} \\
\geq \left(1- \widehat C k^3 \sqrt{\tau_0}\right) \min\left\{|S_{2i-1}|,|S_{2i}|\right\} \geq \left(1- 2\widehat C k^3 \sqrt{\tau_0}\right) d, \quad \text{ for all } i \in \set{k}.
\end{multline}
\end{proposition}

Notice that Lemma \ref{lem:loc-structure-spectral} ensures for any graph $G$ as in Proposition \ref{prop:loc-structure} there must be a certain number of eigenvalues of $A(G)$ that are close to $\pm d$. Propositions \ref{prop:cheeger-applied} and \ref{prop:dual-cheeger-applied} show that such a rigidity in the spectrum must enforce certain clique or biclique like structures in $G$. 

\dteal{Before proving these two porisitions let us note that these two results 
maybe} 
considered as stability versions of the Perron-Frobenius theorem. Indeed, if the maximum degree is bounded by $d$ and if $\widetilde \lambda_k(G)=d$ for some $k \in \N$, then it follows from Perron-Frobenius theorem that $G$ must contain $k$ disjoint $d$-regular components. On the other hand, $\widehat\lambda_{\ell}(G)=-d$ forces $\ell$ disjoint $d$-regular bipartite components. Setting $\tau_0=0$ in \eqref{eq:ho-cheeger} and \eqref{eq:ho-dual-cheeger} one recovers these two conclusions. The content of Propositions \ref{prop:cheeger-applied} and \ref{prop:dual-cheeger-applied} is that this picture is stable under perturbation, that is, eigenvalues close to $\pm d$ force the corresponding extremal structures to be present in approximate forms.

Proofs of Propositions \ref{prop:cheeger-applied} and \ref{prop:dual-cheeger-applied} will employ higher order Cheeger and dual Cheeger inequalities for weighted graphs. It requires us introduce quite a few notation: 
Let $\widetilde G=(\widetilde V,\widetilde E,w)$ be a finite, undirected graph, with positive weights 
$w:\widetilde E\mapsto(0,\infty)$ on the edges. We extend the definition of $w$ to $\widetilde V$ by setting $w(v):=\sum_{u: u \sim v} w(u)$ for all $v \in \widetilde V$. Denote, 
$w(F):=\sum_{e\in F} w(e)$ for $F \subseteq \widetilde E$ and $w(S):=\sum_{v \in S} w(v)$ for $S \subseteq \widetilde V$. Next, for $S \subseteq \widetilde V$ we define its Dirichlet conductance as $\phi_{\widetilde G}(S):= w(E(S, \bar S))/w(S)$, where for an $S \subseteq \widetilde V$ the set $E(S, \bar S)$ is the collection of edges that has one and only one end point in $S$, with the convention that $\phi_{\widetilde G}(\emptyset):=\infty$.  We will also assume that there is no isolated vertex in $\widetilde G$ so that $\phi_{\widetilde G}(S)$ is well defined for any $S \neq \emptyset$. Moreover, for a pairwise disjoint $S, T \subseteq \widetilde V$ such that $S \cup T \neq \emptyset$ we define the dual version of Dirichlet conductance as $\bar \phi_{\widetilde G} (S,T):= 2w(E(S,T))/w(S \cup T)$. 

Next, we set $\ell^2(\widetilde V,w)$ to be the Hilbert space of functions $f: \widetilde V \mapsto \R$ with inner product \(
\langle f, g \rangle_{\ell^2(\widetilde V,w)}:= \sum_{v \in \widetilde V} f(v) g(v) w(v), 
\) and norm $\|f\|_{\ell^2(\widetilde V,w)}:= \sqrt{\langle f, f \rangle_{\ell^2(\widetilde V,w)}}$ for $f,g: \widetilde V \mapsto \R$. For $F: \widetilde V \mapsto \mathscr{H}$, where $\mathscr{H}$ is a Hilbert space we also need the following definition of Rayleigh quotient and its dual version: 
\[
\mathcal{R}_{\widetilde G}(F):= \frac{\displaystyle\sum_{u \sim v} w(u,v)\,\|F(u)-F(v)\bigr\|_{\mathscr{H}}^2}
{ \sum_{v\in \widetilde V} w(v) \|F(v)\|_{\mathscr{H}}^2} \qquad \text{ and } \qquad \overline{\mathcal{R}}_{\widetilde G}(F):= \frac{\displaystyle\sum_{u \sim v} w(u,v)\,\|F(u)+F(v)\bigr\|_{\mathscr{H}}^2}
{ \sum_{v\in \widetilde V} w(v) \|F(v)\|_{\mathscr{H}}^2}.
\]

We are now ready to state the relevant results related to higher order Cheeger and dual Cheeger inequalities that will be used in the proofs of Propositions \ref{prop:cheeger-applied} and \ref{prop:dual-cheeger-applied}, respectively.

\begin{lemma}\label{lem:cheeger}
Let $\widetilde G=(\widetilde V,\widetilde E, w)$ be a finite undirected weighted graph Fix $k \in \N$ and $\tau >0$. Let $\{f_1, f_2,\ldots,f_k\}$ be a collection of orthonormal vectors in $\ell^2(\widetilde V,w)$ such that, for all $i \in \set{k}$, 
\[
\mathcal{R}_{\widetilde G}(f_i) \leq \tau \qquad \text{ and } \qquad f_i(v_\star)=0, \text{ for some } v_\star \in \widetilde V. 
\]
Then, there exist an absolute constant $\bar C < \infty$ and pairwise disjoint subsets $\{S_i\}_{i \in \set{k}}$ of $\widetilde V\setminus \{v_\star\}$ such that  $\phi_{\widetilde G}(S_i) \leq C k^3 \sqrt{\tau}$ for all $i \in \set{k}$. 
\end{lemma}

\begin{lemma}\label{lem:dual-cheeger}
Let $\widetilde G=(\widetilde V,\widetilde E, w)$ be a finite undirected weighted graph Fix $k \in \N$ and $\tau >0$. Let $\{f_1, f_2,\ldots,f_k\}$ be a collection of orthonormal vectors in $\ell^2(\widetilde V,w)$ such that, for all $i \in \set{k}$, 
\[
\overline{\mathcal{R}}_{\widetilde G}(f_i) \leq \tau \qquad \text{ and } \qquad f_i(v_\star)=0, \text{ for some } v_\star \in \widetilde V. 
\]
Then, there exist and absolute constant $C < \infty$ and pairwise disjoint subsets $\{S_i\}_{i \in \set{2k}}$ of $\widetilde V\setminus \{v_\star\}$, with $S_{2i-1} \cup S_{2i} \neq \emptyset$ such that  $\bar\phi_{\widetilde G}(S_{2i-1}, S_{2i}) \geq (1-C k^3 \sqrt{\tau})$ for all $i \in \set{k}$. 
\end{lemma}


Although the main results of \cite{LOGT2014} and in \cite{Liu2015} are about higher order cheeger and dual cheeger inequalities involving the small and large eigenvalues of a normalized laplacian of any finite weighted graph, respectively, the experts in the area will recognize that Lemmas \ref{lem:cheeger} and \ref{lem:dual-cheeger} are essentially established in \cite{LOGT2014} and \cite{Liu2015}, respectively, during the course of the proofs of the main results in those papers. 

We provide an outline of the proofs of Lemmas \ref{lem:cheeger} and \ref{lem:dual-cheeger} for completeness: For $v \in \widetilde V$, set $F(v):=(f_1(v), f_2(v), \ldots, f_k(v))$. Notice that, by assumption, $F(v_\star)=0$. As,
\[
\|F(u)\pm F(v)\|^2 = \sum_{i \in \set{k}} (f_i(u) \pm f_i(v))^2 \qquad \text{ and } \|F(v)\|^2 = \sum_{i \in \set{k}} f_i^2(v),
\]
for any $u, v \in \widetilde V$, it follows that $\mathcal{R}_{\widetilde{G}}(F) \leq \tau$, under the assumptions of Lemma \ref{lem:cheeger}, while $\overline{\mathcal{R}}_{\widetilde{G}}(F) \leq \tau$, under the assumptions of Lemma \ref{lem:dual-cheeger}. 

Repeating the proof of \cite[Theorem 3.7]{LOGT2014} one obtains disjointly supported functions $\psi_1, \psi_2, \ldots, \psi_k: \widetilde V \mapsto \R$ such that $\max_{i \in \set{k}} \mathcal{R}_{\widetilde G}(\psi_i) \leq K k^6 \mathcal{R}_{\widetilde G}(F)$, for some absolute constant $K< \infty$. Although the proof of \cite[Theorem 3.7]{LOGT2014} proceeds with $\{f_i\}_{i \in \set{k}}$ being the set of orthonormal eigenvectors corresponding to the bottom $k$ eigenvalues of the normalized Laplacian of $\widetilde G$, the same proof works for any $k$-set of orthonormal vectors in $\ell^2(\widetilde V, w)$. It further follows from \cite[Lemma 3.5]{LOGT2014} that ${\rm supp}(\psi_i) \subseteq {\rm supp}(F)$ for all $i \in \set{k}$. Since for any  $\psi: \widetilde V \mapsto \mathscr{H}$, there exists some $S \subseteq {\rm supp}(\psi)$ such that $\phi_{\widetilde G}(S) \leq \sqrt{2 \mathcal{R}_{\widetilde G}(\psi)}$ (see \cite[Lemma 2.2]{LOGT2014}) we have the desired conclusions in Lemma \ref{lem:cheeger}.  

To obtain Lemma \ref{lem:dual-cheeger} one proceeds as in the proof of \cite[Theorem 6.1]{Liu2015} and uses \cite[Lemma 6.2]{Liu2015} to find pairwise disjoint subsets $\{S_i\}_{i \in \set{2k}}$ of ${\rm supp}(F)$, with $S_{2i-1} \cup S_{2i} \neq \emptyset$ such that $(1-\bar\phi_{\widetilde G}(S_{2i-1}, S_{2i}))^2 \leq K k^6 \overline{\mathcal{R}}_{\widetilde G}(F)$ for all $i \in \set{k}$, where $K< \infty$ is some absolute constant. This yields Lemma \ref{lem:dual-cheeger}.

We now proceed to prove Propositions \ref{prop:cheeger-applied} and \ref{prop:dual-cheeger-applied}.

\begin{proof}[Proof of Proposition \ref{prop:cheeger-applied}]
Set $\widetilde V:= V \cup \{v_\star\}$, where $v_\star$ is a `ghost' vertex. Set $w(u,v)=1$ for every edge $uv \in E$, and $w(v,v_\star)= d-\deg_{G}(v)$ for every $v \in V$. This defines the weighted graph $\widetilde G$. Observe that with this choice of $w$ we have $w(v)= d$ for $v \in V$. 

Let $\{e_i\}_{i \in \set{k}}$ be a set of orthogonal eigenvectors corresponding to $\{\widetilde \lambda_i(G)\}_{i \in \set{k}}$ of unit $\ell^2$-norm. Since for any real valued vector $x \in \R^{V}$ of unit $\ell^2$-norm,
\begin{equation}\label{eq:G-to-wtG}
\sum_{u \sim v} (x_u \pm x_v)^2 + \sum_{v \in V} (d-\deg_{G}(v)) x_v^2 = d \pm x^{\sf T} A(G) x,
\end{equation}
setting $\widehat f_i$ to be $e_i$ appended with a zero corresponding to $v_\star \in \widetilde V$, and $f_i:= \widehat f_i/\sqrt{d}$ we find that 
\[
\langle f_i, f_j \rangle_{\ell^2(\widetilde V, w)} = \delta_{i,j}, \quad \mathcal{R}_{\widetilde G}(f_i) = \mathcal{R}_{\widetilde G}(\widehat f_i) \leq \tau_0, \qquad \text{ for } i,j \in \set{k}. 
\]
Hence, applying Lemma \ref{lem:cheeger} we obtain pairwise disjoint subsets $\{S_i\}_{i \in \set{k}}$ of $V$ such that 
\[
\frac{e(G[S_i, V\setminus S_i]) + \sum_{v \in V} (d-\deg_{G}(v)) }{d |S_i|}=\phi_{\widetilde G}(S_i) \leq C k^3 \sqrt{\tau_0} \quad \text{ for } i \in \set{k}.
\]
Since $2 e(G[S_i])= \sum_{v \in S_i} \deg_{G}(v)$, the proof is now complete. 
\end{proof}

\begin{proof}[Proof of Proposition \ref{prop:dual-cheeger-applied}]
Let $\widetilde G$ be as in the proof of Proposition \ref{prop:cheeger-applied}. Let $\{\bar e_i\}_{i \in \set{k}}$ be a set of orthogonal eigenvectors corresponding to $\{\widetilde \lambda_{|V(G)|-i+1}(G)\}_{i \in \set{k}}$ of unit $\ell^2$-norm. Set $\widehat f_i$ to be $e_i$ appended with a zero corresponding to $v_\star \in \widetilde V$ and $f_i:= \widehat f_i/\sqrt{d}$. Using \eqref{eq:G-to-wtG} we find that 
\[
\langle f_i, f_j \rangle_{\ell^2(\widetilde V, w)} = \delta_{i,j}, \quad \overline{\mathcal{R}}_{\widetilde G}(f_i) = \mathcal{R}_{\widetilde G}(\widehat f_i) \leq \tau_0, \qquad \text{ for } i,j \in \set{k}. 
\]
Therefore, applying Lemma \ref{lem:dual-cheeger} we obtain pairwise disjoint subsets $\{S_i\}_{i \in \set{2k}}$ of $V$ such that $S_{2i-1} \cup S_{2i} \neq \emptyset$, and 
\begin{equation}\label{eq:bar-phi-to-biclique}
\min\left\{|S_{2i-1}|\cdot |S_{2i}|, d |S_{2i-1}|, d |S_{2i}|\right\} \geq e(G[S_{2i-1}, S_{2i}]) \geq (1-Ck^3 \sqrt{\tau_0}) \cdot \frac{d (|S_{2i-1}|+|S_{2i}|)}{2},
\end{equation}
for all $i \in \set{k}$, where in the leftmost inequality we have also used that the maximum degree in $G$ is at most $d$. To finish the proof, 
upon using \eqref{eq:bar-phi-to-biclique} it follows that 
\begin{equation}\label{eq:S-ratio}
\frac{1-Ck^3 \sqrt{\tau_0}}{1+C k^3 \sqrt{\tau_0}} \leq \frac{|S_{2i-1}|}{|S_{2i}|} \leq \frac{1+Ck^3 \sqrt{\tau_0}}{1- C k^3 \sqrt{\tau_0}}.
\end{equation}
Plugging the above bounds in \eqref{eq:bar-phi-to-biclique} it further yields 
\begin{equation}\label{eq:S-ratio1}
\frac{\min\{|S_{2i-1}|, |S_{2i}|\}}{d} \geq \frac{(1-Ck^3 \sqrt{\tau_0})}{(1+C k^3 \sqrt{\tau_0})}. 
\end{equation}
The upper and lower bounds in \eqref{eq:S-ratio} also imply that 
\begin{equation}\label{eq:S-ratio2}
\frac{|S_{2i-1}|+|S_{2i}|}{2} \geq \frac{\max\left\{|S_{2i-1}|,|S_{2i}|\right\}}{1+C k^3\sqrt{\tau_0}}.
\end{equation}
Upon combining \eqref{eq:bar-phi-to-biclique}, \eqref{eq:S-ratio1}, and \eqref{eq:S-ratio2} the proof completes. 
\end{proof}



\subsection{Structural rigidity in near regular graphs with a few large eigenvalues}
This section contains two results about near regular graphs that have at most two large eigenvalues and a negligible $\ell^2$ mass for the rest of the eigenvalues. The first result shows that there does not exist any graph with the above assumption if the two large eigenvalues are such that their absolute values are separated. The second result shows that if there is only large eigenvalue then \dteal{the graph} must be close to a clique. These are the remaining pieces for completing the proof of Proposition \ref{prop:loc-structure}.

\begin{lemma}\label{lem:neg-theta}
Let $G=(V,E)$ be a finite undirected simple graph with maximum degree at most $d$. 
Fix $\tau >0$ and $\theta \in (-1,1)\setminus \{0\}$. Assume 
\begin{equation}\label{eq:deg-lbd}
\sum_{v \in V} (d-\deg_G(v)) \leq \tau |V|. 
\end{equation}
Then, there exists some $\uptau_\theta >0$, depending only on $\theta$, such that for any $\tau \in (0, \uptau_\theta)$ there does not exist any graph $G$ such that 
\begin{equation}\label{eq:eig-neg-theta}
\lambda_1(G) \geq (1-\tau)d, \quad \lambda_2(G) \in [(\theta-\tau)d, (\theta+\tau)d], \quad \text{ and } \quad \sum_{i =3}^{|V|} \lambda_i^2(G) \leq \tau d^2. 
\end{equation}
\end{lemma}

\dteal{Before starting the proof let us describe the overall strategy. First, using the near regularity of $G$, quantified via \eqref{eq:deg-lbd}, we will show that the leading eigenvector should be close to the constant vector in $\ell^2$-norm. Building on this step, for $\theta <0$, we will argue that most of the edges in $G$ should only be between two subsets of $V$ determined by the signs of the eigenvector corresponding to $\lambda_2(G)$. We will see that $|\theta| <1$ will then result in a contradiction on the product of the size of these two subsets. For $\theta >0$, the situation is different: repeating the same argument one now obtains $G$ must consist of two loosely connected cliques whose vertex sets are the same as above. 
As $\langle w_2, {\bf 1}_V \rangle \approx \langle w_2, w_1\rangle = 0$ this will further imply that the cliques are almost of equal sizes. However, as $\theta < 1$, we one can further obtain that the minimum of these two vertex sets must be $o(d)$ and hence $|V|=o(d)$, which gives rise to a contradiction due to the near regularity and the lower bound on $e(G)$. Below we carry out the details.

}

\begin{proof}[Proof of Lemma \ref{lem:neg-theta}]
If possible, let us assume that there exists a graph with \eqref{eq:deg-lbd} and \eqref{eq:eig-neg-theta}. Let $\{w_i\}_{i=1}^{|V|}$ be an orthonormal set of eigenvectors of corresponding to $\{\lambda_i(G)\}_{i=1}^{|V|}$. Therefore,  its adjacency matrix 
\begin{equation}\label{eq:def-A-B}
\dteal{A(G)= \sum_{i=1}^{2} \lambda_i(G) w_i w_i^{\sf T} +B, \quad \text{ where }\quad B:=\sum_{i=3}^{|V|} \lambda_i(G) w_i w_i^{\sf T}.} 
\end{equation}
We claim that, for any $\tau >0$ sufficiently small,
\begin{equation}\label{eq:w_1-almost-constant-theta}
\sum_{v \in V} \left(w_1(v) - \frac{1}{\sqrt{(1+\theta^2) d}}\right)^2 \leq \tau^{1/3}.
\end{equation}
\dteal{Note that \eqref{eq:w_1-almost-constant-theta} shows that $w_1$ is an almost constant vector.} 

Turning to prove \eqref{eq:w_1-almost-constant-theta}, as a first step we will show that
\begin{equation}\label{eq:w_1-const}
\sum_{i=2}^{|V|} \langle {\bf 1}_V, w_i \rangle^2 = \left\| \dteal{{\bf 1}_V} - \langle {\bf 1}_V, w_1 \rangle w_1 \right\|^2 \leq 11 \sqrt{\tau} d.
\end{equation}
As $\{w_i\}_{i=1}^{|V|}$ are orthonormal the equality in \eqref{eq:w_1-const} is obvious. To prove the inequality in \eqref{eq:w_1-const}, writing $c_i:=\langle \dteal{{\bf 1}_V}, w_i\rangle$, for $i \in \set{|V|}$ and using orthonormality again, we observe that $\sum_{i=1}^{|V|} c_i^2 =|V|$. Using \eqref{eq:eig-neg-theta}, and the fact that $\lambda_1(G) \leq d$ (as the maximum degree is at most $d$) we also obtain
\begin{equation}\label{eq:c-i-bd}
( 1+\theta^2-4\tau)d^2 \leq \lambda_1^2(G)+\lambda_2^2(G) \leq 2 e(G)= \dteal{{\bf 1}^{\sf T}_V} A(G) \dteal{{\bf 1}_V} = \sum_{i=1}^{|V|} \lambda_i c_i^2 \leq d |V| - \sum_{i=2}^{|V|} (\lambda_1(G) - \lambda_i(G)) c_i^2.
\end{equation}
Next we need to find a good upper bound on $|V|$. We use \eqref{eq:deg-lbd}, \eqref{eq:eig-neg-theta}, and the fact $ |\lambda_1(G)|\leq d$ to deduce
\[
(1-\tau) d|V| \leq  d|V| - \sum_{v \in V}(d -\deg_G(v)) \leq 2e(G) =\sum_{i=1}^{|V|} \lambda_i^2(G) \leq (1+\theta^2+ 3\tau)d^2. 
\]
This, in turn implies that 
\begin{equation}\label{eq:S-barS-sum}
|V| \leq (1+\theta^2 +7 \tau)d, \quad \text{ for any } \tau >0 \text{ sufficiently small}. 
\end{equation}

%


Since, $\min_{i \ge 2} (\lambda_1(G) -\lambda_i(G)) \geq c_\theta d$, for $\tau$ sufficiently small, where $c_\theta:= (1- \max\{\theta, 0\})/2$ (use \eqref{eq:eig-neg-theta} again), from \eqref{eq:c-i-bd} and \eqref{eq:S-barS-sum} it further follows that 
\[
\sum_{i=2}^{|V|} c_i^2 \leq c_\theta^{-1} \left(|V| - \frac1d\left(\lambda_1^2(G)+\lambda_2^2(G)\right)\right) \leq 11 \sqrt{\tau} d,
\]
for $\tau$ sufficiently small depending on $c_\theta$. 

As the maximum degree is at most $d$ and $2e(G) \geq (1+\theta^2-4\tau)d^2$ (see \eqref{eq:c-i-bd}) we further deduce that $|V| \geq (1+\theta^2-4\tau)d$. Using this lower bound on $|V|$, as $w_1$ is the Perron eigenvector and therefore all its entries are nonnegative, we obtain from \eqref{eq:w_1-const} 
\begin{equation}\label{eq:V-w_1}
\left||V|^{1/2}- \langle {\bf 1}_V, w_1 \rangle \right|\leq  \frac{\sum_{i=2}^{|V|} \langle {\bf 1}_V, w_i \rangle^2}{|V|^{1/2}} \leq 11 \sqrt{\tau d}.
\end{equation}
Hence, \dteal{by triangle inequality},
\begin{multline*}
\left\|w_1 - \frac{1}{\sqrt{(1+\theta^2)d}} {\bf 1}_V\right\| \\
\leq \frac{1}{\sqrt{(1+\theta^2) d}} \left\| {\bf 1}_V - \langle {\bf 1}_V, w_1 \rangle w_1 \right\| + \frac{1}{\sqrt{(1+\theta^2)d}}\left||V|^{1/2} - \langle {\bf 1}_V, w_1 \rangle\right|
+ \left| \sqrt{\frac{|V|}{(1+\theta^2) d}}-1 \right| 
\leq 21 \tau^{1/4},
\end{multline*}
which yields \eqref{eq:w_1-almost-constant-theta}. 

Next, \dteal{recalling  \eqref{eq:def-A-B}, we note} 
\[
-B_{v,v'}=  d w_1(v)w_1(v') +\theta d w_2(v)w_2(v') -(d-\lambda_1(G)) w_1(v)w_1(v') -(\theta d-\lambda_2(G)) w_2(v)w_2(v') - {\bf 1}_{(v,v') \in E}.
\] 
\dteal{This yields} 
\begin{multline*}
\sum_{v,v' \in V} \left( {\bf 1}_{(v,v') \in E} -\frac{1}{1+\theta^2}- \theta dw_2(v) w_2(v')\right)^2 \\
\lesssim \|B\|_{{\rm HS}}^2 + (d-\lambda_1(G))^2 + (\theta d-\lambda_2(G))^2 + \sum_{v,v' \in V} \left(d w_1(v)w_1(v') -\frac{1}{1+\theta^2}\right)^2 \\
\leq 3\tau d^2 +  \left| 1+ \sqrt{\frac{|V|}{(1+\theta^2) d}}\right|^2 \cdot \left\|w_1 -  \frac{1}{\sqrt{(1+\theta^2) d}}{\bf 1}\right\|^2\cdot d^2 \leq 3\tau^{1/3} d^2,
\end{multline*}
where \dteal{$\|\cdot\|_{{\rm HS}}$ denotes the Hilbert-Schmidt norm, and} in the penultimate step we used the fact that  
\[
\|\widetilde w \widetilde w^{\sf T} - \widehat w \widehat w^{\sf T}\|_{{\rm HS}}^2 \leq (\|\widetilde w\|+\|\widehat w \|)^2 \|\widetilde w - \widehat w\|^2,
\]
for any two vectors $\widetilde w$ and $\widehat w$, 
and in the last step we used the tight upper and lower bound on $|V|$ and \eqref{eq:w_1-almost-constant-theta}. 

Thus, we have derived that
\begin{equation}\label{eq:second-evec-neg-theta}
\sum_{(v,v') \in E(G)} \left(\frac{\theta^2}{1+\theta^2} - \theta d w_2(v) w_2(v')\right)^2  + \sum_{(v,v') \notin E(G)} \left(\frac{1}{1+\theta^2} + \theta d w_2(v) w_2(v')\right)^2\leq \tau^{1/4} d^2.
\end{equation}

Now we need to consider the cases $\theta >0$ and $\theta <0$ separately. First, let us consider the case $\theta <0$. 

Set $S:=\{v \in V: w_2(v) \geq 0\}$. As $\theta <0$, and as $v, v' \in S$ or $v,v' \in \bar S\dteal{=: V\setminus S}$ implies that $w_2(v)w_2(v') \geq 0$, we find from \eqref{eq:second-evec-neg-theta} that
\begin{equation}\label{eq:within-edges}
e(G[S]) + e(G[\bar S]) \leq 4 \theta^{-4} \tau^{1/4} d^2.
\end{equation}
\dteal{This indeed shows that $G$ is close to a bipartite graph with almost all its edges are between $S$ and $\bar S$.} 

\dteal{By \eqref{eq:within-edges}, using the lower bound on $e(G)$, we next find that} 
\[
|S| \cdot |\bar S| \geq e(G[S,\bar S]) \geq e(G) - e(G[S])-e(G[\bar S]) \geq \left( \frac{1+\theta^2}{2} -\tau^{1/5}\right)d^2,
\]
for $\tau >0$ sufficiently small. Hence, using the lower bound $|V| \geq (1+\theta^2-4\tau)d$, and as $|\theta| < 1$,
\begin{equation}\label{eq:S-impossible}
\frac{|S|}{|V|} \cdot \frac{|\bar S|}{|V|} \geq \frac{1}{2(1+\theta^2)} \cdot \left( 1 - \tau^{1/6}\right) > \frac14,
\end{equation}
for $\tau >0$ sufficiently small depending on $\theta$. However, since $|S|+|\bar S| =|V|$, \eqref{eq:S-impossible} is an impossibility, implying that no such graph exists. 

It remains to treat the case $\theta \in (0,1)$. To \dteal{implement the strategy mentioned above}, we need the following upper bound on the $\ell^\infty$-norm of $w_2$:
\begin{equation}\label{eq:w_2-ell-infty-norm}
\|w_2\|_\infty \leq 2/|\theta|\sqrt{d}.
\end{equation}
The proof of \eqref{eq:w_2-ell-infty-norm} is straightforward. Indeed, fix any $v \in V$. From the eigenvector-eigenvalue equation, and applying Cauchy-Schwarz inquality we have
\[
\left|\lambda_2(G) w_2(v)\right| = \left|\sum_{u: u \sim v}  w_2(u)\right| \leq \sqrt{\deg_G(v)} \sqrt{\sum_{u: \dteal{(u, v)\in E}} w_2^2(u)} \leq \sqrt{d},
\]
where in the last step we have used that the degree is bounded above by $d$ and $w_2$ is of unit $\ell^2$-norm. Now the bound \eqref{eq:w_2-ell-infty-norm} follows immediately from the lower bound on $|\lambda_2(G)|$ obtained from \eqref{eq:eig-neg-theta}. 

\dteal{Continuing with the proof}, by \eqref{eq:second-evec-neg-theta}, as $\theta >0$,
\begin{equation}\label{eq:theta-pos}
e(G[S, \bar S]) +\binom{|S|}{2} -e(G[S]) + \binom{|\bar S|}{2} -e(G[\bar S]) \leq \tau^{1/5}d^2,
\end{equation}
for $\tau$ sufficiently small depending on $\theta$. \dteal{Notice that this indeed shows that  $G[S]$ and $G[\bar S]$ are loosely connected, and these are essentially cliques. 

\dteal{Next} the bound \eqref{eq:theta-pos} together} with \eqref{eq:w_2-ell-infty-norm} imply that 
\begin{multline}\label{eq:second-evec-structure}
\sum_{v,v' \in S} \left(\frac{\theta^2}{1+\theta^2} - \theta d w_2(v) w_2(v')\right)^2 + \sum_{v,v' \in \bar S} \left(\frac{\theta^2}{1+\theta^2} - \theta d w_2(v) w_2(v')\right)^2 \\
\leq 2 \sum_{(v,v') \in E(G)} \left(\frac{\theta^2}{1+\theta^2} - \theta d w_2(v) w_2(v')\right)^2 
+  2\left(|S|^2 +|\bar S|^2 -2 e(G[S]) - 2 e(G[\bar S])\right)\cdot \left( \theta^4 + \theta^2 d^2 \|w_2\|_\infty^4 \right) \\
\leq \tau^{1/6} d^2,
\end{multline}
for $\tau$ sufficiently small, where we have also used \eqref{eq:second-evec-neg-theta}. 
Since, for any two vectors $\widetilde w$ and $\widehat w$ we have the lower bound
\[
2\|\widetilde w \widetilde w^{\sf T} - \widehat w \widehat w^{\sf T}\|_{{\rm HS}}^2 \geq \|\widetilde w -\widehat w\|^2 \cdot \|\widetilde w +\widehat w\|^2 \geq \min\left\{\|\widetilde w - \widehat w\|^4, \|\widetilde w + \widehat w\|^4\right\},
\]
and as $w_2$ is nonnegative on $S$ and negative on $\bar S$, we further deduce from \eqref{eq:second-evec-structure} that 
\begin{equation}\label{eq:second-evec-structure-1}
\left\| w_2{\bf 1}_S - \sqrt{\frac{\theta}{(1+\theta^2)d}} {\bf 1}_S\right\|^2 +  \left\|w_2{\bf 1}_{\bar S} + \sqrt{\frac{\theta}{(1+\theta^2)d}} {\bf 1}_{\bar S} \right\|^2 \leq  \tau^{1/13}. 
\end{equation}
This, in particular, implies that 
\begin{equation}\label{eq:S-star-bd}
|S\setminus S_\star| + |\bar S \setminus \bar S_\star| \leq \tau^{1/40}d, 
\end{equation}
where 
\[
S_\star:= \left\{v \in S: \left|w_2(v) - \sqrt{\frac{\theta}{(1+\theta^2) d}}\right| \leq  \frac{\tau^{1/40}}{\sqrt{d}}\right\} \, \text{ and } \, \bar S_\star:= \left\{v \in \bar S: \left|w_2(v) + \sqrt{\frac{\theta}{(1+\theta^2) d}}\right| \leq  \frac{\tau^{1/40}}{\sqrt{d}}\right\}.
\]
Upon using the definitions of $S_\star$ and $\bar S_\star$, triangle inequality, and \eqref{eq:w_2-ell-infty-norm} it moreover follows that 
\begin{multline*}
\left| \theta d w_2(v) w_2(v') + \frac{\theta^2}{1+\theta^2} \right| \\
\leq \left|\sqrt{\theta d} w_2(v) - {\frac{\theta}{\sqrt{1+\theta^2}}}\right| \cdot \sqrt{\theta d} \|w_2\|_\infty + \left|\sqrt{\theta d} w_2(v') + \frac{\theta}{\sqrt{1+\theta^2}}\right| \cdot {\frac{\theta}{\sqrt{1+\theta^2}}} \leq \tau^{1/41},
\end{multline*}
for any $v \in S_\star$ and $v' \in \bar S_\star$. Hence, \dteal{by triangle inequality,} 
\begin{multline*}
|S_\star| \cdot |\bar S_\star| \cdot \left(\frac{1-\theta^2}{1+\theta^2} - \tau^{1/41}\right)^2 \leq  \sum_{(v,v') \notin E(G)} \left(\frac{1}{1+\theta^2} + \theta d w_2(v) w_2(v')\right)^2 + e(G[S,\bar S]) \cdot \left( 1 + \theta^2 d^2 \|w_2\|_\infty^4 \right).
\end{multline*}
Therefore, this last observation together with the upper bound on $|V|$, and \eqref{eq:second-evec-neg-theta}, \eqref{eq:w_2-ell-infty-norm}, \eqref{eq:theta-pos}, and \eqref{eq:S-star-bd} imply that $|S| \cdot |\bar S| \leq \tau^{1/41} d^2$, which along with the tight upper and lower bounds on $|V|$ yield
\begin{equation}\label{eq:min-S}
\min\left\{|S|, |\bar S|\right\} \leq \tau^{1/83} d,
\end{equation}
for $\tau$ sufficiently small, depending on $\theta$. 

On the other hand, as $w_2 \perp w_1$, using \eqref{eq:w_1-almost-constant-theta} and \eqref{eq:S-star-bd} we find
\begin{multline*}
\left| \sum_{v \in S_\star \cup  \bar S_\star} w_2(v) \right| \leq \left| w_2^{\sf T} (\dteal{{\bf 1}_V} -\sqrt{(1+\theta^2)d}w_1)\right| + \tau^{1/40}\|w_2\|_\infty d  \leq \|\dteal{{\bf 1}_V} -\sqrt{(1+\theta^2)d} w_1\|  +\tau^{1/41}\sqrt{d}\\
\leq \tau^{1/42} \sqrt{d}.
\end{multline*}
Hence, recalling the definitions of $S_\star$ and $\bar S_\star$ together with the upper bound on $|V|$ we derive from above that
\begin{multline*}
\sqrt{\frac{\theta}{(1+\theta^2) d}} \cdot \left| |S_\star| - |\bar S_\star|\right| \leq \sum_{v \in S_\star} \left|w_2(v) - \sqrt{\frac{\theta}{(1+\theta^2) d}}\right| + \sum_{v \in \bar S_\star} \left|w_2(v) + \sqrt{\frac{\theta}{(1+\theta^2) d}}\right| +\tau^{1/42} \sqrt{d}\\
\leq \tau^{1/43} \sqrt{d}.
\end{multline*}
The last observation combined with \eqref{eq:S-star-bd} and \eqref{eq:min-S} further yield
\[
\max\{|S|, |\bar S|\} \leq \tau^{1/84} d.
\]
Since $|V| =|S|+ |\bar S|$, this is a contradiction to the lower bound $|V| \geq (1+\theta^2-4\tau)d$ implying that such a graph cannot exist. 
The proof of the lemma is now complete. 
\end{proof}

\begin{lemma}\label{lem:only-theta}
Let $G=(V,E)$ be a finite undirected simple graph with maximum degree at most $d$. Fix $\theta \in (0,1)$ and $\tau >0$. Suppose the following bounds hold on the eigenvalues:
\begin{equation}\label{eq:only-theta}
\lambda_1(G) \in [(\theta -\tau)d, (\theta+\tau) d] \qquad \text{ and } \qquad \sum_{i \geq 2} \lambda_i^2(G) \leq \tau d^2.
\end{equation}
Then, there exists some constants $ c_\theta, \bar c_\theta \in (0,1)$, depending only on $\theta$, such that for any $\tau \in [c_\theta/d, \bar c_\theta]$, the graph $G$ contains a subgraph $G_\star$ with minimum degree at least $(\theta -\tau^{1/3})d$. 
\end{lemma}

\begin{proof}
\dteal{Observe that}
\[
\sum_i \lambda_i^4(G) = {\rm Hom}(C_4, G)= N(C_4, G)+ 4 N(K_{1,2},G) + 8 e(G).
\]
As the maximum degree in $G$ is at most $d$, upon using \eqref{eq:only-theta}, it follows from above that 
\[
N(C_4, G) \geq (1-\bar\tau)^4 \theta^4 d^4 - 8(1+4 \bar \tau)\theta^2 d^3 \geq (1-5 \bar \tau) \theta^4 d^4 \geq (1-10 \bar \tau) (2 e(G))^2,
\]
whenever $\bar\tau \in [c_\theta/d, \bar c_\theta]$, where $\bar \tau:= \tau/\theta$. The desired conclusion now follows from \cite[Lemma A.2]{BK25}.
\end{proof}

\subsection{\dteal{Proof of Proposition \ref{prop:loc-structure}}}
Below we combine Propositions \ref{prop:cheeger-applied} and \ref{prop:dual-cheeger-applied}, and Lemmas \ref{lem:loc-structure-spectral}, \ref{lem:neg-theta}, and \ref{lem:only-theta} to prove Proposition \ref{prop:loc-structure}. 

\begin{proof}[Proof of Proposition \ref{prop:loc-structure}]
First let us consider the case $\ell \in 2 \N+1$. By Lemma \ref{lem:loc-structure-spectral} we see that \eqref{eq:lambda-assumption-cheeger} holds with $\tau_0=\tau_\star^{c_\star}$ and $k =\lfloor \delta \rfloor$ and hence Proposition \ref{prop:cheeger-applied} applies. Let $\{S_i\}_{i\in\set{k+1}}$ be the partition of $V$ obtained upon applying  Proposition \ref{prop:cheeger-applied}, where $S_{k+1}:=V\setminus (S_1 \cup S_2 \cup \cdots \cup S_k)$. Let $\widehat G$ be the edge disjoint union of the graphs $\{G[S_i]\}_{i \in \set{k+1}}$. Set 
\[
\widehat S_i:=\{v \in S_i: \deg_{G[S_i]}(v) \geq (1-\tau_\star^{c_\star/5})d\}, \qquad i \in \set{k}.
\]
Since $\sum_{v \in S_i} (d- \deg_{G[S_i]}(v) = d|S_i| -2 e(G[S_i]))$, we deduce that
%
%
\[
|S_i \setminus \widehat S_i| \leq \tau_\star^{c_\star/5} |S_i| \leq 4 \tau_\star^{c_\star/5} e(G[S_i])/d \leq \tau_\star^{c_\star/6}d, 
\]
where in the first step we used the leftmost assertion in \eqref{eq:ho-cheeger} and Markov inequality, the leftmost assertion in \eqref{eq:ho-cheeger} again in the penultimate step, and the upper bound on $e(G)$ in the last step. This, in turn, implies that 
\begin{equation}\label{eq:hat-S-i}
|\widehat S_i| \geq |S_i| - |S_i\setminus \widehat S_i| \geq 2 e(G[S_i])/|S_i| - \tau_\star^{c_\star/6} d \geq (1 - \tau_\star^{c_\star/7}) d
\end{equation}
and
\begin{equation}\label{eq:hat-S-i-deg}
\deg_{G[\widehat S_i]}(v) \geq \deg_{G[S_i]}(v) - |S_i\setminus \widehat S_i| \geq (1 - \tau_\star^{c_\star/7}) d,
\end{equation}
for all $i \in \set{k}$. Further, as the maximum degree is $d$,
\[
\sum_{i \in \set{k}} |S_i| \leq 2 d^{-1} \sum_{i \in \set{k}} e(G[S_i]) \leq 2e(G)/d,
\]
\dteal{For the moment assume $\delta=k \in \N$. Then the above upper bound} together with the upper bound on $e(G)$ and the lower bound on \dteal{$|\widehat S_i|$} from \eqref{eq:hat-S-i} imply that
\begin{equation}\label{eq:S-i-ubd}
|\widehat S_i| \leq |S_i| \leq (k +\tau_\star)d - (k-1) (1-\tau_\star^{c_\star/5})d \leq (1+\tau^{c_\star/6})d. 
\end{equation}
\dteal{So, for} $\delta \in \N$ \eqref{eq:hat-S-i}, \eqref{eq:hat-S-i-deg}, and \eqref{eq:S-i-ubd} yield the desired conclusion. 

The case of $\delta \notin \N$ requires some additional arguments. To this end, by the rightmost assertion in \eqref{eq:ho-cheeger},  
\begin{equation}
e(G) -e (\widehat G) \leq \sum_{i=1}^k e(G[S_i, V\setminus S_i]) \leq \tau_\star^{c_\star/3} d \sum_{i=1}^k |S_i| \leq 4 \tau_\star^{c_\star/3}\sum_{i=1}^k e(G[S_i]) \leq \tau_\star^{c_\star/4} d^2,
\end{equation}
for $\tau_\star$ sufficiently small, where yet again we used the leftmost assertion in \eqref{eq:ho-cheeger} in the penultimate step, and the upper bound on $e(G)$ in the last step. Therefore, by Hoffman-Wielandt inequality (see \cite[Lemma 2.1.19]{AGZ2010}), 
\[
\max_{j \in \set{\lceil \delta \rceil}} \left|\widetilde \lambda_j(G) - \widetilde \lambda_j(\widehat G)\right| \leq \sqrt{\sum_i \left(\widetilde \lambda_j(G) - \widetilde \lambda_j(\widehat G)\right)^2} \leq \sqrt{2(e(G) -e (\widehat G))} \leq \tau_\star^{c_\star/9} d. 
\]
\dteal{Thus}, by Lemma \ref{lem:loc-structure-spectral}(i)
\begin{equation}\label{eq:spectra-widehatG-1}
d \geq \widetilde \lambda_1(\widehat G) \geq \widetilde \lambda_2(\widehat G) \geq \cdots \geq \widetilde \lambda_{\lfloor \delta \rfloor}(\widehat G) \geq (1-\tau_\star^{c_\star/10})d, \qquad \widetilde \lambda_{\lceil \delta \rceil}(\widehat G)/d \in [\{\delta\}^{1/\ell} - \tau_\star^{c_\star/10},\{\delta\}^{1/\ell} + \tau_\star^{c_\star/10}],
\end{equation}
and
\begin{equation}\label{eq:spectra-widehatG-2}
\sum_{j \geq \lceil \delta \rceil +1} \lambda_j^2(\widehat G) \leq 2e(G) - \sum_{j=1}^{\lceil \delta \rceil} \widetilde \lambda_j^2(\widehat G) \leq \tau_\star^{c_\star/11} d^2.
\end{equation}
On the other hand, applying \eqref{eq:ho-cheeger} once again,
\begin{equation}\label{eq:Gi-near-reg1}
\widetilde \lambda_1(G[S_i]) \geq \frac{{\bf 1}_{S_i}^{\sf T} A(G[S_i]){\bf 1}_{S_i}}{|S_i|} \geq (1-\tau_\star^{c_\star/3})d, \qquad \text{ for all } i \in \set{k},
\end{equation}
and
\begin{equation}\label{eq:Gi-near-reg2}
\sum_{v \in S_i} (d -\deg_{G[S_i]}(v)) \leq \tau^{c_\star/3} |S_i| d, \qquad \text{ for all } i \in \set{k}.
\end{equation}
Recall $k=\lfloor \delta\rfloor$. Since $\widehat G$ is the disjoint unions of $\{G[S_i]\}_{i \in \set{k+1}}$ the spectrum of $\widehat G$ is simply the union of the spectra of $\{G[S_i]\}_{i \in \set{k+1}}$, it follows from \eqref{eq:spectra-widehatG-1}, \eqref{eq:spectra-widehatG-2}, and \eqref{eq:Gi-near-reg1} (use pigeonhole principle) that we have the following two potential scenarios: (a) $\widetilde \lambda_1(G[S_{k+1}]) = \widetilde \lambda_{\lceil \delta \rceil}(\widehat G)$, or (b) $\widetilde \lambda_2(G[S_{i}]) = \widetilde \lambda_{\lceil \delta \rceil}(\widehat G)$ for some $i \in \set{k}$. However, by \eqref{eq:spectra-widehatG-1}-\eqref{eq:Gi-near-reg2}, and Lemma \ref{lem:neg-theta} case (b) is an impossibility. Hence, case (a) must hold. So, Lemma \ref{lem:only-theta} applies and it yields the desired property on $G[S_{k+1}]$. 

By \eqref{eq:spectra-widehatG-1}, we have $2e(G[S_{k+1}]) \geq (\{\delta\}^{2/\ell} 2 \tau_\star^{c_\star/10}) d^2$. Therefore, arguing similarly as in steps leading to \eqref{eq:S-i-ubd} we deduce that it continues to hold for $\delta \notin \N$ case, as well as \eqref{eq:hat-S-i} and \eqref{eq:hat-S-i-deg}. 
Hence, $\{G[\widehat S_i]\}_{i \in \set{k}}$ have the desired properties as well. This completes the proof of part (i).  


The proof of part (ii) follows a similar line of arguments. Below are the details. We begin by applying Lemma \ref{lem:loc-structure-spectral}. It 
implies that there exists some nonnegative integer $k \leq \lfloor \delta \rfloor$ such that 
\begin{equation}\label{eq:widehat-lambda-ineq0}
d \geq \widetilde \lambda_1(G) \geq \widetilde\lambda_2(G) \geq \cdots \geq \widetilde\lambda_{k}(G) \geq (1 -  \tau_\star^{c_\star})d. 
\end{equation}
and
\begin{equation}\label{eq:widehat-lambda-ineq}
-d \leq \widehat \lambda_1(G) \leq \widehat\lambda_2(G) \leq \cdots \leq \widehat\lambda_{k'}(G) \ \dteal{\leq} -(1 -  \tau_\star^{c_\star})d. 
\end{equation}
where $k':=\lfloor \delta \rfloor -k$, and for brevity we write $\widehat \lambda_i(G):= \widetilde \lambda_{|V(G)|-i+1}(G)$ for all $i$. 

We claim that $k \geq \lfloor \delta\rfloor/2$. If not, as the entries of $A(G)$ are nonnegative, using Lemma \ref{lem:loc-structure-spectral}(ii) we find
\begin{multline*}
0 \leq \sum_i  \widetilde \lambda_i^3(G) \leq \sum_{i \geq \lceil \delta \rceil +1} |\lambda_i(G)|^3 + \sum_{i \in \set{k}}\widetilde \lambda_i^3(G) +\sum_{i\in \set{k'}} \widehat \lambda_{i}^3(G) + {\bf 1}_{\{\delta \notin \N\}} |\lambda_{\lceil \delta \rceil}(G)|^3\\
\leq \tau_\star^{c_\star}d^3 + k d^3 -k'(1-3\tau_\star^{c_\star})d^3 +(\{\delta\}^{3/\ell}+ \tau^{c_\star/2})d^3 \leq (k-k' + \{\delta\}^{3/\ell} + \tau_\star^{c_\star/3})d^3 < 0,
\end{multline*}
as $k' - k \geq 1$ for such values of $k$, yielding a contradiction. \dteal{Thus $k \geq \lfloor \delta\rfloor/2$. This will be used below.}

Since \eqref{eq:widehat-lambda-ineq} holds we can apply Proposition \ref{prop:dual-cheeger-applied}. Let $\{T_i\}_{i \in \set{2k'}}$ be the pairwise disjoint subsets obtained upon applying Proposition \ref{prop:dual-cheeger-applied}.  Set $T:= V \setminus (T_1 \cup T_2 \cup \cdots \cup T_{2k'})$. Observe that, by \eqref{eq:ho-dual-cheeger},
\begin{multline}\label{eq:bipartite-excees-edge}
\dteal{e(G[T_{2i-1} \cup T_{2i}, V \setminus (T_{2i-1} \cup T_{2i})]) + 2\left(e(G[T_{2i-1}])+ e(G[T_{2i}])\right)} \\
= \sum_{v \in T_{2i-1} \cup T_{2i}} (\deg_{G}(v) - \deg_{G[T_{2i-1}, T_{2i}]}(v)) \leq \sum_{v \in T_{2i-1} \cup T_{2i}} (d - \deg_{G[T_{2i-1}, T_{2i}]}(v)) \\
\leq d(|T_{2i-1}|+|T_{2i}|) - 2 e(G[T_{2i-1}, T_{2i}])  
\leq \tau_\star^{c_\star/5} d (|T_{2i-1}|+|T_{2i}|) \leq \tau_\star^{c_\star/6}d^2, 
\end{multline}
for any $i \in \set{k'}$, where in the last step we \dteal{also 
used} the upper bound on $e(G)$ to bound $|T_{2i-1}|+T_{2i}|$ from above. 
Next define
\[
\widehat T_{2i-1}:= \left\{v \in T_{2i-1}: \deg_{G[T_{2i-1}, T_{2i}]}(v) \geq (1-\tau_\star^{c_\star/12})d\right\}\]
and
\[
\widehat T_{2i}:= \left\{v \in T_{2i}: \deg_{G[T_{2i-1}, T_{2i}]}(v) \geq (1-\tau_\star^{c_\star/12})d\right\},  
\]
for $i \in \set{k'}$. By Markov's inequality, we deduce from \eqref{eq:bipartite-excees-edge} that
\begin{equation}\label{eq:widehat-T}
|T_{2i-1}\setminus \widehat T_{2i-1}| + |T_{2i}\setminus \widehat T_{2i}| \leq \tau_\star^{c_\star/12} d. 
\end{equation}
Therefore, for any $v \in \widehat T_{2i-1} \cup \widehat T_{2i}$,
\begin{equation}\label{eq:widehat-Gi-deg-lbd}
\deg_{G[\widehat T_{2i-1}, \widehat T_{2i}]}(v) \geq \deg_{G[T_{2i-1}, T_{2i}]}(v) -  |T_{2i-1}\setminus \widehat T_{2i-1}| - |T_{2i}\setminus \widehat T_{2i}| \geq (1-\tau_\star^{c_\star/13})d.
\end{equation}
Further, by \eqref{eq:ho-dual-cheeger} and \eqref{eq:widehat-T}, 
\begin{equation}\label{eq:V-widehat-Gi-lbd}
\min\{|\widehat T_{2i-1}|, |\widehat T_{2i}| \} \geq (1-\tau_\star^{c_\star/13}) d, \qquad \text{ for any } i \in \set{k'}.
\end{equation}

To proceed further, we let $G_0$ to be the edge disjoint union of the graphs $\{\widehat G_i\}_{i \in \set{k'}}$, where for brevity we write $\widehat G_i:=G[\widehat T_{2i-1}, \widehat T_{2i}]$, for $i \in \set{k'}$, and $G'$ to be the edge disjoint union $G_0$  and $G[T]$.  By \eqref{eq:bipartite-excees-edge} and \eqref{eq:widehat-T}, we have $e(G)-e(G') \leq \tau_\star^{c_\star/13} d^2$. Hence, by Hoffman-Wielandt inequality, \eqref{eq:widehat-lambda-ineq0}-\eqref{eq:widehat-lambda-ineq}, and Lemma \ref{lem:loc-structure-spectral}(i),
\begin{equation}\label{eq:spectra-widehatG-1-even1}
d \geq  \widetilde \lambda_1(G') \geq  \widetilde \lambda_2(G') \geq \cdots \geq \widetilde \lambda_{k}(G') \geq (1-\tau_\star^{c_\star/27})d, 
\end{equation}
\begin{equation}\label{eq:spectra-widehatG-1-even2}
-d \leq  \widehat \lambda_1(G') \leq  \widehat \lambda_2(G') \leq \cdots \leq \widehat \lambda_{k'}(G') \leq -(1-\tau_\star^{c_\star/27})d,
\end{equation}
\begin{equation}\label{eq:spectra-widehatG-1-even3}
| \lambda_{\lceil \delta \rceil}(G')|/d \in [\{\delta\}^{1/\ell} - \tau_\star^{c_\star/27},\{\delta\}^{1/\ell} + \tau_\star^{c_\star/27}],
\end{equation}
and
\begin{equation}\label{eq:spectra-widehatG-2-even}
\sum_{j \geq \lceil \delta \rceil +1} \lambda_j^2(G') \leq 2e(G) - \sum_{j=1}^{\lceil \delta \rceil}  \lambda_j^2(G') \leq \tau_\star^{c_\star/28} d^2.
\end{equation}
Above, the inequality \eqref{eq:spectra-widehatG-1-even3} is to be understood to hold only for $\delta \notin \N$.

We claim that, for all $i \in \set{k'}$, 
\begin{equation}\label{eq:G-T-i-spectra}
d \geq \widetilde \lambda_1(\widehat G_i) \geq \widehat \lambda_1(\widehat G_i) \geq (1-\tau_\star^{c_\star/27}) d \quad \text{ and } \quad \sum_{j \geq 3} \lambda_j^2(\widehat G_i) \leq \tau_\star^{c_\star/28} d^2.
\end{equation}
Turning to prove \eqref{eq:G-T-i-spectra}, as $\widehat G_i$ is a bipartite graph with edges only between $\widehat T_{2i-1}$ and $\widehat T_{2i}$, applying Courant-Fisher-Weyl min-max principle and \eqref{eq:ho-dual-cheeger} we find that 
\[
\widehat \lambda_{1}(\widehat G_i) \leq \frac{{\bm u}_i^{\sf T} A(\widehat G_i) {\bm u}_i}{|\widehat T_{2i-1}| + |\widehat T_{2i}|} = - \frac{2 e(\widehat G_i)}{|\widehat T_{2i-1}| + |\widehat T_{2i}|} \leq -(1-\tau_\star^{c_\star/5})d, \qquad \text{ where } {\bm u}_i:=\begin{pmatrix} {\bf 1}_{\widehat T_{2i-1}} \\ -{\bf 1}_{\widehat T_{2i}} \end{pmatrix},
\]
for all $i \in \set{k'}$. Since the spectrum of $G'$ is the disjoint union of the spectra of $\{\widehat G_i\}_{i \in \set{k'}}$ and of $G[T]$, the last observation together with \eqref{eq:spectra-widehatG-1-even2} imply that there exists one and only one eigenvalue of $\widehat G_i$ that is less than $-(1-\tau_\star^{c_\star/27})d$, for any $i \in \set{k'}$. Hence, by Perron-Frobenius theorem $\widetilde \lambda_1(\widehat G_i) \geq (1-\tau_\star^{c_\star/27})d$ as well. 

To complete the proof of the claim \eqref{eq:G-T-i-spectra}, we will argue that there does not exist any other eigenvalue of $\widehat G_i$ whose absolute value exceeds $\tau_\star^{c_\star/57} d$. Nonexistence of such eigenvalues together with \eqref{eq:spectra-widehatG-2-even} and the observations made in the previous paragraphs indeed implies \eqref{eq:G-T-i-spectra}. 


Now, if possible, let us assume that there exists $m \geq 3$ eigenvalues that exceed $\tau_\star^{c_\star/57} d$ in absolute values.  \dteal{The subgraph} $\widehat G_i$ being a bipartite graph it does not contain any triangle. This yields 
\begin{equation}\label{eq:small-triangle}
\left| \sum_{j \in \set{m}} \lambda_j^3(\widehat G_i)\right| \leq   \left|\sum_{j \geq m+1} \lambda_j^3(\widehat G_i)\right| + \left|\sum_j  \lambda_j^3(\widehat G_i)\right| \leq \tau_\star^{c_\star/28}d^3 + N(K_{3}, \widehat G_i) \leq \tau_\star^{c_\star/28}d^3,
\end{equation}
where in the penultimate step we also used \eqref{eq:spectra-widehatG-2-even}. We have already argued above that there is only one negative eigenvalue of $\widehat G_i$ that is less that $-(1-\tau_\star^{c_\star/27})d$. Therefore, barring $\widehat \lambda_1(\widehat G_i)$ and possibly at most one more (\dteal{the one corresponding to \eqref{eq:spectra-widehatG-1-even3}}), any other eigenvalue that is greater than $\tau_\star^{c_\star/57}d$ in absolute value must be positive (use \eqref{eq:spectra-widehatG-1-even1}-\eqref{eq:spectra-widehatG-2-even}). 
Further, such eigenvalues must be bounded by $(1-\tau_\star^{c_\star/27})d$ from below. So, for $\delta \notin \N$ if there exists an eigenvalue whose absolute value lies in $[(\{\delta\}^{1/\ell} -\tau_\star^{c_\star/27})d, (\{\delta\}^{1/\ell} +\tau_\star^{c_\star/27})d]$, then depending on its sign, we have
\begin{equation}\label{eq:small-triangle1}
\left((m-3) \pm \{\delta \}^{3/\ell} -\tau_\star^{c_\star/28}\right)d^3 \leq \sum_{j \in \set{m}} \lambda_j^3(\widehat G_i) \leq \left( (m-3) \pm \{\delta \}^{3/\ell} + \tau_\star^{c_\star/28}\right)d^3,
\end{equation}
and otherwise
\begin{equation}\label{eq:small-triangle2}
\left((m-2) -\tau_\star^{c_\star/28}\right)d^3 \leq \sum_{j \in \set{m}} \lambda_j^3(\widehat G_i) \leq \left( (m-2)+ \tau_\star^{c_\star/28}\right)d^3.
\end{equation}
Both \eqref{eq:small-triangle1} and \eqref{eq:small-triangle2} contradict \eqref{eq:small-triangle}. This yields the claim \eqref{eq:G-T-i-spectra}. 






Armed with \eqref{eq:G-T-i-spectra} we are nearly at the end of the proof. Indeed, by \eqref{eq:G-T-i-spectra} we obtain that $2e(\widehat G_i) \leq (2+\tau_\star^{c_\star/28})d^2$, which together with  \eqref{eq:ho-dual-cheeger} imply $\max\{|T_{2i-1}|, |T_{2i}|\} \leq (1+\tau_\star^{c_\star/28})d$. Therefore, by \eqref{eq:widehat-Gi-deg-lbd}, 
\begin{equation}\label{eq:widehat-Gi-deg-lbd-1}
\deg_{G[\widehat T_{2i-1}, \widehat T_{2i}]}(v)/ \max\{|\widehat T_{2i-1}|, |\widehat T_{2i}|\} \geq 1 -\tau_\star^{c_\star/29}, \qquad \text{ for all } v \in \widehat T_{2i-1} \cup \widehat T_{2i} \text{ and all } i \in \set{k'}. 
\end{equation}
Notice that \eqref{eq:V-widehat-Gi-lbd} and \eqref{eq:widehat-Gi-deg-lbd-1} provide the required bipartite structures. 

Set $\delta':= \delta -2k'$. As $k \geq \lfloor \delta \rfloor /2$ and $k':=\lfloor \delta \rfloor -k$, we have $\delta'= \delta +2k -2 \lfloor \delta \rfloor \geq \delta - \lfloor \delta \rfloor \geq 0$. 
Recalling that $G'$ is the edge disjoint union of $G[T]$ and $\{\widehat G_i\}_{i \in \set{k'}}$, 
%
we deduce from \eqref{eq:spectra-widehatG-1-even1}-\eqref{eq:spectra-widehatG-2-even} and \eqref{eq:G-T-i-spectra} that  
\begin{equation}\label{eq:spectra-after-bipartite}
d \geq  \widetilde \lambda_1(G[T]) \geq  \widetilde \lambda_2(G[T]) \geq \cdots \geq \widetilde \lambda_{\lfloor \delta '\rfloor}(G[T]) \geq (1-\tau_\star^{c_\star/27})d, 
\end{equation}
\begin{equation}\label{eq:spectra-after-bipartite1}
\sum_{j \geq \lceil \delta' \rceil +1} \lambda_j^2(G[T])  \leq \tau_\star^{c_\star/28} d^2,
\end{equation}
and if $\delta' \notin \N$, 
\begin{equation}\label{eq:spectra-after-bipartite2}
| \lambda_{\lceil \delta' \rceil}(G[T])|/d \in [\{\delta\}^{1/\ell} - \tau_\star^{c_\star/27},\{\delta\}^{1/\ell} + \tau_\star^{c_\star/27}].
\end{equation}
Notice that the bounds in \eqref{eq:spectra-after-bipartite}-\eqref{eq:spectra-after-bipartite2} are almost identical with the bounds on the eigenvalues that we worked with in part (i), except for an indeterminacy in the sign of $\lambda_{\lceil \delta' \rceil}(G[T])$. Since Lemma \ref{lem:neg-theta} allows for both positive and negative signs of the second largest eigenvalue in absolute value, 
one can argue similarly as in part (i) to extract the desired almost clique like structures in $G[T]$. We omit further details. 
The proof of the proposition is now finally complete.
\end{proof}

\bibliographystyle{plain}
\bibliography{mybib.bib}{}


\end{document}